\documentclass[reqno]{amsart}
\usepackage{amsmath,amsrefs,amssymb,color}

\numberwithin{equation}{section}

\usepackage[colorinlistoftodos]{todonotes}

\DeclareMathOperator{\loc}{\text{loc}}
\DeclareMathOperator{\sym}{\text{sym}}
\DeclareMathOperator{\sign}{\text{sign}}
\DeclareMathOperator{\bigO}{O}
\DeclareMathOperator{\smallo}{o}
\newcommand{\R}{\mathbb{R}}
\newcommand{\Sph}{\mathbb{S}}
\newcommand{\N}{\mathbb{N}}
\newcommand{\<}{\left<}
\renewcommand{\>}{\right>}
\renewcommand{\[}{\left[}
\renewcommand{\]}{\right]}
\renewcommand{\(}{\left(}
\renewcommand{\)}{\right)}

\newcommand{\ve}{\varepsilon}
\newcommand{\gammae}{\overline{\gamma}_\ve}

\newtheorem{theorem}{Theorem}[section]
\newtheorem{definition}[theorem]{Definition}
\newtheorem{proposition}[theorem]{Proposition}
\newtheorem{claim}[theorem]{Claim}
\newtheorem{lemma}[theorem]{Lemma}
\newtheorem{rem}[theorem]{Remark}

\begin{document}

\title[Sign-changing blow-up for the Moser--Trudinger equation]{Sign-changing blow-up for the\\Moser--Trudinger equation}

\author{Luca Martinazzi}

\address{Luca Martinazzi, Universit\`a degli Studi di Padova, Dipartimento di Matematica Tullio Levi-Civita, Via Trieste, 63, 35121 Padova}
\email{luca.martinazzi@math.unipd.it}

\author{Pierre-Damien Thizy}

\address{Pierre-Damien Thizy, Institut Camille Jordan, Universit\'e Claude Bernard Lyon 1, 21 avenue Claude Bernard, 69622 Villeurbanne Cedex, France}
\email{pierre-damien.thizy@univ-lyon1.fr}

\author{J\'er\^ome V\'etois}

\address{J\'er\^ome V\'etois, McGill University, Department of Mathematics and Statistics, 805 Sherbrooke Street West, Montreal, Quebec H3A 0B9, Canada}
\email{jerome.vetois@mcgill.ca}

\thanks{\noindent The first author was supported by the Swiss National Foundation Grants PP00P2-144669 and PP00P2-170588/1. The second author was supported by the French ANR Grant BLADE-JC. The third author was supported by the Discovery Grant RGPIN-2016-04195 from the Natural
Sciences and Engineering Research Council of Canada.}

\date{April 10, 2021}

\begin{abstract}
Given a sufficiently symmetric domain $\Omega\Subset\mathbb{R}^2$, for any $k\in \mathbb{N}\setminus \{0\}$ and $\beta>4\pi k$ we construct blowing-up solutions $\left(u_\varepsilon\right)\subset H^1_0\left(\Omega\right)$ to the Moser--Trudinger equation such that as $\varepsilon\downarrow 0$, we have $\left\|\nabla u_\varepsilon\right\|_{L^2}^2\to \beta$, $u_\varepsilon \rightharpoonup u_0$ in $H^1_0$ where $u_0$ is a sign-changing solution of the Moser--Trudinger equation and $u_\varepsilon$ develops $k$ positive spherical bubbles, all concentrating at $0\in \Omega$. These $3$ features (lack of quantization, non-zero weak limit and bubble clustering) stand in sharp contrast to the positive case ($u_\varepsilon>0$) studied by the second author and Druet~\cite{DruThi}.
\end{abstract}

\maketitle

\section{Introduction and main result}\label{Sec1}

Given a smooth, bounded domain $\Omega\subset\R^2$ and a smooth, positive function $h$ on $\overline\Omega$, we consider the Moser--Trudinger functional $I_h:H^1_0\left(\Omega\)\to\R$ defined as 
$$I_h\(u\):=\int_\Omega h \exp\(u^2\) dx\qquad\forall u\in H^1_0\(\Omega\).$$
For any $\beta>0$, let $E_{h,\beta}$ be the set of all the critical points $u\in H^1_0\(\Omega\)$ of $I_h$ under the constraint $\left\|\nabla u\right\|_{L^2}^2=\beta$. Note that $u\in E_{h,\beta}$ if and only if $u$ is a solution of the problem
\begin{equation}\label{Sec1Eq1}\tag{$\mathcal{E}_{h,\beta}$}
\left\{\begin{aligned}&\Delta u=\lambda h f\(u\)&&\text{in }\Omega\\&u=0&&\text{on }\partial\Omega,\end{aligned}\right.
\end{equation}
where we use the notation $\Delta:=-\partial^2_{x_1}-\partial^2_{x_2}$ ,
\begin{equation}\label{Sec1Eq2}
f\(u\):=u \exp\(u^2\)\quad\text{and}\quad\lambda:= \frac{2\beta}{D I_h\(u\).u}=\frac{\beta}{\int_\Omega h u^2 \exp\(u^2\)dx}\,.
\end{equation}

\noindent 
We first introduce the following definition in the spirit of~\cite{ZurichBook}*{Chapter 5} (see also Remark~\ref{RemStable}):

\begin{definition}\label{Def1}
We say that $\beta>0$ is a \emph{stable energy level of $I_h$} if, for all $(h_\varepsilon)$, $(\beta_\varepsilon)$ and $(\lambda_\varepsilon)$ such that $h_\varepsilon\to h$ in $C^2(\bar{\Omega})$ and $\beta_\varepsilon\to \beta$ with $\lambda_\varepsilon=O(1)$, any family $(u_\varepsilon)$ such that $u_\varepsilon$ solves $(\mathcal{E}_{h_\varepsilon,\beta_\varepsilon})$ with $\lambda=\lambda_\varepsilon$ for all $\varepsilon$ converges in $C^2(\bar{\Omega})$ to some $u$ solving \eqref{Sec1Eq1} as $\varepsilon\to 0$, up to a subsequence. We say that $\beta>0$ is a \emph{positively stable energy level of $I_h$} if the same holds true with $u_\varepsilon\ge 0$.
\end{definition}

As a consequence of the Moser--Trudinger inequality~\cites{Mos,Tru}, every $\beta\in\(0,4\pi\)$ is a stable energy level of $I_h$. Druet--Thizy~\cite{DruThi} obtained that every $\beta\in\(0,\infty\)\backslash4\pi\mathbb{N}^\star$ ($\mathbb{N}^\star:=\mathbb{N}\setminus\{0\}$) is a positively stable energy level of $I_h$. In contrast to this result, we obtain in this paper that every $\beta\ge4\pi$ is an unstable energy level provided $\Omega$ and $h$ are such that $0\in\Omega$ and the following symmetric condition holds true for some even number $l\in2\N^\star$:
\begin{enumerate}
\item[(A)] $\Omega$ is symmetric and $h$ is even with respect to the lines 
$$\ell_j:=\left\{\(t\cos\(\frac{j\pi}{2l}\),t\sin\(\frac{j\pi}{2l}\)\):\,t\in\R\right\},\quad 0\le j\le 2l-1.$$
\end{enumerate}
Under this assumption, we obtain the following:

\begin{theorem}\label{Th1}
Let $\Omega\subset \mathbb{R}^2$ be a smooth, bounded domain, $l\in2\N^*$, $\alpha\in\(0,1\)$ and $h\in C^{l-2,\alpha}\(\Omega\)\cap C^2\(\overline\Omega\)$ be a positive function such that $0\in\Omega$ and (A) holds true. Then every $\beta\ge4\pi$ is an unstable energy level of $I_h$.
\end{theorem}

In order to prove Theorem~\ref{Th1} we will construct a sign-changing weak limit $w_0$ with arbitrary energy $\beta_0\in  (0,8\pi l)$ and use a Lyapunov-Schmidt procedure to glue to $w_0$ an arbitrary number $k\in \mathbb{N}^*$ of bubbles, all concentrating at the origin. This is in sharp contrast to the positive case studied by Druet--Thizy~\cite{DruThi}, in which blow-up can happen only at energy levels $\beta\in 4\pi \mathbb{N}^\star$, the weak limit vanishes and the bubbles blow up at distinct points. See also~\cites{delPinoMusRuf1,delPinoMusRuf2,DruMalMarThi} for the constructive counterpart of~\cite{DruThi}. 

\smallskip
To be more concrete, given $h\in C^{l-2,\alpha}\(\Omega\)\cap C^2\(\overline\Omega\)$ and $\beta_0> 0$, using the symmetry of $\Omega$ and $h$, we will construct $w_0\in E_{h,\beta_0}$ such that
\begin{equation}\label{Sec1Eq3}
w_0\(x_1,0\)\sim a_0 x_1^l,\quad \text{as }x_1\to 0,\quad \text{for some }a_0>0.
\end{equation} 
Up to a perturbation and a diagonal argument, we can assume that $w_0$ is non degenerate, and construct families $h_\ve\to h$ in $C^2\(\overline\Omega\)$, $\beta_\ve\to \beta_0$ and $w_\varepsilon$, $0\le\ve\le \ve_0$, smooth with respect to $\ve$ such that $w_\ve\in E_{h_\ve,\beta_\ve}$ and $0>w_\ve\(0\)\uparrow 0$ as $\ve\to 0$.
The behaviour \eqref{Sec1Eq3} of the weak limit $w_0$ near the origin will be crucial to glue bubbles and the value of $w_\ve\(0\)\uparrow 0$ will determine the parameter $\gammae\to \infty$ (see \eqref{Sec32Eq1}), which is the approximate height of the bubbles.

\smallskip
In fact, if $\overline B_\gamma$ is the radial solution to $\Delta \overline{B}_\gamma=f\big(\overline{B}_\gamma\big)$ with $\overline{B}_\gamma\(0\)=\gamma$, we will attach to the function $w_\ve$ a fixed number $k$ of perturbations of $\overline B_{\gamma_\ve}$ along the $x_1$ axis, at points $\(\tau_{\ve,1},0\),\dots,\(\tau_{\ve,k},0\)$. The centers $(\tau_{\ve,i},0)$ of the bubbles will satisfy for some $\delta\in (0,1)$,
\begin{equation}\label{Sec1Eq3b}
-\frac{k d_\ve}{\delta}<\tau_{\ve,1}<\dots<\tau_{\ve,k}<\frac{ k d_\ve}{\delta },\quad |\tau_{\ve,i}-\tau_{\ve,j}|>\delta d_\ve, \quad d_\ve:=\gammae^{-1/l}\to 0,
\end{equation}
and, up to scaling, $\(\tau_{\ve,1}/d_\ve,\dots,\tau_{\ve,k}/d_\ve\)$ will converge to a zero of $N=(N^1,\dots, N^k)$, defined in a suitable convex subset of $\R^k$ as
\begin{equation}\label{Sec1Eq4}
N^i\(y_1,\dots, y_k\):= a_0 ly_i^{l-1} -\sum_{j\ne i}\frac{2}{y_i-y_j}.
\end{equation}
Note that, contrary to the case studied in~\cites{delPinoMusRuf1,delPinoMusRuf2,DruThi}, the function $h$ (more specifically, its gradient) plays no role in \eqref{Sec1Eq4}, hence at main order it does not influence the location of the bubbles, which instead depends on $a_0>0$ and $l$ as in \eqref{Sec1Eq3} and on $k$. 

\smallskip
A diagonal argument allows to treat the case $\beta_0=0$. Thus we finally obtain:

\begin{theorem}\label{Th2} 
Given $\Omega$, $l$, $\alpha$ and $h$ as in Theorem \ref{Th1} and $\beta\ge 4\pi$, $k\in \mathbb{N}^*$, $\beta_0\ge 0$ such that $\beta=\beta_0+4\pi k$, there exist $w_0\in E_{h,\beta_0}$ and $\ve_0>0$ such that for every $\ve\in (0,\ve_0)$, we can find $h_\ve\to h$ in $C^2(\overline{\Omega})$, $\beta'_\ve\to \beta_0$, as $\ve\to 0$, $w_\ve\in E_{h_\ve,\beta_\ve'}$ as in \eqref{Sec1Eq3}, numbers $\beta_\ve\to \beta$, $\gammae\to \infty$, $\gamma_\ve$, $\tau_\ve$, $\theta_\ve\in \R^k$, with $\gamma_{\ve,i}\sim \gammae$, $\theta_{\ve,i}\to0$ as $\ve\to 0$, $\tau_{\ve,i}$ as in \eqref{Sec1Eq3b}, and a function $u_\ve \in E_{h_\ve,\beta_\ve}$ of the form
\begin{equation}\label{Sec1Eq5}
u_\ve = w_\ve +\sum_{i=1}^k(1+\theta_{\ve,i}) B_{\ve,\gamma_{\ve,i},\tau_{\ve,i}}+\Psi_{\ve,\gamma_\ve,\tau_\ve}+\Phi_{\ve,\gamma_\ve,\theta_\ve,\tau_\ve}\,,
\end{equation}
where the approximate bubble $B_{\ve,\gamma_{\ve,i},\tau_{\ve,i}}\in H^1_0(\Omega)$ is as in Section \ref{Sec31} and the remainder $H^1_0(\Omega)$-terms $\Psi_{\ve,\gamma_\ve,\tau_\ve}$ and $\Phi_{\ve,\gamma_\ve,\theta_\ve,\tau_\ve}$ are given by Propositions~\ref{Pr4} and~\ref{Pr8}. In particular,
$$\|\nabla \Psi_{\ve,\gamma_\ve,\tau_\ve}\|_{L^2}=\smallo(1), \quad \|\nabla \Phi_{\ve,\gamma_\ve,\theta_\ve,\tau_\ve}\|_{L^2}=\smallo(1),\quad \text{as }\ve\to 0.$$
\end{theorem}

\smallskip
In contrast with several other works constructing blowing-up solutions to the Moser--Trudinger equation, starting with del Pino--Musso--Ruf~\cites{delPinoMusRuf1,delPinoMusRuf2}, our Lyapunov-Schmidt reduction will be performed in $H^1_0\(\Omega\)$, avoiding the use of weighted $C^0$-norms. While this is more in the spirit of the seminal work of Rey~\cite{Rey}, the elegance of working with the Hilbert space $H^1_0\(\Omega\)$ requires a very precise ansatz (see Section~\ref{Sec3}), and a very sharp analysis of the radial bubble $\overline{B}_\gamma$, as obtained in~\cites{DruThi, MalMar, ManMar} and further extended in our Section~\ref{App}. In fact, we will construct the ansatz in two steps. First we construct the approximate solution
$$U_{\ve,\gamma,\tau}=w_\ve +\sum_{i=1}^k B_{\ve,\gamma_{i},\tau_{i}}+\Psi_{\ve,\gamma,\tau},$$
for every $\tau$ as in \eqref{Sec1Eq3b} and $\gamma$ in a fairly broad range (see \eqref{defGammaepsilon0} and Proposition \ref{Pr4}). Then we shall strongly restrict the range of $\gamma$ (Proposition \ref{Pr5} and \eqref{defGammaepsilon}) and add the terms $\theta_i B_{\ve,\gamma_i,\tau_i}$ to the ansatz.
This will be crucial in the energy estimates of Section~\ref{Sec4}. In order to estimate the error terms near the bubbles, we shall use the spherical profile of the bubble to treat the blow-up regions as approximate spheres and apply Poincar\'e--Sobolev-type estimates, as given in Section~\ref{SecPS}.
Finally we will perform the Lyapunov-Schmidt reduction to find the correct value $(\gamma_\ve,\theta_\ve,\tau_\ve)$ and the correction term $\Phi_{\ve,\gamma_\ve,\theta_\ve,\tau_\ve}$, to finally obtain $u_\ve$ as in \eqref{Sec1Eq5} (see also Remark \ref{RemCoupling}). 

\smallskip
Recently, Problem \eqref{Sec1Eq1} has received attention also when the nonlinearity $f$ is suitably perturbed. Mancini and the second author~\cite{ManThi} constructed radial (both positive and nodal) solutions $u_\gamma$ to \eqref{Sec1Eq1} on the unit disk, blowing up at $0$ and having non-zero weak limit as $\gamma\to \infty$, in the case $h\equiv 1$, $f_\gamma(u)=\lambda_\gamma u +\beta_\gamma u\exp(u^2)$ or $f_\gamma(u)=\beta_\gamma u\exp(u^2-a u)$ for suitable $\lambda_\gamma$, $\beta_\gamma$ and $a>0$. Grossi--Mancini--Naimen--Pistoia~\cite{GMNP} constructed nodal solutions $u_p$ to \eqref{Sec1Eq1} with $h\equiv 1$ and $f_p(u)=u\exp\(u^2+|u|^p\)$, having \emph{one} blow-up point as $p\downarrow 1$. Naimen~\cite{Nai} further gave a very detailed blow-up analysis of the blow-up of \emph{radial} nodal solutions to \eqref{Sec1Eq1} when $h\equiv 1$ and $f(u)=u\exp\(u^2+\alpha|u|^\beta\)$, $\alpha>0$. To our knowledge, our work is the first one in which non-zero weak limits appear in the unperturbed case $f(u)=u\exp(u^2)$, and an arbitrary number of bubbles concentrates at the same point. Indeed, these two phenomena cannot occur in the unperturbed case without (radial) symmetry breaking.

\section{Preliminary steps}\label{Sec2}

This section is devoted to the construction of a smooth family of critical points satisfying some regularity, symmetry and asymptotic conditions which we will then use in the next sections to construct our blowing-up solutions.

\begin{definition}\label{Def2}
For every $l\in \N^*$, $p\in\N$ and $\alpha\in\(0,1\)$, we let $C_{l,\sym}^{p,\alpha}\(\Omega\)$ be the vector space of all functions in $C^{p,\alpha}\(\Omega\)$  that are even with respect to the line $\ell_{2j}$ for all $j\in\left\{0,\dotsc,l-1\right\}$, where $\ell_{2j}$ is as in (A).
\end{definition}

\begin{definition}\label{Def3}
Let $\beta>0$, $h$ be a continuous, positive function on $\overline\Omega$ and $w\in E_{h,\beta}$. Then we say that $w$ is non-degenerate if there does not exist any solution $v\ne0$ to the problem
\begin{equation}\label{Sec2Eq}
\left\{\begin{aligned}&\Delta v=\lambda h f'\(w\)v&&\text{in }\Omega\\&v=0&&\text{on }\partial\Omega,\end{aligned}\right.
\end{equation}
where $\lambda$ and $f$ are as in \eqref{Sec1Eq2}. We let $E_{h,\beta}^{nd}$ be the set of all non-degenerate elements of $E_{h,\beta}$.
\end{definition}

\noindent The main result of this section is the following:

\begin{proposition}\label{Pr0}
Let $\Omega$ be a smooth, bounded domain, $l\in 2\N^\star$, $\beta_0\in\(0,8l\pi\)$, $\alpha\in\(0,1\)$ and $h\in C^{l-2,\alpha}\(\Omega\)\cap C^{0,\alpha}\(\overline\Omega\)$ be a positive function such that $0\in\Omega$ and {\normalfont(A)} holds true.  Then we have the following:
\begin{enumerate}
\item[(i)] There exist $w_0\in E_{h,\beta_0}\cap C_{l,\sym}^{l,\alpha}\(\Omega\)\cap C^2\(\overline\Omega\)$ and $a_0>0$ such that $w_0\(x_1,0\)\sim a_0x_1^l$ as $x_1\to0$.
\item[(ii)] There exists $\kappa_0>0$ such that for every $\kappa\in\(-\kappa_0,\kappa_0\)\backslash\left\{0\right\}$, $w_\kappa:=\(1+\kappa\)w_0\in E^{nd}_{h_\kappa,\beta_\kappa}\cap C_{l,\sym}^{l,\alpha}\(\Omega\)\cap C^2\(\overline\Omega\)$, where $\beta_\kappa:=(1+\kappa)^2\beta_0$ and $h_\kappa:=h\exp\(-\kappa\(\kappa+2\)w_0^2\)$. 
\item[(iii)] For every $\kappa\in\(-\kappa_0,\kappa_0\)\backslash\left\{0\right\}$, there exist $\hat{h}_\kappa\in C_{l,\sym}^{l-2,\alpha}\(\Omega\)\cap C^2\(\overline\Omega\)$ and $\varepsilon_0\(\kappa\)\in\(0,1\)$ such that for every $\varepsilon\in\(0,\varepsilon_0\(\kappa\)\)$, there exist $\beta_{\kappa,\varepsilon}>0$ and $w_{\kappa,\varepsilon}\in E_{h_{\kappa,\varepsilon},\beta_{\kappa,\varepsilon}}\cap C_{l,\sym}^{l,\alpha}\(\Omega\)\cap C^2\(\overline\Omega\)$, where $h_{\kappa,\varepsilon}:=h_\kappa+\varepsilon\hat{h}_\kappa$, such that the families $\(\beta_{\kappa,\varepsilon}\)_{0\le\varepsilon\le\varepsilon_0\(\kappa\)}$ and $\(w_{\kappa,\varepsilon}\)_{0\le\varepsilon\le\varepsilon_0\(\kappa\)}$, where $\beta_{\kappa,0}:=\beta_\kappa$ and $w_{\kappa,0}:=w_\kappa$, are smooth in $\varepsilon$ and moreover $\partial_\varepsilon\[w_{\kappa,\varepsilon}\(0\)\]_{\varepsilon=0}<0$ and $w_{\kappa,\varepsilon}\(0\)<0$ for all $\varepsilon\in\(0,\varepsilon_0\(\kappa\)\)$.
\end{enumerate}
\end{proposition}

\proof[Proof of Proposition~\ref{Pr0} (i)]
Define 
$$\Omega_1:=\left\{\(x_1,x_2\)\in\Omega:\,\left|x_2\right|<x_1\tan\(\frac{\pi}{2l}\)\right\}.$$
Since $\Omega$ satisfies (A), we obtain that $\Omega_1$ is symmetric with respect to the line $\ell_0$. In particular, we can define the vector space $\mathcal{H}$  of all functions in $H^1_0\(\Omega_1\)$ that are even in $x_2$. Note that (A) also gives that $\left.h\right|_{\Omega_1}$ is even in $x_2$. By applying standard variational arguments (see for instance Proposition~6 of Mancini--Martinazzi~\cite{ManMar} in case $h\equiv1$ and $\mathcal{H}=H^1_0\(\Omega_1\)$), we then obtain that for every $\beta_0\in\(0,8l\pi\)$, there exists a critical point $w_0$ of the functional $\left.I_h\right|_{\mathcal{H}}$ under the constraint $\left\|\nabla w_0\right\|_{L^2\(\Omega_1\)}^2=\beta_0/2l$ such that $w_0>0$ in $\Omega_1$. By using (A), we can then extend $w_0$ to the whole domain $\Omega$ as an odd function with respect to the line $\ell_{2j+1}$ for all $j\in\left\{0,\dotsc,l-1\right\}$. We claim that $w_0\in E_{h,\beta_0}$. To see this, for every test function $v\in H_0^1\(\Omega\)$, we define
$$v_{\sym}:=\sum_{j=0}^{l-1}v\circ S_{2j+1}\circ S_1-\sum_{j=0}^{l-1}v\circ S_{2j+1},$$
where $S_{2j+1}:\Omega\to\Omega$ is the symmetry operator with respect to the line $\ell_{2j+1}$. By remarking that $v_{\sym}\in H_0^1\(\Omega_1\)$ and using $v_{\sym}$ as a test function for the Euler--Lagrange equation of $u_0$, we obtain
\begin{equation}\label{Pr1Eq1}
\int_{\Omega_1}\left<\nabla u_0,\nabla v_{\sym}\right>dx=\frac{\beta_0}{l\int_{\Omega_1}hw_0^2\exp\(w_0^2\)dx}\int_{\Omega_1}hf\(w_0\)v_{\sym}dx.
\end{equation}
By changes of variable and using the symmetry of $w_0$ and $h$, we obtain
\begin{align}
\int_{\Omega_1}\left<\nabla w_0,\nabla v_{\sym}\right>dx&=\int_\Omega\left<\nabla w_0,\nabla v\right>dx,\label{Pr1Eq2}\allowdisplaybreaks\\
\int_{\Omega_1}hw_0^2\exp\(w_0^2\)dx&=\frac{1}{2l}\int_{\Omega}hw_0^2\exp\(w_0^2\)dx=\frac{1}{2l}DI_{h}\(w_0\).w_0,\label{Pr1Eq3}\allowdisplaybreaks\\
\int_{\Omega_1}hf\(w_0\)v_{\sym}dx&=\int_\Omega hf\(w_0\)vdx.\label{Pr1Eq4}
\end{align}
By putting together \eqref{Pr1Eq1}--\eqref{Pr1Eq4}, we obtain
$$\int_\Omega\left<\nabla w_0,\nabla v\right>dx=\frac{2\beta_0}{DI_{h}\(w_0\).w_0}\int_\Omega hf\(w_0\)vdx$$
and so $w_0\in E_{h,\beta_0}$. Since $h\in C^{l-2,\alpha}\(\Omega\)\cap C^{0,\alpha}\(\overline\Omega\)$ and $\partial\Omega$ is smooth, by using the Moser--Trudinger inequality together with standard elliptic regularity theory, we then obtain that $w_0\in C^{l,\alpha}\(\Omega\)\cap C^{2,\alpha}\(\overline\Omega\)$. Since $\left.w_0\right|_{\Omega_1}$ is even in $x_2$ and $w_0$ is odd with respect to the line $\ell_{2j+1}$ for all $j\in\left\{0,\dotsc,l-1\right\}$, we then obtain that $w_0$ is even with respect to $\ell_{2j}$ for all $j\in\left\{0,\dotsc,l-1\right\}$, i.e. $w_0\in C_{l,\sym}^{l,\alpha}\(\Omega\)$. Furthermore, since $w_0\in C^{l,\alpha}\(\Omega\)$ and $w_0=0$ on $\ell_{2j+1}$ for all $j\in\left\{0,\dotsc,l-1\right\}$, we obtain that $D^jw_0\(0\)=0$ for all $j\in\left\{0,\dotsc,l-1\right\}$ and so
\begin{equation}\label{Pr1Eq5}
w_0\(x_1,0\)=a_0x_1^l+\bigO\(x_1^{l+\alpha}\)
\end{equation}
as $x_1\to0$ for some $a_0\in\R$. It remains to prove that $a_0>0$. Since $0\in\Omega$ and $\Omega$ is open, there exists $r_0>0$ such that $B\(0,r_0\)\subset\Omega$. For every $\varepsilon>0$, we define
$$S_{l,\varepsilon}\(r_0\):=\left\{\(r\cos\theta,r\sin\theta\):\,0<r<r_0\text{ and }\left|\theta\right|<\pi/\(2\(l+\varepsilon\)\)\right\}$$
and let $v_{l,\varepsilon}:S_{l,\varepsilon}\(r_0\)\to\R$ be the function defined as
$$v_{l,\varepsilon}\(r\cos\theta,r\sin\theta\):=r^{l+\varepsilon}\cos\(\(l+\varepsilon\)\theta\)$$
for all $\(r\cos\theta,r\sin\theta\)\in S_{l,\varepsilon}\(r_0\)$. It is easy to check that $v_{l,\varepsilon}$ is harmonic in $S_{l,\varepsilon}\(r_0\)$, continuous on $\overline{S_{l,\varepsilon}\(r_0\)}$ and $v_{l,\varepsilon}=0$ on $B\(0,r_0\)\cap\partial S_{l,\varepsilon}\(r_0\)$. On the other hand, since $S_{l,\varepsilon}\(r_0\)\subset\Omega_1$, we have that $w_0$ is continuous on $\overline{S_{l,\varepsilon}\(r_0\)}$ and positive on $\overline{S_{l,\varepsilon}\(r_0\)}\backslash\left\{0\right\}$. Furthermore, since $h,w_0>0$ in $S_{l,\varepsilon}\(r_0\)$, it follows from the Euler--Lagrange equation of $w_0$ that $\Delta w_0>0$ in $S_{l,\varepsilon}\(r_0\)$. It follows that there exists $\delta_{l,\varepsilon}>0$ such that $w_0\ge\delta_{l,\varepsilon}v_{l,\varepsilon}$ on $\partial S_{l,\varepsilon}\(r_0\)\cap\partial B\(0,r_0\)$. By comparison, we then obtain that $w_0\ge\delta_{l,\varepsilon}v_{l,\varepsilon}$ in $S_{l,\varepsilon}\(r_0\)$. Since $v_{l,\varepsilon}\(r,0\)=r^{l+\varepsilon}$, by taking $\varepsilon<\alpha$, we then obtain that the number $a_0$ in \eqref{Pr1Eq5} is positive. This ends the proof of (i) in Proposition~\ref{Pr0}.
\endproof

\proof[Proof of Proposition~\ref{Pr0} (ii)]
It is easy to check that $w_\kappa\in E_{h_\kappa,\beta_\kappa}\cap C_{l,\sym}^{2,\alpha}\(\overline\Omega\)$ for all $\kappa\in\(-1,1\)$. It remains to prove that $w_\kappa\in E^{nd}_{h_\kappa,\beta_\kappa}$ for $\kappa\in\(-\kappa_0,\kappa_0\)\backslash\left\{0\right\}$ with $\kappa_0$ small enough. Assume by contradiction that this is not the case, i.e. there exists a sequence of real numbers $\(\kappa_j\)_{j\in\N}$ such that $w_{\kappa_j}$ is degenerate and $\kappa_j\to0$. Let $v_j$ be a nonzero solution of the linearized equation
$$\left\{\begin{aligned}&\Delta v_j=\lambda_{\kappa_j} h_{\kappa_j} f'\(w_{\kappa_j}\)v_j&&\text{in }\Omega\\
&v_j=0&&\text{on }\partial\Omega,\end{aligned}\right.$$
with
$$\lambda_{\kappa_j}:=\frac{2\beta_{\kappa_j} }{ DI_{h_{\kappa_j}}\big(w_{\kappa_j}\big).w_{\kappa_j}}= \frac{2\beta_0}{ DI_{h}\(w_0\).w_0}=\lambda.$$
By renormalizing and passing to a subsequences, we may assume without loss of generality that $\left\|\nabla v_j\right\|_{L^2}=1$ and $\(v_j\)_{j\in\N}$ converges weakly to some function $v_0$ in $H^1_0\(\Omega\)$. By using the compactness of the embedding $H^1_0\(\Omega\)\hookrightarrow L^2\(\Omega\)$ and remarking that $\beta_\kappa\to\beta_0$ and $h_\kappa,w_\kappa\to h,w_0$ in $C^0\(\overline\Omega\)$, we obtain that $\(v_j\)_{j\in\N}$ converges strongly to $v_0$ in $H^1_0\(\Omega\)$ and so $\left\|\nabla v_0\right\|_{L^2}=1$. Furthermore, we obtain that $v_0$ is a solution of \eqref{Sec2Eq} with $\kappa=0$. By using the definitions of $h_\kappa$, $\beta_\kappa$ and $w_\kappa$, in particular noticing that $h_\kappa \exp\(w_\kappa^2\)= h\exp\(w_0^2\)$, and recalling the equation satisfied by $v_j$ and $v_0$, it follows that
$$\int_\Omega h\(1+2w_{\kappa_j}^2\)\exp\(w_0^2\)v_jv_0dx=\frac{1}{2}DI_{h}\(w_0\).w_0\int_\Omega\left<\nabla v_j,\nabla v_0\right>dx=\int_\Omega hf'\(w_0\)v_jv_0dx$$
and so 
\begin{equation}\label{Pr2Eq}
\int_\Omega hw_0^2\exp\(w_0^2\)v_jv_0dx=\int_\Omega h\frac{w_{\kappa_j}^2-w_0^2}{\kappa_j\(\kappa_j+2\)}\exp\(w_0^2\)v_jv_0dx=0.
\end{equation}
By passing to the limit into \eqref{Pr2Eq}, we obtain
$$\int_\Omega hw_0^2\exp\(w_0^2\)v_0^2dx=0,$$
which gives $w_0v_0=0$ in $\Omega$. Since $\left\|\nabla w_0\right\|_{L^2}^2=\beta_0\ne0$, by unique continuation (see Aronszajn~\cite{Aro} and Cordes~\cite{Cor}), we obtain that $w_0\ne0$ in a dense subset $D$ of $\Omega$ and so $v_0=0$ on $D$. By continuity of $v_0$, it follows that $v_0=0$ in $\Omega$. This is in contradiction with $\left\|\nabla v_0\right\|_{L^2}=1$. This ends the proof of (ii) in Proposition~\ref{Pr0}.
\endproof

The result of (iii) in Proposition~\ref{Pr0} will follow from the following:

\begin{proposition}\label{Pr3}
Let $l\in\N^\star$, $p\ge2$, $\alpha\in\(0,1\)$ and $\Omega$ be a smooth, bounded domain such that $0\in\Omega$ and $\Omega$ is symmetric with respect to the line $\ell_{2j}$ for all $j\in\left\{0,\dotsc,l-1\right\}$. Let $\overline{\beta}_0>0$, $\overline{h}\in C_{l,\sym}^{p-2,\alpha}\(\Omega\)\cap C^2\(\overline\Omega\)$ be positive in $\overline\Omega$ and $\overline{w}_0\in E^{nd}_{\overline{h},\overline{\beta}_0}\cap C_{l,\sym}^{p,\alpha}\(\Omega\)\cap C^2\(\overline\Omega\)$ be such that $\overline{w}_0\(0,0\)=0$ and $\overline{w}_0\(r,0\)>0$ for small $r>0$. Let $D$ be the set of all functions $\hat{h}\in C_{l,\sym}^{p-2,\alpha}\(\Omega\)\cap C^2\(\overline\Omega\)$ such that
$$\int_\Omega G_{\overline{h}}\(\cdot,0\)\hat{h}f\(\overline{w}_0\)dx<0,$$
where $G_{\overline{h}}$ is the Green's function of the operator
$$\Delta-\frac{2\overline{\beta}_0\overline{h}f'\(\overline{w}_0\)}{DI_{\overline{h}}\(\overline{w}_0\).\overline{w}_0}$$ 
with boundary condition $\left.G_{\overline{h}}\(\cdot,0\)\right|_{\partial\Omega}=0$. Then $D$ is a non-empty open subset of $C_{l,\sym}^{p-2,\alpha}\(\Omega\)\cap C^2\(\overline\Omega\)$ and for every $\hat{h}\in D$, there exists $\varepsilon_0>0$ such that for every $\varepsilon\in\(0,\varepsilon_0\)$, there exist $\overline{\beta}_\varepsilon>0$ and $\overline{w}_\varepsilon\in E_{\overline{h}_\varepsilon,\overline{\beta}_\varepsilon}\cap C_{l,\sym}^{p,\alpha}\(\Omega\)\cap C^2\(\overline\Omega\)$, where $\overline{h}_\varepsilon:=\overline{h}+\varepsilon\hat{h}$, such that $\(\overline{\beta}_\varepsilon\)_{0\le\varepsilon\le\varepsilon_0}$ and $\(\overline{w}_\varepsilon\)_{0\le\varepsilon\le\varepsilon_0}$ are smooth in $\varepsilon$ and $\partial_\varepsilon\[\overline{w}_\varepsilon\(0\)\]_{\varepsilon=0}<0$ and $\overline{w}_\varepsilon\(0\)<0$ for all $\varepsilon\in\(0,\varepsilon_0\)$.
\end{proposition}

\proof[Proof]
We begin with proving that $D$ is not empty. Since $G_{\overline{h}}\(\cdot,0\)>0$ near 0, $\overline{w}_0\in C_{l,\sym}^{p,\alpha}\(\Omega\)$ and $\overline{w}_0\(r,0\)>0$ for small $r>0$, we obtain that there exists $x_0\in\Omega$ and $r_0>0$ such that $G_{\overline{h}}\(\cdot,0\)\overline{w}_0>0$ in $B\(x_0,r_0\)$ and $B\(x_0,r_0\)\subset\Omega_0$, where 
$$\Omega_0:=\left\{\(x_1,x_2\)\in\Omega:\,0<x_2<x_1\tan\(\pi/l\)\right\}.$$ 
Let $\chi\in C^\infty\(\Omega\)$ be such that $\chi>0$ in $B\(x_0,r_0\)$ and $\chi\equiv0$ in $B\(x_0,r_0\)^c$. Let $\chi_{\sym}$ be the unique function in $C_{l,\sym}^{p-2,\alpha}\(\Omega\)\cap C^2\(\overline\Omega\)$ such that $\chi_{\sym}\equiv\chi$ in $\Omega_0$. By symmetry and since $G_{\overline{h}}\(\cdot,0\)\chi \overline{w}_0>0$ in $B\(x_0,r_0\)$ and $\chi=0$ in $B\(x_0,r_0\)^c$, we obtain 
$$\int_\Omega G_{\overline{h}}\(\cdot,0\)\chi_{\sym}f\(\overline{w}_0\)dx=2l\int_\Omega G_{\overline{h}}\(\cdot,0\)\chi f\(\overline{w}_0\)dx>0,$$
i.e. $-\chi_{\sym}\in D$. This proves that $D$ is not empty. Now, we prove the second part of Proposition~\ref{Pr3}. Since $\overline{h}\in C_{l,\sym}^{p-2,\alpha}\(\Omega\)\cap C^2\(\overline\Omega\)$ and $\overline{w}_0\in E^{nd}_{\overline{h},\overline{\beta}_0}\cap C_{l,\sym}^{p,\alpha}\(\Omega\)\cap C^2\(\overline\Omega\)$, it follows from the implicit function theorem together with standard elliptic regularity that there exist a neighborhood $\mathcal{N}$ of $\overline{h}$ in $C_{l,\sym}^{p-2,\alpha}\(\Omega\)\cap C^2\(\overline\Omega\)$ and a smooth mapping $\overline{w}:\mathcal{N}\to C_{l,\sym}^{p,\alpha}\(\Omega\)\cap C^2\(\overline\Omega\)$ such that $\overline{w}\(\overline{h}\)=\overline{w}_0$ and for every $\tilde{h}\in \mathcal{N}$, $\widetilde{U}=\overline{w}\big(\tilde{h}\big)$ is a solution of the problem
\begin{equation}\label{Pr3Eq1}
\left\{\begin{aligned}&\Delta\widetilde{U}=\frac{2\overline{\beta}_0\tilde{h}f\big(\widetilde{U}\big)}{DI_{\overline{h}}\(\overline{w}_0\).\overline{w}_0}&&\text{in }\Omega\\&\widetilde{U}=0&&\text{on }\partial\Omega.\end{aligned}\right.
\end{equation}
Note that \eqref{Pr3Eq1} is equivalent to $\widetilde{U}\in E_{\tilde{h},\overline\beta\big(\tilde{h}\big)}$, where 
$$\overline\beta\big(\tilde{h}\big):=\frac{\overline{\beta}_0DI_{\tilde{h}}\big(\widetilde{U}\big).\widetilde{U}}{DI_{\overline{h}}\(\overline{w}_0\).\overline{w}_0}\,.$$ 
In particular, we obtain that for every $\hat{h}\in D$, there exists $\varepsilon_0\in\(0,1\)$ such that for every $\varepsilon\in\(0,\varepsilon_0\)$, there exist $\overline{\beta}_\varepsilon=\overline\beta\(\overline{h}_\varepsilon\)>0$ and $\overline{w}_\varepsilon=\overline{w}\(\overline{h}_\varepsilon\)\in E_{\overline{h}_\varepsilon,\overline{\beta}_\varepsilon}\cap C_{l,\sym}^{p,\alpha}\(\Omega\)\cap C^2\(\overline\Omega\)$, where $\overline{h}_\varepsilon:=\overline{h}+\varepsilon\hat{h}$ such that $\(\overline{\beta}_\varepsilon\)_{0\le\varepsilon\le\varepsilon_0}$ and $\(\overline{w}_\varepsilon\)_{0\le\varepsilon\le\varepsilon_0}$ are smooth in $\varepsilon$. Furthermore, by differentiating \eqref{Pr3Eq1}, we obtain
$$\left\{\begin{aligned}&\(\Delta-2\overline{\beta}_0\(DI_{\overline{h}}\(\overline{w}_0\).\overline{w}_0\)^{-1}\overline{h}f'\(\overline{w}_0\)\)\partial_\varepsilon\[\overline{w}_\varepsilon\]_{\varepsilon=0}=\frac{2\overline{\beta}_0\hat{h}f\(\overline{w}_0\)}{DI_{\overline{h}}\(\overline{w}_0\).\overline{w}_0}&&\text{in }\Omega\\&\partial_\varepsilon\[\overline{w}_\varepsilon\]_{\varepsilon=0}=0&&\text{on }\partial\Omega.\end{aligned}\right.$$
Since $\hat{h}\in D$, it follows that
$$\partial_\varepsilon\[\overline{w}_\varepsilon\(0\)\]_{\varepsilon=0}=\int_\Omega G_{\overline{h}}\(\cdot,0\)\hat{h}f\(\overline{w}_0\)dx<0.$$
Since $\overline{w}_0\(0\)=0$, by taking $\varepsilon_0$ smaller if necessary, we then obtain that $\overline{w}_\varepsilon\(0\)<0$ for all $\varepsilon\in\(0,\varepsilon_0\)$. This ends the proof of Proposition~\ref{Pr3}.
\endproof

\proof[Proof of Proposition~\ref{Pr0} (iii)]
The result of (iii) in Proposition~\ref{Pr0} is a direct consequence of Proposition~\ref{Pr3} applied to $\overline{\beta}_0:=\beta_\kappa$, $\overline{h}_0:=h_\kappa$ and $\overline{w}_0:=w_\kappa$.
\endproof

\section{Construction of the ansatz}\label{Sec3}

This section is devoted to the construction of our ansatz. We let $\Omega$, $l$, $\alpha$ and $h$ be as in Theorem~\ref{Th1}, fix $\beta>4\pi$, $\beta_0>0$ and $k\in\N^*$ such that $\beta = \beta_0 +4k\pi$ and let $a_0$, $\kappa_0$, $\varepsilon_0$, $\beta_\kappa$, $w_\kappa$, $h_\kappa$, $\beta_{\kappa,\varepsilon}$, $u_{\kappa,\varepsilon}$ and $h_{\kappa,\varepsilon}$ be as in Proposition~\ref{Pr0}. To prove that $\beta$ is an unstable energy level of $I_h$, by using a diagonal argument, one can easily see that it suffices to show that for every $\kappa\in\(-\kappa_0,\kappa_0\)\backslash\left\{0\right\}$, the number $\beta_\kappa+4k\pi$ is an unstable energy level of $I_{h_\kappa}$. In what follows, we fix $\kappa\in\(-\kappa_0,\kappa_0\)\backslash\left\{0\right\}$ and for the sake of simplicity, we drop the dependance in $\kappa$ from our notations. More precisely, we denote $\varepsilon_0:=\varepsilon_0\(\kappa\)$, $\beta_0:=\beta_\kappa$, $h_0:=h_\kappa$, $w_0:=w_\kappa$, $\beta_\varepsilon:=\beta_{\kappa,\varepsilon}$, $h_\varepsilon:=h_{\kappa,\varepsilon}$ and $w_\varepsilon:=w_{\kappa,\varepsilon}$. Remark that the new function $w_0$ still satisfies the properties of (i) in Proposition~\ref{Pr0} but now this function is moreover non-degenerate.

\subsection{The bubbles}\label{Sec31} 
For every $\gamma_0>0$, we let $\overline{B}_{\gamma_0}$ be the unique radial solution to the problem 
$$\left\{\begin{aligned}&\Delta\overline{B}_{\gamma_0}=f\(\overline{B}_{\gamma_0}\)&&\text{in }\R^2\\&\overline{B}_{\gamma_0}\(0\)=\gamma_0,&&\end{aligned}\right.$$
where $f\(s\):=s\exp\(s^2\)$ for all $s\in\R$. Note that by standard ordinary differential equations theory, $\overline{B}_{\gamma_0}$ is defined on $\[0,\infty\)$. For every $\varepsilon\in\(0,\varepsilon_0\)$, $\gamma_0>0$ and $x_0\in\Omega$, we then define
$$\overline{B}_{\varepsilon,\gamma_0,x_0}\(x\):=\overline{B}_{\gamma_0}\big(\sqrt{\lambda_\varepsilon h_\varepsilon\(x_0\)}\left|x-x_0\right|\big)\qquad\forall x\in\R^2,$$
where
$$\lambda_\varepsilon:=\frac{2\beta_\varepsilon}{DI_{h_\varepsilon}\(w_\varepsilon\).w_\varepsilon}\longrightarrow\frac{2\beta}{DI_{h}\(w_0\).w_0} =:\lambda>0,$$
so that $\overline{B}_{\varepsilon,\gamma_0,x_0}$ solves the problem 
$$\left\{\begin{aligned}&\Delta\overline{B}_{\varepsilon,\gamma_0,x_0}=\lambda_\varepsilon h_\varepsilon\(x_0\)f\(\overline{B}_{\varepsilon,\gamma_0,x_0}\)&&\text{in }\R^2\\&\overline{B}_{\varepsilon,\gamma_0,x_0}\(x_0\)=\gamma_0.&&\end{aligned}\right.$$
For every $r>0$ such that $B\(x_0,r\)\subset\Omega$, we then let $B_{\varepsilon,\gamma_0,x_0,r}:\Omega\to\R$ be the function defined as
\begin{equation}\label{DefBeps}
B_{\varepsilon,\gamma_0,x_0,r}\(x\):=\left\{\begin{aligned}&\overline{B}_{\varepsilon,\gamma_0,x_0}\(x\)-C_{\varepsilon,\gamma_0,x_0,r}+A_{\varepsilon,\gamma_0,x_0,r}H\(x,x_0\)&&x\in B\(x_0,r\)\\
&A_{\varepsilon,\gamma_0,x_0,r}G\(x,x_0\)&&\text{otherwise}\end{aligned}\right.
\end{equation}
for all $x\in\Omega$, where $G$ is the Green's function of the Laplace operator in $\Omega$ with boundary condition $\left.G\(\cdot,x_0\)\right|_{\partial\Omega}=0$, $H$ is the regular part of $G$, i.e. 
$$G\(x,x_0\)=\frac{1}{2\pi}\ln\frac{1}{\left|x-x_0\right|}+H\(x,x_0\)$$
and $A_{\varepsilon,\gamma_0,x_0,r}$, $C_{\varepsilon,\gamma_0,x_0,r}$ are constants chosen so that $B_{\varepsilon,\gamma_0,x_0,r}\in C^1\(\overline\Omega\)$, i.e.
\begin{align}
A_{\varepsilon,\gamma_0,x_0,r}&:=\int_{B(x_0,r)}\Delta\overline{B}_{\varepsilon,\gamma_0,x_0}\,dx,\label{Sec31Eq1}\allowdisplaybreaks\\
C_{\varepsilon,\gamma_0,x_0,r}&:=\overline{B}_{\gamma_0}\big(\sqrt{\lambda_\varepsilon h_\varepsilon\(x_0\)}r\big)-\frac{A_{\varepsilon,\gamma_0,x_0,r}}{2\pi}\ln\frac{1}{r}\,.\label{Sec31Eq2}
\end{align}

\subsection{The primary ansatz}\label{Sec32}  
For every $\varepsilon\in\(0,\varepsilon_0\)$ and $\delta\in\(0,1\)$, let $\Gamma_\varepsilon^k$ and $T_\varepsilon^k\(\delta\)$ be the sets of parameters defined as
\begin{align}
\Gamma_\varepsilon^k\(\delta\)&:=\big\{\gamma=\(\gamma_1,\dotsc,\gamma_k\)\in\(0,\infty\)^k:\,\left|\gamma_i-\overline{\gamma}_\varepsilon\right|<\delta\overline{\gamma}_\varepsilon,\,\forall i\in\left\{1,\dotsc,k\right\}\big\},\label{defGammaepsilon0}\allowdisplaybreaks\\
T_\varepsilon^k\(\delta\)&:=\Big\{\tau=\(\tau_1,\dotsc,\tau_k\)\in\R^k:\,{-\frac{kd_\varepsilon}{\delta}<\tau_1<\dots<\tau_k<\frac{kd_\varepsilon}{\delta}}\nonumber\\
&\hspace{139pt}\text{and }\left|\tau_i-\tau_j\right|>\delta d_\varepsilon,\,
\forall i,j\in\left\{1,\dotsc,k\right\},i\ne j\Big\},\label{defTepsilon}
\end{align}
where 
\begin{equation}\label{Sec32Eq1}
\overline{\gamma}_\varepsilon:=\frac{2\(k+l-1\)}{l\left|w_\varepsilon\(0\)\right|}\ln\frac{1}{\left|w_\varepsilon\(0\)\right|}\quad\text{and}\quad d_\varepsilon:=\overline{\gamma}_\varepsilon^{-1/l}.
\end{equation} 
From \eqref{Sec32Eq1}, $w_0(0)=0$ and $\partial_\ve[w_\ve(0)]_{\ve=0}\ne 0$, we get
\begin{equation}\label{Sec32Eq1bis}
w_\ve(0)\sim  -\frac{2\(k+l-1\)}{l}\frac{\ln\gammae}{\gammae},\quad \ve\sim\frac{w_\ve(0)}{\partial_\ve[w_\ve(0)]_{\ve=0}}=\bigO\(\frac{\ln\gammae}{\gammae}\),\quad \text{as }\ve\to 0,
\end{equation}
and since $w_0\(r,0\)\sim a_0r^l$ as $r\to0$, using the continuity of $\partial_\ve w_\ve(x)$ jointly in $\ve$ and $x$, and \eqref{Sec32Eq1bis} we get for some $\ve_1 \in (0,\ve)$
\begin{equation}\label{Pr5Eq3}
\begin{split}
w_\varepsilon\(\overline{\tau_i}\)&=  w_0\(\overline{\tau_i}\) +[w_\ve\(\overline{\tau_i}\)-w_0\(\overline{\tau_i}\)] = \bigO\(d_\varepsilon^l\) +\ve \partial_\ve[w_\ve\(\overline{\tau_i}\)]_{\ve=\ve_1} \\ 
&\sim \ve \partial_\ve[w_\ve\(0\)]_{\ve=0} \sim w_\varepsilon\(0\)\sim-\frac{2\(k+l-1\)\ln\overline\gamma_\varepsilon}{l\overline\gamma_\varepsilon},\quad \text{as }\varepsilon\to0,
\end{split}
\end{equation}
uniformly in $\tau\in T_\varepsilon^k\(\delta\)$.
For every $(\gamma,\tau)\in \Gamma_\varepsilon^k\(\delta\)\times T_\varepsilon^k\(\delta\)$, we define
$$\widetilde{U}_{\varepsilon,\gamma,\tau}:=w_\varepsilon+\sum_{i=1}^kB_{\varepsilon,\gamma_i,\tau_i},$$
where $B_{\varepsilon,\gamma_i,\tau_i}:=B_{\varepsilon,\gamma_i,\overline{\tau_i},r_\varepsilon}$, $\overline{\tau_i}:=\(\tau_i,0\)$, and for $\delta_0\in(0,1/2)$ to be fixed later,
\begin{equation}\label{Sec32Eq2}
r_\varepsilon:=\overline\mu_\varepsilon^{\delta_0},\quad \overline\mu_\varepsilon^2:=\exp\big(-\overline\gamma_\varepsilon^2\big).
\end{equation}

\begin{claim}\label{Claim1} 
Set $A_{\varepsilon,\gamma_i,\tau_i}:=A_{\varepsilon,\gamma_i,\overline{\tau_i},r_\varepsilon}$ and $C_{\varepsilon,\gamma_i,\tau_i}:=C_{\varepsilon,\gamma_i,\overline{\tau_i},r_\varepsilon}$.
For every $\delta\in\(0,1\)$ and $i\in\left\{1,\dotsc,k\right\}$, we have
\begin{align}
A_{\varepsilon,\gamma_i,\tau_i}&=\frac{4\pi}{\gamma_i}+\bigO\(\frac{1}{\overline\gamma_\varepsilon^3}\)\,,\quad C_{\varepsilon,\gamma_i,\tau_i}=-\frac{2\ln\overline\gamma_\varepsilon}{\gamma_i}+\bigO\(\frac{1}{\overline\gamma_\varepsilon}\),\label{Sec32Eq3}\allowdisplaybreaks\\
\partial_{\gamma_i}\[A_{\varepsilon,\gamma_i,\tau_i}\]&=-\frac{4\pi}{\gamma_i^2}+\bigO\(\frac{1}{\overline\gamma_\varepsilon^4}\)\quad\text{and}\quad\partial_{\gamma_i}\[C_{\varepsilon,\gamma_i,\tau_i}\]=\frac{2\ln\overline\gamma_\varepsilon}{\gamma_i^2}+\bigO\(\frac{1}{\overline\gamma_\varepsilon^2}\)\label{Sec32Eq4}
\end{align}
as $\varepsilon\to0$, uniformly in $\(\gamma,\tau\)\in \Gamma_\varepsilon^k\(\delta\)\times T_\varepsilon^k\(\delta\)$. Furthermore, for every $a\ge0$ and $\delta'\in\(0,1-\sqrt{\delta_0}\)$ (i.e. such that $\(1-\delta'\)^2>\delta_0$), we have
\begin{equation}
\partial_{\tau_i}\[A_{\varepsilon,\gamma_i,\tau_i}\]=\bigO\(\frac{1}{\gammae^a}\)\quad\text{and}\quad\partial_{\tau_i}\[C_{\varepsilon,\gamma_i,\tau_i}\]\sim-\frac{\partial_{x_1}h_\varepsilon\(\overline{\tau_i}\)}{h_\varepsilon\(\overline{\tau_i}\)\gamma_i}=\bigO\(\frac{1}{\overline\gamma_\varepsilon}\)\label{Sec32Eq5}
\end{equation}
as $\varepsilon\to0$, uniformly in $\(\gamma,\tau\)\in \Gamma_\varepsilon^k\(\delta'\)\times T_\varepsilon^k\(\delta\)$. 
\end{claim}

The proof of Claim~\ref{Claim1} is based on a precise asymptotic study of the bubbles $\overline{B}_\gamma$ and is postponed to the Appendix.

\subsection{Correction of the error at the bottom of the bubbles}\label{Sec33}  
In this section, we modify our ansatz so to correct the error made outside the balls $B\(\overline{\tau_i},2r_\varepsilon\)$. We prove the following:

\begin{proposition}\label{Pr4}
Let $\Omega$, $l$ $\alpha$ and $h$ be as in Theorem~\ref{Th1}. Let $k$, $\varepsilon_0$, $h_\varepsilon$, $w_\varepsilon$, $\lambda_\varepsilon$, $\overline\gamma_\varepsilon$, $\overline{\tau_i}$, $r_\varepsilon$, $\delta_0$, $\Gamma_\varepsilon^k\(\delta\)$, $T_\varepsilon^k\(\delta\)$ and $\widetilde{U}_{\varepsilon,\gamma,\tau}$ be as in Sections~\ref{Sec31} and~\ref{Sec32}. Let $\chi\in C^\infty\(\R\)$ be such that $0\le\chi\le1$ in $\R$, $\chi\equiv1$ in $\(-\infty,1\]$ and $\chi\equiv0$ in $\[2,\infty\)$. Define
$$\chi_{\varepsilon,\tau}\(x\):=1-\sum_{i=1}^k\chi\(\(\left|x-\overline{\tau_i}\right|+r_\varepsilon^2-r_\varepsilon\)/r_\varepsilon^2\)\qquad\forall x\in\R^2.$$
For every $\delta\in\(0,1\)$ and $\delta'\in (0,1-\sqrt{2\delta_0})$, there exist $\varepsilon_1(\delta,\delta')\in (0,\varepsilon_0)$ 
and $C_1=C_1\(\delta,\delta'\)>0$ such that for every $\varepsilon\in\(0,\varepsilon_1\(\delta,\delta'\)\)$ and $(\gamma,\tau)\in \Gamma_\varepsilon^k\(\delta'\)\times T_\varepsilon^k\(\delta\)$, there exists a unique solution $\Psi_{\varepsilon,\gamma,\tau}\in C^{l,\alpha}\(\Omega\)\cap C^2\(\overline\Omega\)$ to the problem
\begin{equation}\label{Pr4Eq1}
\left\{\begin{aligned}&\Delta\(w_\varepsilon+\Psi_{\varepsilon,\gamma,\tau}\)=\lambda_\varepsilon h_\varepsilon\chi_{\varepsilon,\tau}f\big(\widetilde{U}_{\varepsilon,\gamma,\tau}+\Psi_{\varepsilon,\gamma,\tau}\big)&&\text{in }\Omega\\&\Psi_{\varepsilon,\gamma,\tau}=0&&\text{on }\partial\Omega\end{aligned}\right.
\end{equation}
such that $\Psi_{\varepsilon,\gamma,\tau}$ is even in $x_2$, continuously differentiable in $(\gamma,\tau)$ and
\begin{align}
&\left\|\Psi_{\varepsilon,\gamma,\tau}\right\|_{C^1}\le\frac{C_1}{\overline{\gamma}_\varepsilon}\,,\quad\left\|D_\gamma\[\Psi_{\varepsilon,\gamma,\tau}\]\right\|_{C^1}\le\frac{C_1}{\overline{\gamma}_\varepsilon^2},\label{Pr4Eq2}\allowdisplaybreaks\\
&\left\|D_\tau\[\Psi_{\varepsilon,\gamma,\tau}\]\right\|_{H^1}+\left\|D_\tau\[\Psi_{\varepsilon,\gamma,\tau}\]\right\|_{C^0}\le\frac{C_1}{\overline{\gamma}_\varepsilon}\,.\label{Pr4Eq3}
\end{align}
Finally, setting $U_{\ve,\gamma,\tau}:= \widetilde{U}_{\varepsilon,\gamma,\tau}+\Psi_{\ve,\gamma,\tau}$, there exists $p_0=p_0\(\delta_0,\delta'\)$ such that for every $p\in\[1,p_0\]$, $a\ge0$ and $i\in\left\{1,\dotsc,k\right\}$, we have
\begin{align}
&\left\|\exp\big(U_{\ve,\gamma,\tau}^2\big)\mathbf{1}_{A\(\overline{\tau_i},r_\ve,R_\ve\)}\right\|_{L^p}=\bigO\(\frac{1}{\gammae^a}\),\,\left\|\exp\big(U_{\ve,\gamma,\tau}^2\big)B_{\ve,\gamma_i,\tau_i}^a\mathbf{1}_{\Omega_{R_\ve,\tau}}\right\|_{L^p}=\bigO\(\frac{1}{\gammae^a}\),\label{Pr4Eq3b}\allowdisplaybreaks\\
&\left\|\partial_{\tau_i}\[\chi_{\ve,\tau}\]f\(U_{\ve,\gamma,\tau}\)\right\|_{L^p}=\bigO\(\frac{1}{\gammae^a}\),\,\left\| f'\(U_{\ve,\gamma,\tau}\)\partial_{\tau_i}\[U_{\ve,\gamma,\tau}\]\mathbf{1}_{\Omega_{r_\ve,\tau}}\right\|_{L^p}=\bigO\(\gammae\)\label{Pr4Eq3c}
\end{align}
uniformly in $(\gamma,\tau)\in \Gamma_\varepsilon^k\(\delta'\)\times T_\varepsilon^k\(\delta\)$, where $R_\ve:=\exp\(-\gammae\)$ and
\begin{equation}\label{defAOmega}
A\(\overline{\tau_i},r,R\):=B\(\overline{\tau_i},R\)\backslash B\(\overline{\tau_i},r\)\quad\text{and}\quad\Omega_{r,\tau}:=\Omega\backslash\(\bigcup_{i=1}^k B\(\overline{\tau_i},r\)\)
\end{equation}
for all $R>r>0$
\end{proposition}

In other words, the function
\begin{equation}\label{EqUegt}
U_{\varepsilon,\gamma,\tau}:=\widetilde{U}_{\varepsilon,\gamma,\tau}+\Psi_{\ve,\gamma,\tau}=w_\varepsilon+\sum_{i=1}^kB_{\varepsilon,\gamma_i,\tau_i}+\Psi_{\ve,\gamma,\tau},
\end{equation}
where $\Psi_{\ve,\gamma,\tau}$ is given by Proposition~\ref{Pr4}, is an exact solution outside the balls $B\(\overline\tau_i,r_\ve+r_\ve^2\)$ for all $i\in\left\{1,\dotsc,k\right\}$, and it satisfies
\begin{equation}\label{Pr4Eq4b}
\Delta U_{\varepsilon,\gamma,\tau}=\left\{
\begin{array}{ll}
\Delta B_{\ve,\gamma_i,\tau_i} =\Delta \overline{B}_{\ve,\gamma_i,\overline{\tau_i}}=\lambda_\ve h_\ve\(\overline{\tau_i}\)f\big(\overline{B}_{\ve,\gamma_i,\overline{\tau_i}}\big) & \text{in } B\(\overline{\tau_i},r_\ve\)\\
\lambda_\ve h_\ve \chi_{\ve,\tau}f\(U_{\ve,\gamma,\tau}\)&\text{in }\Omega_{r_\ve,\tau}.
\end{array}
\right.
\end{equation}
Since the proof of Proposition~\ref{Pr4} is lenghty, but not necessary to understand the rest of the construction, it is postponed to Section~\ref{ProofPr4}.

For later use, we also observe that \eqref{DefBeps}, \eqref{Sec32Eq1}, \eqref{Sec32Eq2}, \eqref{Sec32Eq3} and \eqref{Pr4Eq2} give $U_{\ve,\gamma,\tau}=\delta_0\gammae(1+\smallo(1))$ in $\Omega^i_\ve:=B(\overline{\tau_i},r_\ve+r_\ve^2)\setminus B(\overline{\tau_i},r_\ve)$, hence
\begin{equation}\label{Estannuli}
f(U_{\ve,\gamma,\tau})=\bigO\Big(\overline{\mu}_\ve^{-2\delta_0^2+\smallo(1)}\Big),\quad \text{in }\Omega^i_\ve.
\end{equation}

\subsection{Adjustment of the values at the centers of the bubbles}\label{Sec34} 
In this section, we refine the range of the parameters $\gamma_i$ so to optimize the error made in the regions $B\(\overline{\tau_i},r_\ve\)$. Let us start by expanding
\begin{equation}\label{UBEF}
U_{\ve,\gamma,\tau}\(x\)=\overline{B}_{\ve,\gamma_i,\overline{\tau_i}}\(x\)+ E^{\(i\)}_{\ve,\gamma,\tau}+F^{\(i\)}_{\ve,\gamma,\tau}\(x\)
\end{equation}
for all $x\in B\(\overline{\tau_i},r_\ve\)$,
where
\begin{align}
&E^{\(i\)}_{\ve,\gamma,\tau}:=w_\ve\(\overline{\tau_i}\)-C_{\ve,\gamma_i,\tau_i}+ A_{\ve,\gamma_i,\tau_i}H\(\overline{\tau_i},\overline{\tau_i}\)+\sum_{j\ne i}A_{\ve,\gamma_j,\tau_j}G\(\overline{\tau_i},\overline{\tau_j}\)+\Psi_{\ve,\gamma,\tau}\(\overline{\tau_i}\),\label{eqEive}\allowdisplaybreaks\\
&F^{\(i\)}_{\ve,\gamma,\tau}\(x\):=w_\ve\(x\)-w_\ve\(\overline{\tau_i}\)+A_{\ve,\gamma_i,\tau_i}\(H(x,\overline{\tau_i}\)-H\(\overline{\tau_i},\overline{\tau_i})\)\nonumber\\
&\qquad\qquad\qquad+\sum_{j\ne i}A_{\ve,\gamma_j,\tau_j} \( G\(x,\overline{\tau_j}\)-G\(\overline{\tau_i},\overline{\tau_j}\)\)+\Psi_{\ve,\gamma,\tau}\(x\)-\Psi_{\ve,\gamma,\tau}\(\overline{\tau_i}\).\label{eqFive}
\end{align}

Note that $F^{\(i\)}_{\ve,\gamma,\tau}\(\overline{\tau_i}\)=0$, so $F^{\(i\)}_{\ve,\gamma,\tau}$ is small in $B\(\overline \tau_i,r_\ve\)$. Instead the constant $E^{\(i\)}_{\ve,\gamma,\tau}$ might be large depending on the choice of $\gamma$ and $\tau$. In the next proposition we show that we can choose $\overline{\gamma}_\ve\(\tau\)\sim \overline{\gamma}_\ve$ depending on $\tau$ and $\ve$ in such a way that $E^{\(i\)}_{\ve,\overline{\gamma}_\ve\(\tau\),\tau}=0$ for all $i\in\left\{1,\dotsc,k\right\}$.

\begin{proposition}\label{Pr5}
Let $\delta_0$, $\varepsilon_1$ and $\Psi_{\varepsilon,\gamma,\tau}$ be as in Proposition~\ref{Pr4}. Then for every $\delta\in\(0,1\)$ and $\delta'\in\(0,1-\sqrt{2\delta_0}\)$, there exists $\varepsilon_2\(\delta,\delta'\)\in\(0,\varepsilon_1\(\delta,\delta'\)\)$ such that for every $\varepsilon\in\(0,\varepsilon_2\(\delta,\delta'\)\)$ and $\tau\in T_\varepsilon^k\(\delta\)$, there exists a unique $\overline{\gamma}_\varepsilon\(\tau\)=\(\overline{\gamma}_{1,\varepsilon}\(\tau\),\dotsc,\overline{\gamma}_{k,\varepsilon}\(\tau\)\)\in\Gamma_\varepsilon^k\(\delta'\)$ such that $\overline{\gamma}_{k,\varepsilon}\(\tau\)$ is continuous in $\tau$ and for every $i\in\left\{1,\dots,k\right\}$, we have
\begin{equation}\label{Pr5Eq1}
U_{\varepsilon,\overline{\gamma}_\varepsilon\(\tau\),\tau}\(\overline{\tau_i}\)=\overline\gamma_{i,\varepsilon}\(\tau\)\quad\text{and}\quad\overline\gamma_{i,\varepsilon}\(\tau\)\sim\overline\gamma_\varepsilon
\end{equation}
as $\varepsilon\to0$, uniformly in $\tau\in T_\varepsilon^k\(\delta\)$.
\end{proposition}

\proof
For every $\gamma\in\Gamma_\varepsilon^k\(\delta'\)$, we denote $\tilde\gamma:=\gamma/\overline\gamma_\varepsilon$. We let $I:=\(1-\delta',1+\delta'\)$ and $E_{\varepsilon,\tau}:I^k\to\R^k$, $E_{\varepsilon,\tau}=\big(E_{\varepsilon,\tau}^{\(1\)},\dotsc,E_{\varepsilon,\tau}^{(k)}\big)$ be the function defined by
$$E_{\varepsilon,\tau}^{\(i\)}\(\tilde\gamma\):=\frac{\overline\gamma_\varepsilon}{\ln\overline\gamma_\varepsilon}E^{\(i\)}_{\ve,\gamma,\tau}\quad\forall\gamma\in I^k,\,i\in\left\{1,\dotsc,k\right\}.$$
In particular, $E_{\varepsilon,\tau}\in C^1(I^k)$. By definition of $d_\varepsilon$, $G$ and $H$, we obtain
\begin{equation}\label{Pr5Eq4}
G\(\overline{\tau_i},\overline{\tau_j}\)\sim\frac{1}{2\pi}\ln\frac{1}{d_\varepsilon}\sim\frac{\ln\overline\gamma_\varepsilon}{2l\pi}\quad\text{and}\quad H\(\overline{\tau_i},\overline{\tau_i}\)=\bigO\(1\)
\end{equation}
as $\varepsilon\to0$, uniformly in $\tau\in T_\varepsilon^k\(\delta\)$. It follows from \eqref{Pr5Eq3}, \eqref{Sec32Eq3}, \eqref{Sec32Eq4}, \eqref{Pr4Eq2} and \eqref{Pr5Eq4} that $E_{\varepsilon,\tau}\to E_{0}=\big(E_{0}^{\(1\)},\cdots,E_{0}^{(k)}\big)$ in $C^1(I^k)$ as $\varepsilon\to0$, uniformly in $\tau\in T^k_\varepsilon\(\delta\)$, where
$$E_{0}^{\(i\)}\(\tilde\gamma\):=\frac{2}{\tilde\gamma_i}+\frac{2}{l}\sum_{j\ne i}\frac{1}{\tilde\gamma_j}-\frac{2\(k+l-1\)}{l}$$
for all $i\in\left\{1,\dotsc,k\right\}$ and $\tilde\gamma=\(\tilde\gamma_1,\dotsc,\tilde\gamma_k\)\in I^k$. In particular, 
\begin{equation}\label{Pr5Eq8}
E_0\(1,\dotsc,1\)=0\quad\text{and}\quad\det\(DE_0\(1,\dotsc,1\)\)\ne0.
\end{equation}
By applying the implicit function theorem, it follows from \eqref{Pr5Eq8} that there exists $\varepsilon_2\(\delta,\delta'\)\in\(0,\varepsilon_1\(\delta,\delta'\)\)$ such that for every $\varepsilon\in\(0,\varepsilon_2\(\delta,\delta'\)\)$ and $\tau\in T_\varepsilon^k\(\delta\)$, there exists a unique $\tilde{\gamma}_\varepsilon\(\tau\)\in I^k$ such that $\tilde{\gamma}_\varepsilon\(\tau\)$ is continuous in $\tau$, $E_{\varepsilon,\tau}\(\tilde{\gamma}_\varepsilon\(\tau\)\)=0$ and $\tilde{\gamma}_\varepsilon\(\tau\)\to\(1,\dotsc,1\)$ as $\varepsilon\to0$, uniformly in $\tau\in T_\varepsilon^k\(\delta\)$, i.e. there exists a unique $\overline{\gamma}_\varepsilon\(\tau\)=\overline\gamma_\varepsilon\tilde{\gamma}_\varepsilon\(\tau\)\in\Gamma_\varepsilon^k\(\delta'\)$ such that $\overline\gamma_\varepsilon\(\tau\)$ is continuous in $\tau$ and \eqref{Pr5Eq1} holds true. This ends the proof of Proposition~\ref{Pr5}.
\endproof


Now, we refine the set $\Gamma_\varepsilon^k(\delta')$ by defining
\begin{equation}\label{defGammaepsilon}
\overline{\Gamma}_\varepsilon^k\(\tau\):=\big\{\gamma=\(\gamma_1,\dotsc,\gamma_k\)\in\(0,\infty\)^k:\,|\gamma_i-\overline\gamma_{i,\varepsilon}\(\tau\)|<\frac{\delta_\ve}{{\gammae}},\;\forall i\in\left\{ 1,\dotsc,k\right\}\big\},
\end{equation}
where $\overline\gamma_{1,\varepsilon}\(\tau\),\dotsc,\overline\gamma_{k,\varepsilon}\(\tau\)$ are the numbers obtained in Proposition~\ref{Pr5} and
\begin{equation}\label{defdeltaepsilon}
\delta_\ve:= \overline\mu_\varepsilon^{\delta_1+1/2},
\end{equation}
where $\overline\mu_\varepsilon$ is as in \eqref{Sec32Eq2} and $\delta_1\in\(0,1/2\)$ is a number that we shall fix later. 

\smallskip
Note that for every $\delta,\delta'\in\(0,1\)$, we have 
\begin{equation}\label{inclusion}
\overline{\Gamma}_\varepsilon^k\(\tau\)\subset\Gamma_\varepsilon^k\(\delta'\)
\end{equation}
for small $\varepsilon>0$, uniformly in $\tau\in T_\varepsilon^k\(\delta\)$. Therefore, we can fix 
$$\delta':=\frac{1-\sqrt{2\delta_0}}{2}$$ 
in what follows and let $\varepsilon_3\(\delta\)\in\(0,\varepsilon_2\(\delta,\delta'\)\)$ be such that \eqref{inclusion} holds true together with the results of Propositions~\ref{Pr4} and~\ref{Pr5} for all $\varepsilon\in\(0,\varepsilon_3\(\delta\)\)$ and $\tau\in T_\varepsilon^k\(\delta\)$.

\subsection{An additional variation in the directions of the bubbles}\label{Sec35}
We now introduce an additional family of parameters $\theta=\(\theta_1,\dotsc,\theta_k\)\in\R^k$ and define our final ansatz as
$$U_{\varepsilon,\gamma,\tau,\theta}:=U_{\varepsilon,\gamma,\tau}+\sum_{i=1}^k\theta_iB_{\varepsilon,\gamma_i,\tau_i}=w_\ve +\sum_{i=1}^k\(1+\theta_i\)B_{\varepsilon,\gamma_i,\tau_i}+\Psi_{\varepsilon,\gamma,\tau},$$
for
\begin{equation}\label{defThetaepsilon}
\theta\in \Theta^k_\varepsilon(\delta):=\bigg\{\theta=\(\theta_1,\dotsc,\theta_k\)\in\R^k:\,\left|\theta_i\right|<\frac{\delta_\ve \ln\overline\gamma_\varepsilon}{\overline\gamma_\varepsilon^4},\;\forall i\in\left\{1,\dotsc,k\right\}\bigg\},
\end{equation}
where $\overline\gamma_\varepsilon$ and $\delta_\varepsilon$ are as in \eqref{Sec32Eq1} and \eqref{defdeltaepsilon}. Finally, we define
$$P^k_\varepsilon\(\delta\):=\Big\{\(\gamma,\tau,\theta\)\in\(0,\infty\)^k\times T^k_\varepsilon\(\delta\)\times \Theta^k_\varepsilon:\,\gamma\in\overline\Gamma^k_\varepsilon\(\tau\)\Big\},$$
where $T^k_\varepsilon\(\delta\)$, $\overline\Gamma^k_\varepsilon\(\tau\)$ and $\Theta^k_\varepsilon$ are defined as in \eqref{defTepsilon}, \eqref{defGammaepsilon} and \eqref{defThetaepsilon}, respectively.

\subsection{Pointwise estimates near the centers of the bubbles}\label{Sec36}

We can now prove the following:

\begin{proposition}\label{PrEive} Let $\overline\gamma_{\ve}\(\tau\)$ be as in Proposition~\ref{Pr5}. Then for every $i\in\left\{1,\dotsc,k\right\}$ and $\delta\in \(0,1\)$ we have
\begin{align*}
&\partial_{\gamma_i}\big[E^{\(i\)}_{\ve,\gamma,\tau}\big]=-\frac{2\ln \overline{\gamma}_\ve}{\overline{\gamma}_\ve^2}+\smallo\(\frac{\ln \overline{\gamma}_\ve}{\overline{\gamma}_\ve^2}\),\quad\partial_{\gamma_j}\big[E^{\(i\)}_{\ve,\gamma,\tau}\big]=-\frac{2\ln \overline{\gamma}_\ve}{l\overline{\gamma}_\ve^2}+\smallo\(\frac{\ln \overline{\gamma}_\ve}{\overline{\gamma}_\ve^2}\),\quad \text{for }j\ne i,\allowdisplaybreaks\\
&E^{\(i\)}_{\ve,\gamma,\tau}=-\frac{2\ln \overline{\gamma}_{\ve}}{\overline{\gamma}_{\ve}^2}\((\gamma_i-\overline{\gamma}_{i,\ve}\(\tau\))+\sum_{j\ne i}\frac{\gamma_j-\overline{\gamma}_{j,\ve}\(\tau\)}{l}\)+\smallo\(\left|\gamma-\overline{\gamma}_{\ve}\(\tau\)\right|\frac{\ln \overline{\gamma}_{\ve}}{\overline{\gamma}_{\ve}^2}\)
\end{align*}
as $\ve \to 0$, uniformly in $\tau \in  T_\ve^k\(\delta\)$ and $\gamma \in \overline{\Gamma}_\ve^k\(\tau\)$.
\end{proposition}

\proof Using \eqref{Sec32Eq4}, \eqref{eqEive}, \eqref{Pr4Eq2} and noticing that for $\(\gamma_1,\dots,\gamma_k\)\in \overline{\Gamma}_\ve^k\(\tau\)$ we have $\gamma_j\sim \overline \gamma_\ve$ for $j=1,\dots,k$, we get
\begin{align*}
\partial_{\gamma_i}\big[E^{\(i\)}_{\ve,\gamma,\tau}\big]&=-\partial_{\gamma_i}\[C_{\ve,\gamma_i,\tau_i}\]+\partial_{\gamma_i}\[A_{\ve,\gamma_i,\tau_i}\] H\(\overline{\tau_i},\overline{\tau_i}\) +\partial_{\gamma_i}\[\Psi_{\ve,\gamma,\tau}\]\\
&=-\frac{2\ln \overline{\gamma}_\ve}{\gamma_i^2}\(1+\smallo\(1\)\)-\frac{4\pi}{\gamma_i^2}\bigO\(1\)+\bigO\(\frac{1}{\overline \gamma_\ve^2}\)=-\frac{2\ln \overline{\gamma}_\ve}{\overline{\gamma}_\ve^2}+\smallo\(\frac{\ln \overline{\gamma}_\ve^2}{\overline{\gamma}_\ve^2}\)
\end{align*}
as $\ve\to 0$. For the case $j\ne i$, using \eqref{Sec32Eq1}, we estimate
$$G\(\overline{\tau_i},\overline{\tau_j}\)=\frac{1}{2\pi}\ln\frac{1}{d_\ve}+\bigO\(1\)=\frac{1}{2\pi l}\ln \overline{\gamma}_\ve +\bigO\(1\),$$
uniformly in $\tau \in T_\ve^k\(\delta\),$
hence
\begin{align*}
\partial_{\gamma_j}\big[E^{\(i\)}_{\ve,\gamma,\tau}\big]&=\partial_{\gamma_j}\[A_{\ve,\gamma_j,\tau_j}\] G\(\overline{\tau_i},\overline{\tau_j}\) +\partial_{\gamma_j}\[\Psi_{\ve,\gamma,\tau}\]\\
&=-\frac{4\pi}{\gamma_j^2}\(\frac{1}{2\pi l}\ln\overline{\gamma}_\ve +\bigO\(1\)\)+\bigO\(\frac{1}{\overline \gamma_\ve^2}\)=-\frac{2\ln \overline{\gamma}_\ve}{l\overline{\gamma}_\ve^2}+\smallo\(\frac{\ln \overline{\gamma}_\ve^2}{\overline{\gamma}_\ve^2}\).
\end{align*}
Now, since $E^{\(i\)}_{\ve,\overline{\gamma}_\ve\(\tau\),\tau}=0$, integrating the gradient of $E^{\(i\)}_{\ve,\gamma,\tau}$ with respect to $\gamma$ from $\overline{\gamma}_\ve\(\tau\)$ to a generic $\gamma\in \overline{\Gamma}^k_\ve\(\tau\)$, the last identity follows at once.
\endproof

\begin{proposition}\label{PrFive}
For every $i\in\left\{1,\dotsc,k\right\}$ and $\delta\in\(0,1\)$, we have
\begin{equation}\label{Pr6Eq0}
F^{\(i\)}_{\ve,\gamma,\tau}\(x\)=\Bigg(a_0l\tau_i^{l-1}-\frac{2}{\overline\gamma_\varepsilon}\sum_{j\ne i}\frac{1}{\tau_i-\tau_j}\Bigg)\(x_1-\tau_i\)+\smallo\(\frac{\left|x-\overline{\tau_i}\right|}{\overline\gamma_\varepsilon d_\ve}\),
\end{equation}
and for every $i,j\in \left\{1,\dotsc,k\right\}$,
\begin{equation}\label{Pr6Eq0b}
\partial_{\gamma_j}\big[F^{\(i\)}_{\ve,\gamma,\tau}\big]\(x\)= \bigO\(\frac{\left|x-\overline{\tau_i}\right|}{\overline\gamma_\ve^2 d_\ve}\)
\end{equation}
as $\varepsilon\to0$, uniformly in $x=\(x_1,x_2\)\in B\(\overline{\tau_i},r_\varepsilon\)$ and $\(\gamma,\tau,\theta\)\in P^k_\varepsilon\(\delta\)$.
\end{proposition}

\proof
Note that $F^{\(i\)}_{\ve,\gamma,\tau}\(\overline{\tau_i}\)=0$. Then, by using \eqref{Sec32Eq1bis}, \eqref{Sec32Eq3} and \eqref{Pr4Eq3} and since $w_\varepsilon=w_0+\bigO\(\varepsilon\)$ in $C^1\(\Omega\)$, $w_0\(r,0\)\sim a_0r^l$ as $r\to0$ and $\partial_{x_2}w_0\(0,0\)=0$, we obtain 
\begin{multline*}
F^{\(i\)}_{\ve,\gamma,\tau}\(x\)=\int_0^1\<\nabla F^{\(i\)}_{\ve,\gamma,\tau}\(\(1-t\)\overline{\tau_i}+tx\),x-\overline{\tau_i}\>dt\\
=\int_0^1\Bigg(\<\nabla w_0\(\(1-t\)\overline{\tau_i}+tx\),x-\overline{\tau_i}\>-\sum_{j\ne i}\frac{A_{\varepsilon,\gamma_j,\tau_j}\<\(1-t\)\overline{\tau_i}+tx-\overline{\tau_j},x-\overline{\tau_i}\>}{2\pi\left|\(1-t\)\overline{\tau_i}+tx-\overline{\tau_j}\right|^2}\Bigg)dt\\
+\bigO\(\varepsilon\left|x-\overline\tau_i\right|\)=\Bigg(a_0l\tau_i^{l-1}-\frac{2}{\overline\gamma_\varepsilon}\sum_{j\ne i}\frac{1}{\tau_i-\tau_j}\Bigg)\(x_1-\tau_i\)+\smallo\(\frac{\left|x-\overline\tau_i\right|}{\overline\gamma_\varepsilon^{1-1/l}}\)
\end{multline*}
as $\varepsilon\to0$, uniformly in $x=\(x_1,x_2\)\in B\(\overline{\tau_i},r_\varepsilon\)$, $\tau\in T^k_\varepsilon\(\delta\)$ and $\gamma\in\overline{\Gamma}^k_\varepsilon\(\tau\)$, hence proving \eqref{Pr6Eq0}. Differentiating \eqref{eqFive} and using Claim~\ref{Claim1}, \eqref{Pr6Eq0b} also follows at once.
\endproof

\begin{proposition}\label{Pr6}
For every $i\in\left\{1,\dotsc,k\right\}$ and $\delta\in\(0,1\)$, we have
\begin{multline}\label{Pr6Eq1}
U_{\varepsilon,\gamma,\tau,\theta}\(x\)=\overline{B}_{\varepsilon,\gamma_i,\tau_i}\(x\)+\Bigg(a_0l\tau_i^{l-1}-\frac{2}{\overline\gamma_\varepsilon}\sum_{j\ne i}\frac{1}{\tau_i-\tau_j}\Bigg)\(x_1-\tau_i\)\\
+\smallo\(\frac{\left|x-\overline{\tau_i}\right|}{\overline\gamma_\varepsilon^{1-1/l}}\)+\bigO\(\frac{\delta_\ve\ln\overline\gamma_\varepsilon}{\overline\gamma_\varepsilon^3}\)
\end{multline}
as $\varepsilon\to0$, uniformly in $x=\(x_1,x_2\)\in B\(\overline{\tau_i},r_\varepsilon\)$ and $\(\gamma,\tau,\theta\)\in P^k_\varepsilon\(\delta\)$. In particular, for every $\delta\in\(0,1\)$, there exists $\varepsilon_4\(\delta\)\in\(0,\varepsilon_3\(\delta\)\)$, where $\varepsilon_3\(\delta\)$ is as in Section~\ref{Sec35}, such that
\begin{equation}\label{positivity}
\overline{B}_{\varepsilon,\gamma_i,\tau_i}\(x\)>0\text{ and }U_{\varepsilon,\gamma,\tau,\theta}>0\text{ in }B\(\overline{\tau_i},r_\varepsilon\)
\end{equation}
for all $\varepsilon\(\delta\)\in\(0,\varepsilon_4\(\delta\)\)$, $\(\gamma,\tau,\theta\)\in P^k_\varepsilon\(\delta\)$ and $i\in\left\{1,\dotsc,k\right\}$.
\end{proposition}

\proof
In order to prove \eqref{Pr6Eq1}, it suffices to write
$$U_{\varepsilon,\gamma,\tau,\theta}\(x\)= \overline{B}_{\varepsilon,\gamma_i,\tau_i}\(x\) + E^{\(i\)}_{\ve,\gamma,\tau}+F^{\(i\)}_{\ve,\gamma,\tau}\(x\)+\sum_{j=1}^k\theta_j B_{\ve,\gamma_j,\tau_j}\(x\)\quad \text{in }B\(\overline{\tau_i},r_\ve\)$$
and apply Proposition~\ref{PrEive} to bound
$$E^{\(i\)}_{\ve,\gamma,\tau}+\sum_{j=1}^k\theta_i B_{\ve,\gamma_j,\tau_j}=\bigO\(\frac{\delta_\ve \ln\overline\gamma_\varepsilon}{\overline\gamma_\varepsilon^3}\)$$
and Proposition~\ref{PrFive} to estimate $F^{\(i\)}_{\ve,\gamma,\tau}\(x\)$. It then follows from \eqref{Pr6Eq1} and \eqref{Pr10Eq1} that \eqref{positivity} holds true for small $\varepsilon>0$, uniformly in $\(\gamma,\tau,\theta\)\in P^k_\varepsilon\(\delta\)$.
\endproof

\section{Proof of Theorems~\ref{Th1} and \ref{Th2}}\label{Sec4}

This section is devoted to the proof of Theorems~\ref{Th1} and \ref{Th2}. We let $\Omega$, $l$, $\alpha$ and $h$ be as in Theorem~\ref{Th1}, fix $\beta>4\pi$, $\beta_0>0$ and $k\in\N^*$ such that $\beta = \beta_0 +4k\pi$ and let $\beta_\varepsilon$, $h_\varepsilon$, $w_\varepsilon$, $\lambda_\varepsilon$, $\overline\gamma_\varepsilon$, $\overline\mu_\varepsilon$, $d_\varepsilon$, $r_\varepsilon$, $\delta_\varepsilon$, $\delta_0$, $\delta_1$, $\overline\gamma_{i,\varepsilon}\(\tau\)$, $\overline{B}_{\varepsilon,\gamma_0,x_0}$, $A_{\varepsilon,\gamma,x,r}$, $C_{\varepsilon,\gamma,x,r}$, $G$, $H$, $B_{\varepsilon,\gamma_i,\tau_i}$, $\widetilde{U}_{\varepsilon,\gamma,\tau}$, $\chi_{\varepsilon,\tau}$, $\Psi_{\varepsilon,\gamma,\tau}$, $U_{\varepsilon,\gamma,\tau,\theta}$, $\Gamma^k_\varepsilon\(\delta\)$, $\overline\Gamma^k_\varepsilon\(\tau\)$, $T^k_\varepsilon\(\delta\)$, $\Theta^k_\varepsilon$ and $P^k_\varepsilon\(\delta\)$ be as in Section~\ref{Sec3}. We define
\begin{equation}\label{Sec4Eq}
R_{\varepsilon,\gamma,\tau,\theta}:=U_{\varepsilon,\gamma,\tau,\theta}-\Delta^{-1}\[\lambda_\varepsilon h_\varepsilon f\(U_{\varepsilon,\gamma,\tau,\theta}\)\].
\end{equation}
As a first step, we obtain the following:

\begin{proposition}\label{Pr7}
Let $\varepsilon_4$ be as in Proposition~\ref{Pr6}. Assume that
\begin{equation}\label{Pr7Eq0}
\tfrac{3-\sqrt5}{4}<\delta_0<\tfrac12\quad\text{and}\quad0<\delta_1<3\delta_0-2\delta_0^2-\tfrac12\,.
\end{equation}
Then for every $\delta\in\(0,1\)$, there exist $\varepsilon_5\(\delta\)\in\(0,\varepsilon_4\(\delta\)\)$ and $C_5=C_5\(\delta\)>0$ such that
\begin{equation}\label{Pr7Eq1}
\left\|R_{\varepsilon,\gamma,\tau,\theta}\right\|_{H^1_0}\le C_5\frac{\delta_\ve\ln\overline\gamma_\varepsilon}{\overline\gamma_\varepsilon^2}
\end{equation}
for all $\varepsilon\in\(0,\varepsilon_5\(\delta\)\)$ and $\(\gamma,\tau,\theta\)\in P^k_\varepsilon\(\delta\)$.
\end{proposition}

\proof
For every $\psi\in H_0^1\(\Omega\)$, using that $B_{\varepsilon,\gamma_i,\tau_i}\in V_{\varepsilon,\gamma,\tau}$, integrating by parts and using \eqref{Pr4Eq4b}, we obtain
\begin{align}\label{Pr7Eq2}
\<R_{\varepsilon,\gamma,\tau,\theta},\psi\>_{H^1_0}&=\<U_{\varepsilon,\gamma,\tau,\theta}-\Delta^{-1}\[\lambda_\varepsilon h_\varepsilon f\(U_{\varepsilon,\gamma,\tau,\theta}\)\],\psi\>_{H^1_0}\nonumber\\
&=\int_\Omega\big(\Delta U_{\varepsilon,\gamma,\tau,\theta}-\lambda_\varepsilon h_\varepsilon f\(U_{\varepsilon,\gamma,\tau,\theta}\)\big)\psi dx\allowdisplaybreaks\nonumber\\
&=\lambda_\varepsilon\int_\Omega\bigg(\sum_{i=1}^k\(1+\theta_i\)h_\varepsilon\(\overline{\tau_i}\)f\(\overline{B}_{\varepsilon,\gamma_i,\tau_i}\)\mathbf{1}_{B\(\overline{\tau_i},r_\varepsilon\)}+h_\varepsilon\chi_{\varepsilon,\tau}f\big(U_{\varepsilon,\gamma,\tau}\big)\nonumber\\
&\quad-h_\varepsilon f\(U_{\varepsilon,\gamma,\tau,\theta}\)\bigg)\psi dx+\sum_{i=1}^k\theta_i\int_\Omega\psi\Delta B_{\varepsilon,\gamma_i,\tau_i}dx.
\end{align}
By using the definition of $\chi_{\varepsilon,\tau}$ together with the mean value theorem, we obtain
\begin{align}\label{Pr7Eq3}
&\left|\sum_{i=1}^k\(1+\theta_i\)h_\varepsilon\(\overline{\tau_i}\)f\(\overline{B}_{\varepsilon,\gamma_i,\tau_i}\)\mathbf{1}_{B\(\overline{\tau_i},r_\varepsilon\)}+h_\varepsilon\chi_{\varepsilon,\tau}f\big(U_{\varepsilon,\gamma,\tau}\big)-h_\varepsilon f\(U_{\varepsilon,\gamma,\tau,\theta}\)\right|\nonumber\\
&\quad\le\sum_{i=1}^k\(\left|h_\varepsilon\(\overline{\tau_i}\)f\(\overline{B}_{\varepsilon,\gamma_i,\tau_i}\)-h_\varepsilon f\(U_{\varepsilon,\gamma,\tau,\theta}\)\right|+\left|\theta_i\right|h_\varepsilon\(\overline{\tau_i}\)f\(\overline{B}_{\varepsilon,\gamma_i,\tau_i}\)\)\mathbf{1}_{B\(\overline{\tau_i},r_\varepsilon\)}\nonumber\\
&\qquad+h_\varepsilon\sum_{i=1}^k\big|f\big(U_{\varepsilon,\gamma,\tau}\big)\big|\mathbf{1}_{A\(\overline{\tau_i},r_\varepsilon,r_\varepsilon+r_\varepsilon^2\)}+h_\varepsilon\big|f\big(U_{\varepsilon,\gamma,\tau}\big)-f\(U_{\varepsilon,\gamma,\tau,\theta}\)\big|\mathbf{1}_{\Omega_{r_\varepsilon,\tau}}\nonumber\allowdisplaybreaks\\
&\quad\le\sum_{i=1}^k\big(h_\varepsilon f'\(\(1-t_1\)\overline{B}_{\varepsilon,\gamma_i,\tau_i}+t_1U_{\varepsilon,\gamma,\tau,\theta}\)\left|U_{\varepsilon,\gamma,\tau,\theta}-\overline{B}_{\varepsilon,\gamma_i,\tau_i}\right|\nonumber\\
&\qquad\quad\quad+\left|\nabla h_\varepsilon\(\(1-t_2\)\overline{\tau_i}+t_2x\)\right|\left|x-\overline{\tau_i}\right|f\(\overline{B}_{\varepsilon,\gamma_i,\tau_i}\)\big)\mathbf{1}_{B\(\overline{\tau_i},r_\varepsilon\)}\nonumber\allowdisplaybreaks\\
&\quad\quad+h_\varepsilon\sum_{i=1}^k\big|f\big(U_{\varepsilon,\gamma,\tau}\big)\big|\mathbf{1}_{A\(\overline{\tau_i},r_\varepsilon,r_\varepsilon+r_\varepsilon^2\)}+\bigg(h_\varepsilon\(\overline{\tau_i}\)f\(\overline{B}_{\varepsilon,\gamma_i,\tau_i}\)\nonumber\\
&\quad\quad+h_\varepsilon f'\bigg(U_{\varepsilon,\gamma,\tau}+t_3\sum_{i=1}^k\theta_iB_{\varepsilon,\gamma_i,\tau_i}\bigg)\bigg)\sum_{i=1}^k\left|\theta_i\right|B_{\varepsilon,\gamma_i,\tau_i}\mathbf{1}_{\Omega_{r_\varepsilon,\tau}}
\end{align}
for some functions $t_1,t_2,t_3:\Omega\to\[0,1\]$, where $A\(\overline{\tau_i},r_\varepsilon,r_\varepsilon+r_\varepsilon^2\)$ and $\Omega_{r_\varepsilon,\tau}$ are as in \eqref{defAOmega}.  Since $\lambda_\varepsilon\to\lambda_0$ and $h_\varepsilon\to h_0$ in $C^1\(\overline\Omega\)$, it follows from \eqref{Pr7Eq2} and \eqref{Pr7Eq3} that
\begin{multline}\label{Pr7Eq4}
\<R_{\varepsilon,\gamma,\tau,\theta},\psi\>_{H^1_0}=\bigO\bigg(\sum_{i=1}^k\int_\Omega\bigg(\big(f'\(\(1-t_1\)\overline{B}_{\varepsilon,\gamma_i,\tau_i}+t_1U_{\varepsilon,\gamma,\tau,\theta}\)\left|U_{\varepsilon,\gamma,\tau,\theta}-\overline{B}_{\varepsilon,\gamma_i,\tau_i}\right|\\
+\left|x-\overline{\tau_i}\right|f\(\overline{B}_{\varepsilon,\gamma_i,\tau_i}\)\big)\mathbf{1}_{B\(\overline{\tau_i},r_\varepsilon\)} +\big|f\big(U_{\varepsilon,\gamma,\tau}\big)\big|\mathbf{1}_{A\(\overline{\tau_i},r_\varepsilon,r_\varepsilon+r_\varepsilon^2\)}\\
+f'\bigg(U_{\varepsilon,\gamma,\tau}+t_4\sum_{j=1}^k\theta_jB_{\varepsilon,\gamma_j,\tau_j}\bigg)\left|\theta_i\right|B_{\varepsilon,\gamma_i,\tau_i}\mathbf{1}_{\Omega_{r_\varepsilon,\tau}}\bigg) |\psi|dx\bigg).
\end{multline}
For every $i\in\left\{1,\dotsc,k\right\}$, by using \eqref{Pr6Eq1} and remarking that  $f'\(u\)\le3uf\(u\)$ for all $u\ge1$, we obtain
\begin{multline}\label{Pr7Eq5}
\int_{B\(\overline{\tau_i},r_\varepsilon\)}\(f'\(\(1-t_1\)\overline{B}_{\varepsilon,\gamma_i,\tau_i}+t_1U_{\varepsilon,\gamma,\tau,\theta}\)\left|U_{\varepsilon,\gamma,\tau,\theta}-\overline{B}_{\varepsilon,\gamma_i,\tau_i}\right|+\left|x-\overline{\tau_i}\right|f\(\overline{B}_{\varepsilon,\gamma_i,\tau_i}\)\)\\
\times|\psi|dx=\bigO\bigg(\int_{B\(\overline{\tau_i},r_\varepsilon\)}f\(\overline{B}_{\varepsilon,\gamma_i,\tau_i}\)\bigg(\frac{\delta_\varepsilon\ln\overline\gamma_\varepsilon}{\overline\gamma_\varepsilon^2}+\overline{\gamma}_\varepsilon^{1/l}\left|x-\overline{\tau_i}\right|\bigg)|\psi |dx\bigg).
\end{multline}
By integrating by parts, we obtain
\begin{align}\label{Pr7Eq6}
\int_{B\(\overline{\tau_i},r_\varepsilon\)}f\(\overline{B}_{\varepsilon,\gamma_i,\tau_i}\)\left|\psi\right|dx&=\(\lambda_\varepsilon h_\varepsilon\(\overline{\tau_i}\)\)^{-1}\<B_{\varepsilon,\gamma_i,\tau_i},\left| \psi \right|\>_{H_0^1}\nonumber\\
&\le\(\lambda_\varepsilon h_\varepsilon\(\overline{\tau_i}\)\)^{-1}\left\|B_{\varepsilon,\gamma_i,\tau_i}\right\|_{H^1_0}\left\|\psi \right\|_{H^1_0}=\bigO\big(\left\|\psi\right\|_{H^1_0}\big).
\end{align}
On the other hand, for every $p>1$, by using H\"older's inequality together with the continuity of the embedding $H^1_0\(\Omega\)\hookrightarrow L^{p'}\(\Omega\)$, where $p'$ is the conjugate exponent of $p$, we obtain
\begin{align}
&\int_{B\(\overline{\tau_i},r_\varepsilon\)}f\(\overline{B}_{\varepsilon,\gamma_i,\tau_i}\)\left|x-\overline{\tau_i}\right|\left|\psi \right|dx=\bigO\big(\big\|f\(\overline{B}_{\varepsilon,\gamma_i,\tau_i}\)\left|x-\overline{\tau_i}\right|\mathbf{1}_{B\(\overline{\tau_i},r_\varepsilon\)}\big\|_{L^p}\left\|\psi\right\|_{H^1_0}\big),\label{Pr7Eq7}\allowdisplaybreaks\\ 
&\int_{A\(\overline{\tau_i},r_\varepsilon,r_\varepsilon+r_\varepsilon^2\)}\big|f\big(U_{\varepsilon,\gamma,\tau}\big)\big|\left|\psi \right|dx=\bigO\big(\big\|f\big(U_{\varepsilon,\gamma,\tau}\big)\mathbf{1}_{A\(\overline{\tau_i},r_\varepsilon,r_\varepsilon+r_\varepsilon^2\)}\big\|_{L^p}\left\|\psi\right\|_{H^1_0}\big),\label{Pr7Eq8}\allowdisplaybreaks\\ 
&\int_{\Omega_{r_\varepsilon,\tau}}\bigg(f\(\overline{B}_{\varepsilon,\gamma_i,\tau_i}\)+f'\bigg(U_{\varepsilon,\gamma,\tau}+t_4\sum_{j=1}^k\theta_jB_{\varepsilon,\gamma_j,\tau_j}\bigg)B_{\varepsilon,\gamma_i,\tau_i}\bigg)\left|\psi \right|dx\nonumber\\
&\quad=\bigO\bigg(\bigg\|\bigg(f\(\overline{B}_{\varepsilon,\gamma_i,\tau_i}\)+f'\bigg(U_{\varepsilon,\gamma,\tau}+t_4\sum_{j=1}^k\theta_jB_{\varepsilon,\gamma_j,\tau_j}\bigg)B_{\varepsilon,\gamma_i,\tau_i}\bigg)\mathbf{1}_{\Omega_{r_\varepsilon,\tau}}\bigg\|_{L^p}\left\|\psi\right\|_{H^1_0}\bigg).\label{Pr7Eq9}
\end{align}
By rescaling, we obtain
$$\big\|f\(\overline{B}_{\varepsilon,\gamma_i,\tau_i}\)\left|x-\overline{\tau_i}\right|\mathbf{1}_{B\(\overline{\tau_i},r_\varepsilon\)}\big\|_{L^p}^p=\mu_{i,\varepsilon}^{p+2}\int_{B\(0,r_\varepsilon/\mu_{i,\varepsilon}\)}f\(\overline{B}_{\varepsilon,\gamma_i,\tau_i}\(\overline{\tau_i}+\mu_{i,\varepsilon}x\)\)^p\left|x\right|^pdx,$$
where $\mu_{i,\varepsilon}$ is defined by $\mu_{i,\varepsilon}^2:=4\gamma_{i,\varepsilon}^{-2}\exp\(-\gamma_{i,\varepsilon}^2\)$. By using \eqref{Sec32Eq3} and \eqref{Pr10Eq1}, it follows that
\begin{align}\label{Pr7Eq10}
&\big\|f\(\overline{B}_{\varepsilon,\gamma_i,\tau_i}\)\left|x-\overline{\tau_i}\right|\mathbf{1}_{B\(\overline{\tau_i},r_\varepsilon\)}\big\|_{L^p}^p\nonumber\\
&=\bigO\(\mu_{i,\varepsilon}^{p+2}\int_{B\(0,r_\varepsilon/\mu_{i,\varepsilon}\)}f\(\gamma_{i,\varepsilon}-\frac{1}{\gamma_{i,\varepsilon}}\ln\frac{1}{1+\lambda_\varepsilon h_\varepsilon\(\overline{\tau_i}\)\left|x\right|^2}\)^p\left|x\right|^pdx\)\nonumber\allowdisplaybreaks\\
&=\bigO\(\frac{\mu_{i,\varepsilon}^{2-p}}{\overline{\gamma}_\varepsilon^p}\int_{B\(0,r_\varepsilon/\mu_{i,\varepsilon}\)}\frac{\left|x\right|^pdx}{\big(1+\lambda_\varepsilon h_\varepsilon\(\overline{\tau_i}\)\left|x\right|^2\big)^{2p}}\)=\bigO\(\frac{\mu_{i,\varepsilon}^{2-p}}{\overline{\gamma}_\varepsilon^p}\)=\smallo\(\(\frac{\delta_\varepsilon\ln\overline\gamma_\varepsilon}{\overline\gamma_\varepsilon^{2+1/l}}\)^p\)
\end{align}
provided we choose $p$ such that $2-p>p\(\delta_1+1/2\)$, i.e. $1<p<4/\(2\delta_1+3\)$, which is possible since $\delta_1<3\delta_0-2\delta_0^2-1/2<1/2$. As regards the terms in the right-hand sides of \eqref{Pr7Eq8} and \eqref{Pr7Eq9}, by using \eqref{Sec32Eq3} and \eqref{Pr4Eq14} and proceeding as in \eqref{Pr4Eq15}--\eqref{Pr4Eq18} and \eqref{Pr4Eq32}, we obtain
\begin{align}
&\big\|f\big(U_{\varepsilon,\gamma,\tau}\big)\mathbf{1}_{A\(\overline{\tau_i},r_\varepsilon,r_\varepsilon+r_\varepsilon^2\)}\big\|^p_{L^p}=\bigO\(\overline\gamma_\varepsilon^p\ln\(1+r_\varepsilon\)\exp\(\(p\delta_0-1\)\delta_0\overline\gamma_\varepsilon^2+\smallo\(\overline\gamma_\varepsilon\)\)\)\nonumber\\
&\qquad=\bigO\(\overline\gamma_\varepsilon^p\exp\(\(p\delta_0-\frac{3}{2}\)\delta_0\overline\gamma_\varepsilon^2+\smallo\(\overline\gamma_\varepsilon^2\)\)\)=\smallo\(\(\frac{\delta_\varepsilon\ln\overline\gamma_\varepsilon}{\overline\gamma_\varepsilon^2}\)^p\),\label{Pr7Eq11}\allowdisplaybreaks\\
&\bigg\|\bigg(f\(\overline{B}_{\varepsilon,\gamma_i,\tau_i}\)+f'\bigg(U_{\varepsilon,\gamma,\tau}+t_4\sum_{j=1}^k\theta_jB_{\varepsilon,\gamma_j,\tau_j}\bigg)B_{\varepsilon,\gamma_i,\tau_i}\bigg)\mathbf{1}_{\Omega_{r_\varepsilon,\tau}}\bigg\|_{L^p}^p\nonumber\\
&\qquad=\bigO\Bigg(\overline\gamma_\varepsilon^{3p+2}\exp\(\(p\delta_0-1\)\delta_0\overline\gamma_\varepsilon^2+\smallo\(\overline\gamma_\varepsilon^2\)\)+\frac{1}{\overline\gamma_\varepsilon^p}\int_{\Omega_{R_\varepsilon,\tau}}\left|\ln\left|x-\overline{\tau_i}\right|+\bigO\(1\)\right|^pdx\Bigg)\nonumber\\
&\qquad=\smallo\(1\)\label{Pr7Eq12}
\end{align}
as $\varepsilon\to0$, uniformly in $\(\gamma,\tau,\theta\)\in P^k_\varepsilon\(\delta\)$, provided we choose $p$ such that 
\begin{multline*}
\(p\delta_0-\frac{3}{2}\)\delta_0<-\frac{p}{2}\(\delta_1+\frac{1}{2}\)\quad\text{and}\quad p\delta_0-1<0,\\
\text{i.e. }1<p<\min\(\frac{3\delta_0}{2\delta_0^2+\delta_1+1/2},\frac{1}{\delta_0}\)=\frac{3\delta_0}{2\delta_0^2+\delta_1+1/2}\,,
\end{multline*}
which is possible when assuming \eqref{Pr7Eq0}. By putting together \eqref{Pr7Eq4}--\eqref{Pr7Eq12} and using the fact that $\left|\theta_i\right|<\delta_\varepsilon\overline\gamma_\varepsilon^{-4}\ln\overline\gamma_\varepsilon$, we obtain \eqref{Pr7Eq1}. This ends the proof of Proposition~\ref{Pr7}.
\endproof

We let $\mathcal{H}$ be the vector space of all functions in $H^1_0\(\Omega\)$ that are even in $x_2$. For every $\tau\in T^k_\varepsilon\(\delta\)$ and $\gamma\in\overline\Gamma^k_\varepsilon\(\tau\)$, we define
$${V_{\varepsilon,\gamma,\tau}}:=\text{span}\left\{Z_{0,i,\varepsilon,\gamma,\tau},Z_{1,i,\varepsilon,\gamma,\tau},B_{\varepsilon,\gamma_i,\tau_i}\right\}_{1\le i\le k},$$
where
$$Z_{0,i,\varepsilon,\gamma,\tau}:=\partial_{\gamma_i}\[U_{\varepsilon,\gamma,\tau,0}\]\quad\text{and}\quad Z_{1,i,\varepsilon,\gamma,\tau}:=\partial_{\tau_i}\[U_{\varepsilon,\gamma,\tau,0}\]\quad\forall i\in\left\{1,\dotsc,k\right\}.$$
Note that $U_{\varepsilon,\gamma,\tau,0}\in\mathcal{H}$ and $V_{\varepsilon,\gamma,\tau}\subset\mathcal{H}$. We let $\Pi_{\varepsilon,\gamma,\tau}$ and $\Pi_{\varepsilon,\gamma,\tau}^\perp$ be the orthogonal projection of $\mathcal{H}$ onto $V_{\varepsilon,\gamma,\tau}$ and $V_{\varepsilon,\gamma,\tau}^\perp$, respectively. We obtain the following:

\begin{proposition}\label{Pr8}
Assume that \eqref{Pr7Eq0} holds true. Let $\varepsilon_5$ be as in Proposition~\ref{Pr7}. Then for every $\delta\in\(0,1\)$, there exist $\varepsilon_6\(\delta\)\in\(0,\varepsilon_5\(\delta\)\)$ and $C_6=C_6\(\delta\)>0$ such that for every $\varepsilon\in\(0,\varepsilon_6\(\delta\)\)$ and $\(\gamma,\tau,\theta\)\in P^k_\varepsilon\(\delta\)$, there exists a unique solution $\Phi_{\varepsilon,\gamma,\tau,\theta}\in V_{\varepsilon,\gamma,\tau}^\perp$ to the equation
\begin{equation}\label{Pr8Eq1}
\Pi_{\varepsilon,\gamma,\tau}^\perp\(U_{\varepsilon,\gamma,\tau,\theta}+\Phi_{\varepsilon,\gamma,\tau,\theta}-\Delta^{-1}\[\lambda_\varepsilon h_\varepsilon f\(U_{\varepsilon,\gamma,\tau,\theta}+\Phi_{\varepsilon,\gamma,\tau,\theta}\)\]\)=0
\end{equation}
such that
\begin{equation}\label{Pr8Eq2}
\left\|\Phi_{\varepsilon,\gamma,\tau,\theta}\right\|_{H^1_0}\le C_6\frac{\delta_\varepsilon\ln\overline\gamma_\varepsilon}{\overline\gamma_\varepsilon^2}.
\end{equation}
Furthermore, $\Phi_{\varepsilon,\gamma,\tau,\theta}$ is continuous in $\(\gamma,\tau,\theta\)$.
\end{proposition}

The proof of Proposition~\ref{Pr8} relies on the following:

\begin{lemma}\label{Pr8Lem}
Let $\varepsilon_5$ be as in Proposition~\ref{Pr7}. For every $\delta\in\(0,1\)$, there exist $\varepsilon_5'\(\delta\)\in\(0,\varepsilon_5\(\delta\)\)$ and $C_5'=C_5'\(\delta\)>0$ such that for every $\varepsilon\in\(0,\varepsilon_5'\(\delta\)\)$ and $\(\gamma,\tau,\theta\)\in P^k_\varepsilon\(\delta\)$, the operator $L_{\varepsilon,\gamma,\tau,\theta}:V_{\varepsilon,\gamma,\tau}^\perp\to V_{\varepsilon,\gamma,\tau}^\perp$ defined by
\begin{equation}\label{Pr8LemEq0}
L_{\varepsilon,\gamma,\tau,\theta}\(\Phi\)=\Pi_{\varepsilon,\gamma,\tau}^\perp\(\Phi-\Delta^{-1}\[\lambda_\varepsilon h_\varepsilon f'\(U_{\varepsilon,\gamma,\tau,\theta}\)\Phi\]\)\quad\forall\Phi\in V_{\varepsilon,\gamma,\tau}^\perp
\end{equation}
satisfies
\begin{equation}\label{Pr8LemEq1}
\left\|\Phi\right\|_{H^1_0}\le C'_5\left\|L_{\varepsilon,\gamma,\tau,\theta}\(\Phi\)\right\|_{H^1_0}.
\end{equation}
In particular, $L_{\varepsilon,\gamma,\tau,\theta}$ is an isomorphism.
\end{lemma}

Proposition~\ref{Pr8} and Lemma~\ref{Pr8Lem} (together with Proposition~\ref{Pr9} and Lemma~\ref{Pr9Lem}), are the heart of the Lyapunov--Schmidt procedure. We prove them by using a similar approach as in the case of higher dimensions (see for instance Deng--Musso--Wei~\cite{DengMusWei} and Robert--V\'etois~\cites{RobVet1,RobVet2}). Aside from the usual differences in the computations due to the exponential term, the main difference here lies in the use of the Poincar\'e--Sobolev inequalities \eqref{PSB} and \eqref{Pr14Eq2}, which take advantage of the additional dimensions of the kernel $V_{\varepsilon,\gamma,\tau}$ given by the directions of the bubbles. 

\proof[Proof of Lemma~\ref{Pr8Lem}]
We proceed by contradiction. We assume that there exist sequences $\(\varepsilon_n,\gamma_n,\tau_n,\theta_n,\Phi_n\)_{n\in\N^\ast}$ such that $\varepsilon_n\to0$, $\(\gamma_n,\tau_n,\theta_n\)\in P^k_\varepsilon\(\delta\)$ and 
\begin{equation}\label{Pr8LemEq2}
\Phi_n\in V_{\varepsilon_n,\gamma_n,\tau_n}^\perp,\quad\left\|\Phi_n\right\|_{H^1_0}=1\quad\text{and}\quad\left\|L_{\varepsilon_n,\gamma_n,\tau_n,\theta_n}\(\Phi_n\)\right\|_{H^1_0}=\smallo\(1\)
\end{equation}
as $n\to\infty$. For simplicity of notations, we denote $\overline\gamma_n:=\overline\gamma_{\varepsilon_n}$, $r_n:=r_{\varepsilon_n}$, $d_n:=d_{\varepsilon_n}$, $\lambda_n:=\lambda_{\varepsilon_n}$, $h_n:=h_{\varepsilon_n}$, $w_n:=w_{\varepsilon_n}$, $\Psi_n:=\Psi_{\varepsilon_n,\gamma_n,\tau_n}$, $U_n:=U_{\varepsilon_n,\gamma_n,\tau_n,\theta_n}$, $\overline{B}_{i,n}:=\overline{B}_{\varepsilon_n,\gamma_{i,n},\overline{\tau_{i,n}}}$, $B_{i,n}:=B_{\varepsilon_n,\gamma_{i,n},\tau_{i,n}}$, $L_n:=L_{\varepsilon_n,\gamma_n,\tau_n,\theta_n}$, $V_n^\perp:=V_{\varepsilon_n,\gamma_n,\tau_n}^\perp$ and $Z_{j,i,n}:=Z_{j,i,\varepsilon_n,\gamma_n,\tau_n}$ for all $i\in\left\{1,\dotsc,k\right\}$ and $j\in\left\{0,1\right\}$, where $\gamma_n:=\(\gamma_{1,n},\dotsc,\gamma_{k,n}\)$, $\tau_n:=\(\tau_{1,n},\dotsc,\tau_{k,n}\)$, $\overline{\tau_{i,n}}:=\(\tau_{i,n},0\)$ and $\theta_n:=\(\theta_{1,n},\dotsc,\theta_{k,n}\)$. It follows from \eqref{Pr8LemEq2} that 
\begin{equation}\label{Pr8LemEq3}
\lambda_n\int_\Omega h_n f'\(U_n\)\Phi_n^2dx=\left\|\Phi_n\right\|_{H^1_0}^2-\left<\Phi_n,L_n\(\Phi_n\)\right>_{H^1_0}=1+\smallo\(1\)
\end{equation}
as $n\to\infty$. On the other hand, since $f'>0$, $\lambda_n\to\lambda_0$ and $h_n\to h_0$ in $C^0\(\overline\Omega\)$, we obtain
\begin{equation}\label{Pr8LemEq4}
\lambda_n\int_\Omega h_n f'\(U_n\)\Phi_n^2dx=\bigO\(I_n\),\quad\text{where }I_n:=\int_\Omega f'\(U_n\)\Phi_n^2dx.
\end{equation}
In what follows, we will prove that $I_n\to0$ as $n\to\infty$, thus contradicting \eqref{Pr8LemEq3} and \eqref{Pr8LemEq4}. 

\smallskip\noindent
{\it Estimation of $I_n$ in the balls $B\(\overline{\tau_{i,n}},r_n\)$.} For $i\in\left\{1,\dotsc,k\right\}$, by rescaling and using \eqref{Pr8LemEq2}, we obtain
\begin{align}
&\int_{B\(\overline{\tau_{i,n}},r_n\)}f'\(U_n\)\Phi_n^2dx=\mu_{i,n}^2\int_{B\(0,r_n/\mu_{i,n}\)}f'\big(\gamma_{i,n}^{-1}\widehat{U}_n+\gamma_{i,n}\big)\widehat\Phi_n^2dx,\label{Pr8LemEq5}\allowdisplaybreaks\\
&\int_{(\Omega-\overline{\tau_{i,n}})/\mu_{i,n}}\big<\nabla\widehat\Phi_n,\nabla\psi\big>dx-\lambda_n\mu_{i,n}^2\int_{(\Omega-\overline{\tau_{i,n}})/\mu_{i,n}}\hat{h}_nf'\big(\gamma_{i,n}^{-1}\widehat{U}_n+\gamma_{i,n}\big)\widehat\Phi_n\psi dx\nonumber\\
&\hspace{214pt}=\smallo\(\left\|\nabla\psi\right\|_{L^2}\)\qquad\forall\psi\in C^\infty_c\(\R^2\)\label{Pr8LemEq6}
\end{align}
as $n\to\infty$, where $\mu_{i,n}$, $\hat{h}_n$, $\widehat\Phi_n$ and $\widehat{U}_n$ are defined by
\begin{align*}
\mu_{i,n}^2&:=4\gamma_{i,n}^{-2}\exp\(-\gamma_{i,n}^2\),\quad\hat{h}_n\(x\):=h_n\(\overline{\tau_{i,n}}+\mu_{i,n}x\),\allowdisplaybreaks\\
\widehat\Phi_n\(x\)&:=\Phi_n\(\overline{\tau_{i,n}}+\mu_{i,n}x\)\quad\text{and}\quad\widehat{U}_n\(x\):=\gamma_{i,n}\(U_n\(\overline{\tau_{i,n}}+\mu_{i,n}x\)-\gamma_{i,n}\)
\end{align*}
for all $x\in (\Omega-\overline{\tau_{i,n}})/n$. By using \eqref{Sec32Eq3}, \eqref{Pr6Eq1} and \eqref{Pr10Eq1}, we obtain
\begin{equation}\label{Pr8LemEq7}
\widehat{U}_n\(x\)\sim\gamma_{i,n}\(\overline{B}_{i,n}\(\overline{\tau_{i,n}}+\mu_{i,n}x\)-\gamma_{i,n}\)\sim\ln\frac{1}{1+\lambda_n h_n\(\overline{\tau_{i,n}}\)\left|x\right|^2}
\end{equation}
as $n\to\infty$, uniformly in $x\in B\(0,r_n/\mu_{i,n}\)$. By using \eqref{Pr8LemEq7} together with the definition of $\mu_{i,n}$, we obtain
\begin{equation}\label{Pr8LemEq8}
\mu_{i,n}^2f'\big(\gamma_{i,n}^{-1}\widehat{U}_n+\gamma_{i,n}\big)\sim\frac{8}{\big(1+\lambda_n h_n\(\overline{\tau_{i,n}}\)\left|x\right|^2\big)^2}
\end{equation}
as $n\to\infty$, uniformly in $x\in B\(0,r_n/\mu_{i,n}\)$. By remarking that 
$$\big\|\nabla\widehat\Phi_n\big\|_{L^2}=\left\|\nabla\Phi_n\right\|_{L^2}=1$$ 
and using \eqref{Pr8LemEq6} and \eqref{Pr8LemEq8}, we obtain that $\big(\widehat\Phi_n\big)_{n}$ converges, up to a subsequence, weakly in $D^{1,2}\(\R^2\)$, strongly in $L^p_{\loc}\(\R^2\)$ for all $p\ge1$ and pointwise almost everywhere in $\R^2$ to a solution $\widehat\Phi_0$ of the equation
\begin{equation}\label{EqPhi0}
\Delta\widehat\Phi_0=\frac{8\lambda_0 h_0\(0\)\widehat\Phi_0}{\big(1+\lambda_0 h_0\(0\)\left|x\right|^2\big)^2}\qquad\text{in }\R^2.
\end{equation}
Furthermore, since $\Phi_n\in\mathcal{H}$, we obtain that $\widehat\Phi_0$ is even in $x_2$. By using a result of Baraket--Pacard~\cite{BarPac}, it follows that $\widehat\Phi_0\in\text{span}\left\{Z_0,Z_1\right\}$,
where
$$Z_0\(x\):=\frac{1-\lambda_0 h_0\(0\)\left|x\right|^2}{1+\lambda_0 h_0\(0\)\left|x\right|^2}\quad\text{and}\quad Z_1\(x\):=\frac{2\lambda_0 h_0\(0\)x_1}{1+\lambda_0 h_0\(0\)\left|x\right|^2}\quad\forall x\in\R^2\,.$$
In particular, note that the Poincar\'e--Sobolev inequality \eqref{PSB} applies to $\widehat\Phi_0$. On the other hand, for every $i\in\left\{1,\dotsc,k\right\}$, since $\Phi_n\in E_{\varepsilon_n,\gamma_n,\tau_n}^\perp$, we get $\<B_{i,n},\Phi_n\>_{H^1_0}=\<Z_{0,i,n},\Phi_n\>_{H^1_0}=\<Z_{1,i,n},\Phi_n\>_{H^1_0}=0$, which, by integrating by parts and using the equations satisfied by $B_{i,n}$, $Z_{0,i,n}$ and $Z_{1,i,n}$, gives
\begin{align}
&\int_{B\(\overline{\tau_{i,n}},r_n\)}f\(\overline{B}_{i,n}\)\Phi_ndx=0,\label{Pr8LemEq9_0}\allowdisplaybreaks\\
\lambda_nh_n\(\overline{\tau_{i,n}}\)&\int_{B\(\overline{\tau_{i,n}},r_n\)}\hspace{-2pt}f'\(\overline{B}_{i,n}\)\partial_{\gamma_i}\[\overline{B}_{\varepsilon_n,\gamma,\overline{\tau_{i,n}}}\]_{\gamma=\gamma_{i,n}}\Phi_ndx+\<\Phi_n,\partial_{\gamma_i}\[\Psi_{\varepsilon_n,\gamma,\tau_n,0}\]_{\gamma=\gamma_n}\>_{H^1_0}=0\label{Pr8LemEq9}
\end{align}
together with an analogous estimate for the derivative in $\tau_i$. It follows from \eqref{Pr8LemEq2}, \eqref{Pr8LemEq9_0} and the Poincar\'e--Sobolev inequality \eqref{Pr14Eq2} that
\begin{equation}\label{Pr8LemEq10}
\int_{B\(\overline{\tau_{i,n}},r_n\)}f'\(\overline{B}_{i,n}\)\Phi_n^2dx=\bigO\Big(\left\|\nabla\Phi_n\right\|_{L^2}^2\Big)=\bigO\(1\).
\end{equation}
On the other hand, by using Cauchy--Schwartz' inequality together with \eqref{Pr4Eq2}, \eqref{Pr4Eq3} and \eqref{Pr8LemEq2}, we obtain
\begin{equation}\label{Pr8LemEq11}
\<\Phi_n,\partial_{\gamma_i}\[\Psi_{\varepsilon_n,\gamma,\tau_n,0}\]_{\gamma=\gamma_n}\>_{H^1_0}=\smallo\(1\)\text{ and }\<\Phi_n,\partial_{\tau_i}\[\Psi_{\varepsilon_n,\gamma_n,\tau,0}\]_{\tau=\tau_n}\>_{H^1_0}=\smallo\(1\)
\end{equation}
as $n\to\infty$. By rescaling, it follows from \eqref{Pr8LemEq9} and \eqref{Pr8LemEq10} that
\begin{align}
&\mu_{i,n}^2\int_{B\(0,r_n/\mu_{i,n}\)}f'\(\overline{B}_{i,n}\(\overline{\tau_{i,n}}+\mu_{i,n}x\)\)\widehat\Phi_n\(x\)^2dx=\bigO\(1\),\label{Pr8LemEq12}\allowdisplaybreaks\\
&\<\Phi_n,\partial_{\gamma_i}\[\Psi_{\varepsilon_n,\gamma,\tau_n,0}\]_{\gamma=\gamma_n}\>_{H^1_0}=-\lambda_nh_n\(\overline{\tau_{i,n}}\)\mu_{i,n}^2\int_{B(0,r_n/\mu_{i,n})}f'\(\overline{B}_{i,n}\(\overline{\tau_{i,n}}+\mu_{i,n}x\)\)\nonumber\\
&\hspace{162pt}\times\partial_{\gamma_i}\[\overline{B}_{\varepsilon_n,\gamma,\overline{\tau_{i,n}}}\(\overline{\tau_{i,n}}+\mu_{i,n}x\)\]_{\gamma=\gamma_n}\widehat\Phi_n\(x\)dx.\label{Pr8LemEq13}
\end{align}
Here again, we obtain an analogous estimate to \eqref{Pr8LemEq13} for the derivative in $\tau_i$.
By using \eqref{Pr10Eq1} and \eqref{Pr11Eq1} together with the definition of $\mu_{i,n}$, we obtain
\begin{align}
\partial_{\gamma_i}\[\overline{B}_{\varepsilon_n,\gamma,\tau_n}\(\overline{\tau_{i,n}}+\mu_{i,n}x\)\]_{\gamma=\gamma_n}&\longrightarrow Z_0\(x\)\text{ for a.e. }x\in\R^2,\label{Pr8LemEq15}\allowdisplaybreaks\\
\partial_{\gamma_i}\[\overline{B}_{\varepsilon_n,\gamma,\tau_n}\(\overline{\tau_{i,n}}+\mu_{i,n}x\)\]_{\gamma=\gamma_n}&=\bigO\(1\),\label{Pr8LemEq15bis}\\
\partial_{\tau_i}\[\overline{B}_{\varepsilon_n,\gamma_{i,n},\(\tau,0\)}\(\overline{\tau_{i,n}}+\mu_{i,n}x\)\]_{\tau=\tau_{i,n}}&=\frac{Z_1\(x\)}{\mu_{i,n}\gamma_{i,n}}+\smallo\(\frac{1}{\mu_{i,n}\gamma_{i,n}}\)\label{Pr8LemEq16}
\end{align}
as $n\to\infty$, uniformly in $x\in B\(0,r_n/\mu_{i,n}\)$. For every $R>0$, since  $\big(\widehat\Phi_n\big)_{n}$ converges strongly to $\widehat\Phi_0$ in $L^1_{\loc}\(\R^2\)$, it follows from \eqref{Pr8LemEq7}, \eqref{Pr8LemEq8}  and \eqref{Pr8LemEq15} that 
\begin{multline}\label{Pr8LemEq17a}
\mu_{i,n}^2\int_{B\(0,R\)}f'\(\overline{B}_{i,n}\(\overline{\tau_{i,n}}+\mu_{i,n}x\)\)\partial_{\gamma_i}\[\overline{B}_{\varepsilon_n,\gamma,\overline{\tau_{i,n}}}\(\overline{\tau_{i,n}}+\mu_{i,n}x\)\]_{\gamma=\gamma_n}\widehat\Phi_n\(x\)dx\\
\longrightarrow8\int_{B\(0,R\)}\frac{8Z_0\(x\)\widehat\Phi_0\(x\)dx}{\big(1+\lambda_0 h_0\(0\)\left|x\right|^2\big)^2}
\end{multline}
as $n\to\infty$. On the other hand, by using H\"older's inequality together with \eqref{Pr8LemEq7}, \eqref{Pr8LemEq8}, \eqref{Pr8LemEq12}, \eqref{Pr8LemEq15bis} and \eqref{PSB}, we obtain
\begin{align}
&\Bigg|\mu_{i,n}^2\int_{A\(0,R,r_n/\mu_{i,n}\)}f'\(\overline{B}_{i,n}\(\overline{\tau_{i,n}}+\mu_{i,n}x\)\)\partial_{\gamma_i}\[\overline{B}_{\varepsilon_n,\gamma,\overline{\tau_{i,n}}}\(\overline{\tau_{i,n}}+\mu_{i,n}x\)\]_{\gamma=\gamma_n}\widehat\Phi_n\(x\)dx\Bigg|\nonumber\\
&\qquad=\bigO\(\(\mu_{i,n}^2\int_{A\(0,R,r_n/\mu_{i,n}\)}f'\(\overline{B}_{i,n}\(\overline{\tau_{i,n}}+\mu_{i,n}x\)\)dx\)^{1/2}\)\nonumber\\
&\qquad=\bigO\(\(\int_{B\(0,R\)^c}\frac{dx}{\big(1+\lambda_n h_n\(\overline{\tau_{i,n}}\)\left|x\right|^2\big)^2}
\)^{1/2}\)=\smallo_R\(1\),\label{Pr8LemEq17b}\allowdisplaybreaks\\
&\left|\int_{B\(0,R\)^c}\hspace{-3pt}\frac{Z_0\(x\)\widehat\Phi_0\(x\)dx}{\big(1+\lambda_0 h_0\(0\)\left|x\right|^2\big)^2}\right|=\bigO\(\(\int_{B\(0,R\)^c}\hspace{-3pt}\frac{dx}{\big(1+\lambda_0 h_0\(0\)\left|x\right|^2\big)^2}\)^{1/2}\)=\smallo_R\(1\),\label{Pr8LemEq17c}
\end{align}
where $\smallo_R\(1\)\to0$ as $R\to\infty$, uniformly in $n\in\N^\ast$, where $A\(0,R,r_n/\mu_{i,n}\)$ is as in \eqref{defAOmega}. It follows from \eqref{Pr8LemEq11}, \eqref{Pr8LemEq13} and \eqref{Pr8LemEq17a}--\eqref{Pr8LemEq17c} that
\begin{equation}\label{Pr8LemEq17d}
\int_{\R^2}\frac{Z_0\(x\)\widehat\Phi_0\(x\)dx}{\big(1+\lambda_0 h_0\(0\)\left|x\right|^2\big)^2}=0.
\end{equation}
By proceeding in the same way but using \eqref{Pr8LemEq16} instead of \eqref{Pr8LemEq15}--\eqref{Pr8LemEq15bis}, we obtain
\begin{equation}\label{Pr8LemEq17e}
\int_{\R^2}\frac{Z_1\(x\)\widehat\Phi_0\(x\)dx}{\big(1+\lambda_0 h_0\(0\)\left|x\right|^2\big)^2}=0.
\end{equation}
Since $\widehat\Phi_0\in\text{span}\left\{Z_0,Z_1\right\}$, it follows from \eqref{EqPhi0}, \eqref{Pr8LemEq17d} and \eqref{Pr8LemEq17e} that $\widehat\Phi_0\equiv0$. For every $R>0$, by using \eqref{Pr8LemEq8} and since  $\big(\widehat\Phi_n\big)_{n}$ converges strongly to $\widehat\Phi_0$ in $L^2_{\loc}\(\R^2\)$, we then obtain
\begin{equation}\label{Pr8LemEq17f}
\mu_{i,n}^2\int_{B\(0,R\)}f'\big(\gamma_{i,n}^{-1}\widehat{U}_n+\gamma_{i,n}\big)\Phi_n^2dx=\smallo\(1\)
\end{equation}
as $n\to\infty$. On the other hand, by proceeding as in \eqref{Pr8LemEq17b}, we obtain
\begin{equation}\label{Pr8LemEq17g}
\mu_{i,n}^2\int_{A\(0,R,r_n/\mu_{i,n}\)}f'\big(\gamma_{i,n}^{-1}\widehat{U}_n+\gamma_{i,n}\big)\Phi_n^2dx=\smallo_R\(1\),
\end{equation}
where $\smallo_R\(1\)\to0$ as $R\to\infty$, uniformly in $n\in\N^\ast$. It follows from \eqref{Pr8LemEq5}, \eqref{Pr8LemEq17f} and \eqref{Pr8LemEq17g} that
\begin{equation}\label{Pr8LemEq18}
\int_{B\(\overline{\tau_{i,n}},r_n\)}f'\(U_n\)\Phi_n^2dx=\smallo\(1\)
\end{equation}
as $n\to\infty$.

\smallskip\noindent
{\it Estimation of $I_n$ in the annuli $A\(\overline{\tau_{i,n}},r_n,R_n\)$, where $R_n:=\exp\(-\overline\gamma_n\)$.} For every $i\in\left\{1,\dotsc,k\right\}$, for small $p>1$, by using H\"older's inequality, \eqref{Pr4Eq3b} and \eqref{Pr8LemEq2} together with the continuity of the embedding $H^1_0\(\Omega\)\hookrightarrow L^{2p'}\(\Omega\)$, we obtain
\begin{align}\label{Pr8LemEq19}
\int_{A\(\overline{\tau_{i,n}},r_n,R_n\)}f'\(U_n\)\Phi_n^2dx&=\bigO\(\left\|f'\(U_n\)\mathbf{1}_{A\(\overline{\tau_{i,n}},r_n,R_n\)}\right\|_{L^p}\left\|\Phi_n\right\|^2_{H^1_0}\)\nonumber\\
&=\bigO\(\left\|f'\(U_n\)\mathbf{1}_{A\(\overline{\tau_{i,n}},r_n,R_n\)}\right\|_{L^p}\)=\smallo\(1\)
\end{align}
as $n\to\infty$.

\smallskip\noindent
{\it Estimation of $I_n$ in $\Omega_{R_n,\tau_n}$.} By using \eqref{Pr8LemEq2}, we obtain that $\(\Phi_n\)_{n}$ converges, up to a subsequence, weakly in $H^1_0\(\Omega\)$ and pointwise almost everywhere in $\Omega$ to a function $\Phi_0$. Furthermore, \eqref{Pr8LemEq2} gives that
\begin{equation}\label{Pr8LemEq23}
\int_\Omega\left<\nabla\Phi_n,\nabla\psi\right>dx-\lambda_n\int_\Omega h_nf'\(U_n\)\Phi_n\psi dx=\smallo\(1\)
\end{equation}
as $n\to\infty$, for all $\psi\in C^\infty_c\(\Omega\)$. By rescaling as in \eqref{Pr8LemEq5} and using \eqref{Pr8LemEq8} together with the fact that $\widehat\Phi_n\rightharpoonup0$ in $D^{1,2}\(\R^2\)$, we obtain that
\begin{equation}\label{Pr8LemEq24}
\sum_{i=1}^k\int_{B\(\overline{\tau_{i,n}},r_n\)}h_nf'\(U_n\)\Phi_n\psi dx=\smallo\(1\)
\end{equation}
as $n\to\infty$. By using similar estimates as in \eqref{Pr8LemEq19}, we obtain
\begin{equation}\label{Pr8LemEq25}
\sum_{i=1}^k\int_{A\(\overline{\tau_{i,n}},r_n,R_n\)}h_nf'\(U_n\)\Phi_n\psi dx=\smallo\(1\)
\end{equation}
as $n\to\infty$. By using \eqref{Sec32Eq3}, \eqref{Pr4Eq2} and since $w_n\to w_0$ in $C^0\(\overline\Omega\)$, we obtain that $U_n\mathbf{1}_{\Omega_{R_n,\tau_n}}$ is uniformly bounded and converges pointwise to $u_0$ in $\Omega$. Since moreover $\Phi_n\rightharpoonup\Phi_0$ in $H^1_0\(\Omega\)$, $\lambda_n\to\lambda_0$ and $h_n\to h_0$ in $C^0\(\overline\Omega\)$, it follows from \eqref{Pr8LemEq23}--\eqref{Pr8LemEq25} that $\Phi_0$ is a solution of the equation
$$\Delta\Phi_0=\lambda_0 h_0f'\(w_0\)\Phi_0\qquad\text{in }\R^n.$$
Since $w_0$ is non-degenerate, we then obtain that $\Phi_0\equiv0$. It then follows from standard integration theory that 
\begin{equation}\label{Pr8LemEq26}
\int_{\Omega_{R_n,\tau_n}}f'\(U_n\)\Phi_n^2dx=\smallo\(1\)
\end{equation}
as $n\to\infty$.

\smallskip
Finally, by combining \eqref{Pr8LemEq18}, \eqref{Pr8LemEq19} and \eqref{Pr8LemEq26}, we obtain a contradiction with \eqref{Pr8LemEq3} and \eqref{Pr8LemEq4}. This ends the proof of Lemma~\ref{Pr8Lem}.
\endproof

\proof[Proof of Proposition~\ref{Pr8}]
We let $N_{\varepsilon,\gamma,\tau,\theta}:V_{\varepsilon,\gamma,\tau}^\perp\to V_{\varepsilon,\gamma,\tau}^\perp$ and $T_{\varepsilon,\gamma,\tau,\theta}:V_{\varepsilon,\gamma,\tau}^\perp\to V_{\varepsilon,\gamma,\tau}^\perp$ be the operators defined as
\begin{align*}
N_{\varepsilon,\gamma,\tau,\theta}\(\Phi\)&:=\Pi_{\varepsilon,\gamma,\tau}^\perp\big(\Delta^{-1}[\lambda_\varepsilon h_\varepsilon(f\(U_{\varepsilon,\gamma,\tau,\theta}+\Phi\)-f\(U_{\varepsilon,\gamma,\tau,\theta}\)-f'\(U_{\varepsilon,\gamma,\tau,\theta}\)\Phi)]\big),\allowdisplaybreaks\\
T_{\varepsilon,\gamma,\tau,\theta}\(\Phi\)&:=L^{-1}_{\varepsilon,\gamma,\tau,\theta}\(N_{\varepsilon,\gamma,\tau,\theta}\(\Phi\)-\Pi_{\varepsilon,\gamma,\tau}^\perp\(R_{\varepsilon,\gamma,\tau,\theta}\)\)
\end{align*}
for all $\Phi\in V_{\varepsilon,\gamma,\tau}^\perp$, where $R_{\varepsilon,\gamma,\tau,\theta}$ and $L_{\varepsilon,\gamma,\tau,\theta}$ are as in \eqref{Sec4Eq} and \eqref{Pr8LemEq0}. Remark that the equation \eqref{Pr8Eq1} can be rewritten as the fixed point equation $T_{\varepsilon,\gamma,\tau,\theta}\(\Phi\)=\Phi$. For every $C>0$, $\varepsilon\in\(0,\varepsilon_5'\)$ and $\(\gamma,\tau,\theta\)\in P^k_\varepsilon\(\delta\)$, we define
$$V_{\varepsilon,\gamma,\tau,\theta}\(C\):=\left\{\Phi\in V_{\varepsilon,\gamma,\tau}^\perp:\,\left\|\Phi\right\|_{H^1_0}\le C\frac{\delta_\varepsilon\ln\overline\gamma_\varepsilon}{\overline\gamma_\varepsilon^2}\right\}.$$
We will prove that if $C$ is chosen large enough, then $T_{\varepsilon,\gamma,\tau,\theta}$ has a fixed point in $V_{\gamma,\tau,\theta}\(C\)$. By using \eqref{Pr8LemEq1}, we obtain
\begin{equation}\label{Pr8Eq3}
\left\|T_{\varepsilon,\gamma,\tau,\theta}\(\Phi\)\right\|_{H^1_0}\le C_5'\big(\left\|N_{\varepsilon,\gamma,\tau,\theta}\(\Phi\)\right\|_{H^1_0}+\left\|R_{\varepsilon,\gamma,\tau,\theta}\right\|_{H^1_0}\big).
\end{equation}
For every $\Phi_1,\Phi_2\in V_{\varepsilon,\gamma,\tau,\theta}\(C\)$ and $\psi\in V_{\varepsilon,\gamma,\tau}^\perp$, by integrating by parts and applying the mean value theorem, we obtain
\begin{align}\label{Pr8Eq4}
&\<N_{\varepsilon,\gamma,\tau,\theta}\(\Phi_1\)-N_{\varepsilon,\gamma,\tau,\theta}\(\Phi_2\),\psi\>_{H^1_0}\nonumber\\
&\qquad=\lambda_\varepsilon\int_\Omega h_\varepsilon \(f'\(U_{\varepsilon,\gamma,\tau,\theta}+t\Phi_1+\(1-t\)\Phi_2\)-f'\(U_{\varepsilon,\gamma,\tau,\theta}\)\)\(\Phi_1-\Phi_2\)\psi dx\nonumber\allowdisplaybreaks\\
&\qquad=\lambda_\varepsilon\int_\Omega h_\varepsilon f''\(U_{\varepsilon,\gamma,\tau,\theta}+st\Phi_1+s\(1-t\)\Phi_2\)\(t\Phi_1+\(1-t\)\Phi_2\)\(\Phi_1-\Phi_2\)\psi dx
\end{align}
for some functions $s,t:\Omega\to\[0,1\]$. Since $\lambda_\varepsilon\to\lambda_0$, $h_\varepsilon\to h_0$ in $C^0\(\overline\Omega\)$ and $f''$ is increasing, it follows from \eqref{Pr8Eq4} that
\begin{multline}\label{Pr8Eq5}
\<N_{\varepsilon,\gamma,\tau,\theta}\(\Phi_1\)-N_{\varepsilon,\gamma,\tau,\theta}\(\Phi_2\),\psi\>_{H^1_0}\\
=\bigO\bigg(\int_\Omega f''\(\left|U_{\varepsilon,\gamma,\tau,\theta}\right|+\left|\Phi_1\right|+\left|\Phi_2\right|\)\(\left|\Phi_1\right|+\left|\Phi_2\right|\)\left|\Phi_1-\Phi_2\right|\left|\psi\right|dx\bigg).
\end{multline}
For every $p>1$, by using H\"older's inequality together with the continuity of the embedding $H^1_0\(\Omega\)\hookrightarrow L^{3p'}\(\Omega\)$, we obtain
\begin{multline}\label{Pr8Eq6}
\int_{\Omega} f''\(\left|U_{\varepsilon,\gamma,\tau,\theta}\right|+\left|\Phi_1\right|+\left|\Phi_2\right|\)\(\left|\Phi_1\right|+\left|\Phi_2\right|\)\left|\Phi_1-\Phi_2\right|\left|\psi\right|dx\\
=\bigO\Big(\left\|f''\(\left|U_{\varepsilon,\gamma,\tau,\theta}\right|+\left|\Phi_1\right|+\left|\Phi_2\right|\)\right\|_{L^p}\left\|\left|\Phi_1\right|+\left|\Phi_2\right|\right\|_{H^1_0}\left\|\Phi_1-\Phi_2\right\|_{H^1_0}\left\|\psi\right\|_{H^1_0}\Big).
\end{multline}
Since $f''$ is increasing, we obtain
\begin{equation}\label{Pr8Eq7}
f''\(\left|U_{\varepsilon,\gamma,\tau,\theta}\right|+\left|\Phi_1\right|+\left|\Phi_2\right|\)\le f''\big(\widetilde{U}_{\varepsilon,\gamma,\tau,\theta}\big)+f''\big(\widetilde{\Phi}_\varepsilon\big),
\end{equation}
where
$$\widetilde{U}_{\varepsilon,\gamma,\tau,\theta}:=\(1+\delta_\varepsilon\)\left|U_{\varepsilon,\gamma,\tau,\theta}\right|\quad\text{and}\quad\widetilde{\Phi}_\varepsilon:=\(1+\delta_\varepsilon^{-1}\)\(\left|\Phi_1\right|+\left|\Phi_2\right|\).$$
Remark that $\widetilde{\Phi}_\varepsilon\to0$ in $H^1_0\(\Omega\)$ as $\varepsilon\to0$ since $\Phi_1,\Phi_2\in V_{\varepsilon,\gamma,\tau,\theta}\(C\)$. By using H\"older's inequality together with the Moser--Trudinger's inequality and the continuity of the embedding $H^1_0\(\Omega\)\hookrightarrow L^{6p}\(\Omega\)$, we then obtain
\begin{align}\label{Pr8Eq8}
\big\|f''\big(\widetilde{\Phi}_\varepsilon\big)\big\|_{L^p}&=2\big\|\widetilde{\Phi}_\varepsilon\big(3+2\widetilde{\Phi}_\varepsilon^2\big)\exp\big(\widetilde{\Phi}_\varepsilon^2\big)\big\|_{L^{2p}}\nonumber\\
&\le2\big\|\widetilde{\Phi}_\varepsilon\big\|_{H^1_0}\big(3+2\big\|\widetilde{\Phi}_\varepsilon\big\|_{H^1_0}^2\big)\big\|\exp\big(\widetilde{\Phi}_\varepsilon^2\big)\big\|_{L^{2p}}=\smallo\(1\)
\end{align}
as $\varepsilon\to0$. For every $i\in\left\{1,\dotsc,k\right\}$, by remarking that $f''\(s\)\le6sf'\(s\)$ for all $s\ge0$ and using similar estimates as in \eqref{Pr8LemEq7} and \eqref{Pr8LemEq8}, we obtain
\begin{equation}\label{Pr8Eq9}
f''\big(\widetilde{U}_{\varepsilon,\gamma,\tau,\theta}\(x\)\big)=\bigO\(\frac{\overline\gamma_i\(\tau\)\overline\mu_i\(\tau\)^2}{\(\overline\mu_i\(\tau\)^2+\left|x-\overline{\tau_i}\right|^2\)^2}\)
\end{equation}
uniformly in $x\in B\(\overline{\tau_i},r_\varepsilon\)$, where $\overline\mu_i\(\tau\)$ is defined by
$$\overline\mu_i\(\tau\)^2:=4\overline\gamma_i\(\tau\)^{-2}\exp\big(-\overline\gamma_i\(\tau\)^2\big).$$
It follows from \eqref{Pr8Eq9} that
\begin{equation}\label{Pr8Eq10}
\big\|f''\big(\widetilde{U}_{\varepsilon,\gamma,\tau,\theta}\big)\mathbf{1}_{B\(\overline{\tau_i},r_\varepsilon\)}\big\|_{L^p}=\bigO\(\overline\gamma_i\(\tau\)\overline\mu_i\(\tau\)^{-2/p'}\),
\end{equation}
where $p'$ is the conjugate exponent of $p$. By using \eqref{Pr4Eq3b}, we obtain
\begin{equation}\label{Pr8Eq11}
\big\|f''\big(\widetilde{U}_{\varepsilon,\gamma,\tau,\theta}\big)\mathbf{1}_{A\(\overline{\tau_i},r_\varepsilon,R_\varepsilon\)}\big\|_{L^p}=\smallo\(1\)
\end{equation}
as $\varepsilon\to0$, where $R_\varepsilon:=\exp\(-\overline\gamma_\varepsilon\)$, provided we choose $p$ such that $p<1/\delta_0$. Furthermore, since $U_{\varepsilon,\gamma,\tau,\theta}$ is uniformly bounded in $\Omega_{R_\varepsilon,\tau}$, we obtain
\begin{equation}\label{Pr8Eq12}
\big\|f''\big(\widetilde{U}_{\varepsilon,\gamma,\tau,\theta}\big)\mathbf{1}_{\Omega_{R_\varepsilon,\tau}}\big\|_{L^p}=\bigO\(1\).
\end{equation}
By putting together \eqref{Pr8Eq5}--\eqref{Pr8Eq7}, \eqref{Pr8Eq8} and \eqref{Pr8LemEq7} and \eqref{Pr8Eq12}, we obtain
\begin{equation}\label{Pr8Eq13}
\left\|N_{\varepsilon,\gamma,\tau,\theta}\(\Phi_1\)-N_{\varepsilon,\gamma,\tau,\theta}\(\Phi_2\)\right\|_{H^1_0}=\bigO\(\overline\gamma_i\(\tau\)\overline\mu_i\(\tau\)^{-2/p'}\left\|\left|\Phi_1\right|+\left|\Phi_2\right|\right\|_{H^1_0}\left\|\Phi_1-\Phi_2\right\|_{H^1_0}\).
\end{equation}
Remark that since $\Phi_1,\Phi_2\in V_{\varepsilon,\gamma,\tau,\theta}\(C\)$, we obtain
\begin{equation}\label{Pr8Eq14}
\overline\gamma_i\(\tau\)\overline\mu_i\(\tau\)^{-2/p'}\left\|\left|\Phi_1\right|+\left|\Phi_2\right|\right\|_{H^1_0}=\smallo\(1\)
\end{equation}
as $\varepsilon\to0$, provided we choose $p$ such that $2/p'<\delta_1+1/2$, i.e. $p<4/\(3-2\delta_1\)$. It follows from \eqref{Pr8Eq13} and \eqref{Pr8Eq14} that
\begin{equation}\label{Pr8Eq15}
\left\|N_{\varepsilon,\gamma,\tau,\theta}\(\Phi_1\)-N_{\varepsilon,\gamma,\tau,\theta}\(\Phi_2\)\right\|_{H^1_0}=\smallo\big(\left\|\Phi_1-\Phi_2\right\|_{H^1_0}\big)
\end{equation}
as $\varepsilon\to0$. By using \eqref{Pr7Eq1}, \eqref{Pr8Eq3}, \eqref{Pr8Eq15} and since $N_{\varepsilon,\gamma,\tau,\theta}\(0\)=0$, we obtain that there exist $\varepsilon_6\(\delta\)\in\(0,\varepsilon_5\(\delta\)\)$ and $C_6=C_6\(\delta\)>0$ such that for every $\varepsilon\in\(0,\varepsilon_6\(\delta\)\)$ and $\(\gamma,\tau,\theta\)\in P^k_\varepsilon\(\delta\)$, $T_{\varepsilon,\gamma,\tau,\theta}$ is a contraction mapping on $V_{\varepsilon,\gamma,\tau,\theta}\(C_6\)$. We can then apply the fixed point theorem, which gives that there exists a unique solution $\Phi_{\varepsilon,\gamma,\tau,\theta}\in V_{\varepsilon,\gamma,\tau,\theta}\(C_6\)$ to the equation \eqref{Pr8Eq1}. The continuity of $\Phi_{\varepsilon,\gamma,\tau,\theta}$ in $\(\gamma,\tau,\theta\)$ follows from the continuity of $U_{\varepsilon,\gamma,\tau,\theta}$, $Z_{0,i,\varepsilon,\gamma,\tau,\theta}$ and $Z_{1,i,\varepsilon,\gamma,\tau,\theta}$ in $\(\gamma,\tau,\theta\)$. This ends the proof of Proposition~\ref{Pr8}.
\endproof

As a last step, we prove the following:

\begin{proposition}\label{Pr9}
Let $\varepsilon_6$ and $\Phi_{\varepsilon,\gamma,\tau,\theta}$ be as in Proposition~\ref{Pr8}. Then there exists $\delta_7\in\(0,1\)$ such that for every $\delta\in\(0,\delta_7\)$, there exists $\varepsilon_7\(\delta\)\in\(0,\varepsilon_7\(\delta\)\)$ such that for every $\varepsilon\in\(0,\varepsilon_7\(\delta\)\)$, there exists $(\gamma_\varepsilon,\tau_\varepsilon,\theta_\varepsilon)\in P^k_\varepsilon\(\delta\)$ such that
\begin{equation}\label{Pr9Eq1}
U_{\varepsilon,\gamma_\varepsilon,\tau_\varepsilon,\theta_\varepsilon}+\Phi_{\varepsilon,\gamma_\varepsilon,\tau_\varepsilon,\theta_\varepsilon}=\Delta^{-1}\[\lambda_\varepsilon h_\varepsilon f\(U_{\varepsilon,\gamma_\varepsilon,\tau_\varepsilon,\theta_\varepsilon}+\Phi_{\varepsilon,\gamma_\varepsilon,\tau_\varepsilon,\theta_\varepsilon}\)\].
\end{equation}
\end{proposition}

The proof of Proposition~\ref{Pr9} relies on the following:

\begin{lemma}\label{Pr9Lem} 
Set 
$$\widetilde{R}_{\ve,\gamma,\tau,\theta}:=U_{\varepsilon,\gamma,\tau,\theta}+\Phi_{\varepsilon,\gamma,\tau,\theta}-\Delta^{-1}\[\lambda_\varepsilon h_\varepsilon f\(U_{\varepsilon,\gamma,\tau,\theta}+\Phi_{\varepsilon,\gamma,\tau,\theta}\)\].$$
Then for every $i\in\left\{1,\dotsc,k\right\}$ and $\delta\in\(0,1\)$, we have
\begin{align}\label{Pr9LemEq1}
\<\widetilde{R}_{\ve,\gamma,\tau,\theta},Z_{0,i,\varepsilon,\gamma,\tau}\>_{H^1_0}&
=-8\pi \sum_{j=1}^k \partial_{\gamma_i}\big[E^{\(j\)}_{\ve,\gamma,\tau}\big]\(E^{\(j\)}_{\ve,\gamma,\tau} +\theta_j \gammae \)+ \frac{4\pi }{\gammae^2}E^{\(i\)}_{\ve,\gamma,\tau} +\smallo\(\frac{\delta_\ve \ln\overline{\gamma}_\ve}{\overline{\gamma}_\ve^5}\),\allowdisplaybreaks\\
\<\widetilde{R}_{\ve,\gamma,\tau,\theta},B_{\varepsilon,\gamma_i,\tau_i}\>_{H^1_0}&=-8\pi  \gammae \(E^{\(i\)}_{\ve,\gamma,\tau}+\theta_i\gammae\) +\smallo\(\frac{\delta_\ve}{\gammae^2}\),\label{Pr9LemEq2}\allowdisplaybreaks\\
\<\widetilde{R}_{\ve,\gamma,\tau,\theta},Z_{1,i,\varepsilon,\gamma,\tau}\>_{H^1_0}&=-\frac{4\pi}{\overline\gamma_\varepsilon}\Bigg(a_0l\tau_i^{l-1}-\frac{2}{\overline\gamma_\varepsilon}\sum_{j\ne i}\frac{1}{\tau_i-\tau_j}\Bigg)+\smallo\(\frac{1}{\overline\gamma_\varepsilon^2 d_\ve}\)\label{Pr9LemEq3}
\end{align}
as $\varepsilon\to0$, uniformly in $\(\gamma,\tau,\theta\)\in P^k_\varepsilon\(\delta\)$.
\end{lemma}

\begin{rem}\label{RemCoupling}
As an evidence of the \emph{strong} interaction generated by the Moser-Trudinger critical nonlinearity, we stress that the variables $\theta$ and $\gamma$ are intricately coupled in the expansions \eqref{Pr9LemEq1}--\eqref{Pr9LemEq3} used to determine $(\gamma_\ve,\theta_\ve,\tau_\ve)$. This is not the case for 2-dimensional Liouville-type equations (see for instance~\cite{ChenLin-Liouville}), for which it is possible to construct blowing-up solutions without introducing neither the parameter $\theta$ nor the bubbles $B_{\varepsilon,\gamma_i,\tau_i}$ in $V_{\varepsilon,\gamma,\tau}$ (see for instance~\cite{EGP} working also in the $H^1_0(\Omega)$-framework). Finally, even not facing a situation with clustering or nonzero weak limit like ours, it is delicate to get a clean energy expansion in the Moser-Trudinger critical case (see~\cite{delPinoMusRuf2}). In particular, this expansion has to eventually fit with the cancellation pointed out by~\cite{ManMar} for the blow-up solutions.
\end{rem}

\proof[Proof of \eqref{Pr9LemEq1}] We start with computations that will be used also in the proofs of \eqref{Pr9LemEq2}-\eqref{Pr9LemEq2}. Given $Z\in H^1_0(\Omega)$, integration by parts yields
\begin{equation*}
\<\widetilde{R}_{\ve,\gamma,\tau,\theta},Z\>_{H^1_0} =\int_\Omega[\Delta(U_{\varepsilon,\gamma,\tau,\theta}+\Phi_{\varepsilon,\gamma,\tau,\theta})-f_\ve(U_{\varepsilon,\gamma,\tau,\theta}+\Phi_{\varepsilon,\gamma,\tau,\theta})]Zdx,
\end{equation*}
where we use the notation $f_\ve= \lambda_\ve h_\ve  f$.\footnote{We shall always write $f_\ve(U)$ instead of $f_\ve(x,U)$, ignoring the dependence on $x$.} 
We now expand for real numbers $U$ and $R$,
\begin{equation}\label{expUR}
\exp[(U+R)^2]=\exp(U^2) \exp(2UR+R^2)=\exp(U^2)[1+2UR+\bigO(U^2R^2)]
\end{equation}
uniformly for $\left|UR\right|\le 1$ and $\left|R\right|\le 1\le  \left|U\right|$, so that, recalling that $f'\(t\)=\(1+2t^2\)\exp\(t^2\)$,
\begin{equation}\label{Taylor}
f\(U+R\)=f\(U\)+f'\(U\)R+\bigO\(U^3R^2\exp\(U^2\)\),
\end{equation}
and similarly for $f_\ve$ since $\lambda_\ve h_\ve=\bigO\(1\)$. We apply this to
\begin{equation}
U= U_{\ve,\gamma,\tau}=u_\ve+\sum_{i=1}^k B_{\ve,\gamma_i,\tau_i}+\Psi_{\ve,\gamma,\tau},\quad R=\widetilde{\Phi}_{\ve,\gamma,\tau,\theta}:=\sum_{i=1}^k\theta_i B_{\ve,\gamma_i,\tau_i} +\Phi_{\ve,\gamma,\tau,\theta}\label{EqURBPhi}
\end{equation}
to obtain
\begin{multline*}
f_\ve\(U_{\varepsilon,\gamma,\tau,\theta}+\Phi_{\varepsilon,\gamma,\tau,\theta}\)=f_\ve\(U_{\ve,\gamma,\tau}\) +f_\ve'\(U_{\ve,\gamma,\tau}\)\(\sum_{i=1}^k\theta_i B_{\ve,\gamma_i,\tau_i} +\Phi_{\ve,\gamma,\tau,\theta} \)\\
+\bigO\(\exp\(U^2_{\ve,\gamma,\tau}\)U_{\ve,\gamma,\tau}^3\widetilde{\Phi}_{\ve,\gamma,\tau,\theta}^2\).
\end{multline*}
Recalling Proposition~\ref{Pr4}, and in particular that $U_{\ve,\gamma,\tau}$ is an exact solution outside the balls $B\(\overline{\tau_j},r_\ve+r_\ve^2\)$, we get
\begin{align}\label{bigexpansion}
\<\widetilde{R}_{\ve,\gamma,\tau,\theta},Z\>_{H^1_0}&=\sum_{j=1}^k\int_{B\(\overline{\tau_j},r_\ve\)} \[\Delta U_{\ve,\gamma,\tau}-f_\ve\(U_{\ve,\gamma,\tau}\)\] Zdx \nonumber\\
&\quad + \sum_{j=1}^k\int_{\Omega^j_\ve} \[\Delta U_{\ve,\gamma,\tau}-f_\ve\(U_{\ve,\gamma,\tau}\)\] Zdx\allowdisplaybreaks\nonumber\\
&\quad + \sum_{j=1}^k \theta_j \int_\Omega \[\Delta  B_{\ve,\gamma_j,\tau_j}-f_\ve'\(U_{\ve,\gamma,\tau}\)B_{\ve,\gamma_j,\tau_j}\] Zdx\allowdisplaybreaks\nonumber\\
&\quad +\int_\Omega \[\Delta \Phi_{\ve,\gamma,\tau,\theta}-f_\ve'\(U_{\ve,\gamma,\tau}\)\Phi_{\ve,\gamma,\tau,\theta}\] Zdx\nonumber\\
&\quad +\bigO\(\int_{\Omega}\left|U_{\ve,\gamma,\tau}\right|^3 \exp\big(U_{\ve,\gamma,\tau}^2\big)\widetilde{\Phi}_{\ve,\gamma,\tau,\theta}^2 \left|Z\right|dx\)\nonumber\\
&=: \sum_{j=1}^k \[(A)_j+(A')_j+(B)_j\]+(C)+(D),
\end{align}
where $\Omega^{j}_\ve:=B\(\overline{\tau_i},r_\varepsilon+r_\varepsilon^2\)\backslash B\(\overline{\tau_i},r_\varepsilon\)$. We now set $Z=Z_{0,i,\ve,\gamma,\tau}$ in \eqref{bigexpansion} and estimate the various terms.

 In order to evaluate $(A):=\sum_{j=1}^k(A)_j$, expand as in \eqref{UBEF}
\begin{equation}\label{EqUve}
U_{\ve,\gamma,\tau}=\overline{B}_{\ve,\gamma_j,\overline{\tau_j}} + E^{\(j\)}_{\ve,\gamma,\tau}+F^{\(j\)}_{\ve,\gamma,\tau},\quad \text{in }B\(\overline{\tau_j},r_\ve\).
\end{equation}
Using Proposition~\ref{PrFive} and omitting some indices, we get
\begin{equation}\label{EqRive}
R_j\(x\):=E^{\(j\)}_{\ve,\gamma,\tau}+F^{\(j\)}_{\ve,\gamma,\tau}\(x\)=E^{\(j\)}_{\ve,\gamma,\tau}+\bigO\(\frac{\left|x-\overline{\tau_j}\right|}{\gammae d_\ve}\)
=:R_j^s\(x\)+R_j^r,
\end{equation}
for all $x\in B\(\overline{\tau_j},r_\ve\)$, where the letters $s$ and $r$ stand for ``symmetric'' and ``remainder'', respectively. Using \eqref{EqUve} and \eqref{Pr6Eq0b}, we get
\begin{align}\label{EqZ0i1}
Z_{0,i}:=Z_{0,i,\ve,\gamma,\tau}&=\partial_{\gamma_i}\big[\overline{B}_{\ve,\gamma_j,\overline{\tau_j}}+R_j\big]\nonumber\\
&=\partial_{\gamma_i}\big[\overline{B}_{\ve,\gamma_j,\overline{\tau_j}}\(x\)+E^{\(j\)}_{\ve,\gamma,\tau}\big] + \bigO\(\left|x-\overline{\tau_j}\right|\),\quad \text{in }B\(\overline{\tau_j},r_\ve\),
\end{align}
where we also replaced $\bigO\(\left|x-\overline{\tau_j}\right|/\(\overline \gamma_\ve^2 d_\ve\)\)$ by $\bigO\(\left|x-\overline{\tau_j}\right|\)$ for simplicity.
Using Proposition~\ref{Pr4} and \eqref{Pr4Eq4b}, i.e. $\Delta{B}_{\ve,\gamma_j,\tau_j}=\Delta \overline{B}_{\ve,\gamma_j,\overline{\tau_j}}$, in $B\(\overline{\tau_j},r_\ve\)$,
we can write
$$(A)_j=\int_{B\(\overline{\tau_j},r_\ve\)}\[\Delta \overline{B}_{\ve,\gamma_j,\overline{\tau_j}}-f_\ve\(U_{\ve,\gamma,\tau}\) \] Z_{0,i}dx.$$
We now Taylor expand as in \eqref{Taylor} with
$$U=\overline{B}_{\ve,\gamma_j,\overline{\tau_j}},\quad R= R_j= E^{\(j\)}_{\ve,\gamma,\tau}+F^{\(j\)}_{\ve,\gamma,\tau},$$
and since $ \overline{B}_{\ve,\gamma_i,\overline{\tau_i}}$ is an exact solution in $B\(\overline{\tau_i},r_\ve\)$, we estimate
\begin{multline}\label{EqAj2}
(A)_j=\int_{B\(\overline{\tau_j},r_\ve\)}\underbrace{\[\Delta \overline{B}_{\ve,\gamma_j,\overline{\tau_j}}- \lambda_\ve h_\ve\(\overline{\tau_j}\)f\(\overline{B}_{\ve,\gamma_j,\overline{\tau_j}}\)\]}_{=0} Z_{0,i}dx\\
-\int_{B\(\overline{\tau_j},r_\ve\)}\lambda_\ve \(h_\ve -h_\ve\(\overline{\tau_j}\)\) f\(\overline{B}_{\ve,\gamma_j,\overline{\tau_j}}\)Z_{0,i} dx -\int_{B\(\overline{\tau_j},r_\ve\)} \lambda_\ve h_\ve f'\(\overline{B}_{\ve,\gamma_j,\overline{\tau_j}}\)R_jZ_{0,i}dx\\
+\bigO\(\int_{B\(\overline{\tau_j},r_\ve\)}\overline\gamma_\ve^3\exp \big(\overline{B}_{\ve,\gamma_j,\overline{\tau_j}}^2\big) R_j^2  \left|Z_{0,i}\right|dx \).
\end{multline}
Observing that $h_\ve -h_\ve\(\overline{\tau_j}\)=\bigO\(\left|x-\overline{\tau_j}\right|\)$, using \eqref{EqRive} to bound $F^{\(j\)}_{\ve,\gamma,\tau}$, writing
$$f'\(\overline{B}_{\ve,\gamma_j,\overline{\tau_j}}\)=\bigO\(\overline{\gamma}_\ve f\(\overline{B}_{\ve,\gamma_j,\overline{\tau_j}}\)\) =\bigO\big(\overline{\gamma}_\ve^2 \exp\big(\overline{B}_{\ve,\gamma_j,\overline{\tau_j}}^2\big)\big),\quad \text{in }B\(\overline{\tau_j},r_\ve\),$$ 
and using $\left|Z_{0,i}\right|=\bigO\(1\)$, we simplify to
\begin{align}\label{eqAib}
(A)_j&=-\int_{B\(\overline{\tau_j},r_\ve\)} \lambda_\ve h_\ve\(\overline{\tau_j}\)f'\(\overline{B}_{\ve,\gamma_j,\overline{\tau_j}}\) E^{\(j\)}_{\ve,\gamma,\tau} \partial_{\gamma_i}\big[ \overline{B}_{\ve,\gamma_j,\overline{\tau_j}}+ E^{\(j\)}_{\ve,\gamma,\tau}\big]dx\nonumber\\
&\quad+\bigO\(\int_{B\(\overline{\tau_j},r_\ve\)}\overline\gamma_\ve^3\exp\big(\overline{B}_{\ve,\gamma_j,\overline{\tau_j}}^2\big) \(|E^{\(j\)}_{\ve,\gamma,\tau}|^2+\left|x-\overline{\tau_j}\right|\)dx \).
\end{align}

Now write
\begin{equation*}
\partial_{\gamma_i}\big[ \overline{B}_{\ve,\gamma_j,\overline{\tau_j}}+ E^{\(j\)}_{\ve,\gamma,\tau}\big]=\delta_{ij}\partial_{\gamma_i}\[\overline{B}_{\ve,\gamma_i,\overline{\tau_i}}\]+\partial_{\gamma_i}\big[E^{\(j\)}_{\ve,\gamma,\tau}\big],
\quad \text{on } B\(\overline{\tau_j},r_\ve\),\\
\end{equation*}
where $\delta_{ij}$ is the Kronecker symbol. Observing that
\begin{equation}\label{EqBeBj}
\overline{B}_{\ve,\gamma_j,\overline{\tau_j}}\(x\)=\overline{B}_{\gamma_j}\big(\sqrt{\lambda_{\ve,j}}\(x-\overline{\tau_j}\)\big),\quad\text{where }\lambda_{\ve,j}:=\lambda_\ve h_\ve\(\overline{\tau_j}\),
\end{equation}
with the change of variables $\sqrt{\lambda_{\ve,j}}\(x-\overline{\tau_j}\)= y$ and Proposition~\ref{Pr12}, we get
\begin{equation}\label{intf'}
\int_{B\(\overline{\tau_j},r_\ve\)}\lambda_{\ve,j}f'\(\overline{B}_{\ve,\gamma_j,\overline{\tau_j}}\)dx=\int_{B(0,\sqrt{\lambda_{\ve,j}} r_\ve)}f'\(\overline B_{\gamma_j}\)dy=8\pi+\bigO\(\frac{1}{\gammae^2}\),
\end{equation}
where $\overline B_{\gamma_j}$ is as in Proposition~\ref{Pr10}. With the same change of variables and Proposition~\ref{Pr12}, we also get
\begin{align*}
\int_{B\(\overline{\tau_i},r_\ve\)}\lambda_{\ve,i}f'\big(\overline{B}_{\ve,\gamma_i,\overline{\tau_i}}\big) \partial_{\gamma_i}\[\overline{B}_{\ve,\gamma_i,\overline{\tau_i}}\]dx&=\int_{B(0,\sqrt{\lambda_{\ve,i}} r_\ve)}f'\(\overline B_{\gamma_i}\)Z_{0,\gamma_i}dy=-\frac{4\pi+\smallo\(1\)}{\overline\gamma_\ve^2},
\end{align*}
where $Z_{0,\gamma_i}$ is as in Proposition~\ref{Pr11}. Now, using Proposition~\ref{PrEive}, the dominant term in $(A)_j$ becomes
\begin{align*}
&-E^{\(j\)}_{\ve,\gamma,\tau}\int_{B\(\overline{\tau_j},r_\ve\)}\lambda_{\ve,j}f'\(\overline{B}_{\ve,\gamma_j,\overline{\tau_j}}\) \(\partial_{\gamma_i}\big[E^{\(j\)}_{\ve,\gamma,\tau}\big]+\partial_{\gamma_i}\[\overline{B}_{\ve,\gamma_j,\overline{\tau_j}}\]\)dy\\
&\qquad  = -E^{\(j\)}_{\ve,\gamma,\tau}\(\(8\pi+\bigO\(\frac{1}{\overline{\gamma}_\ve^2}\)\)  \partial_{\gamma_i}\big[E^{\(j\)}_{\ve,\gamma,\tau}\big]-\delta_{ij}\frac{4\pi+\smallo\(1\)}{\overline{\gamma}_\ve^2} \)\\
&\qquad=-E^{\(j\)}_{\ve,\gamma,\tau}\(8\pi \partial_{\gamma_i}\big[E^{\(j\)}_{\ve,\gamma,\tau}\big]-\delta_{ij}\frac{4\pi}{\overline{\gamma}_\ve^2} \) +\smallo\(\frac{\delta_\ve \ln\overline{\gamma}_\ve}{\overline{\gamma}_\ve^5}\)
\end{align*}
Concerning the remainder term in \eqref{eqAib}, again using Proposition~\ref{Pr12}, we have
$$\int_{B\(\overline{\tau_j},r_\ve\)}\exp\big(\overline{B}_{\ve,\gamma_j,\overline{\tau_j}}^2\big)\overline\gamma_\ve^3\delta_\ve^2dx=\bigO\(\overline\gamma_\ve \delta_\ve^2\)=\smallo\(\frac{\delta_\ve \ln \overline{\gamma}_{\ve}}{\overline{\gamma}_{\ve}^5}\),$$
and, with the usual change of variables and Proposition~\ref{Pr13}, we obtain
\begin{align*}
\int_{B\(\overline{\tau_j},r_\ve\)}\exp\big(\overline{B}_{\ve,\gamma_j,\overline{\tau_j}}^2\big)\overline\gamma_\ve^3\left|x-\overline{\tau_j}\right|dx=\bigO\(\gammae^3 \mu_{\gamma_j}^{3\delta_0-2\delta_0^2+\smallo\(1\)}\)&=\bigO\(\mu_\ve^{3\delta_0-2\delta_0^2+\smallo\(1\)}\)\\
&=\smallo\(\frac{\delta_\ve \ln \overline{\gamma}_{\ve}}{\overline{\gamma}_{\ve}^5}\),
\end{align*}
where in the last identity, we used that $\delta_\ve=\overline \mu_\ve^{\delta_1+1/2}$ and $3\delta_0-2\delta_0^2 > \delta_1+\frac12$ thanks to \eqref{Pr7Eq0}, so that 
$$\mu_\ve^{3\delta_0-2\delta_0^2+\smallo\(1\)}=\bigO\(\frac{\delta_\ve}{\gammae^{a}}\),\quad \text{for any }a\in\R.$$ 
We therefore get
$$(A)_j=-E^{\(j\)}_{\ve,\gamma,\tau}\(8\pi \partial_{\gamma_i}\big[E^{\(j\)}_{\ve,\gamma,\tau}\big]-\delta_{ij}\frac{4\pi}{\overline{\gamma}_\ve^2} \) +\smallo\(\frac{\delta_\ve \ln\overline{\gamma}_\ve}{\overline{\gamma}_\ve^5}\)$$
Summing over $j$, we then obtain
\begin{equation}\label{EqAfinal}
(A)=\sum_{j=1}^k(A)_j=-8\pi \sum_{j=1}^k E^{\(j\)}_{\ve,\gamma,\tau}\partial_{\gamma_i}\big[E^{\(j\)}_{\ve,\gamma,\tau}\big] +\frac{4\pi}{\gammae^2}E^{\(i\)}_{\ve,\gamma,\tau}+\smallo\(\frac{\delta_\ve \ln\overline{\gamma}_\ve}{\overline{\gamma}_\ve^5}\).
\end{equation}

As for the error term in the annuli, we have from Proposition~\ref{Pr4}, 
$$(A')_j=\int_{\Omega^j_\ve} \(\chi_{\ve,\tau}-1\) f_\ve\(U_{\ve,\gamma,\tau}\)Z_{0,i,\ve,\gamma,\tau}dx,$$
hence, from \eqref{Estannuli}, 
\begin{equation}\label{EqA'final}
(A')=\sum_{j=1}^k\bigO\(|\Omega^j_\ve|\overline\mu_\ve^{-2\delta_0^2+\smallo\(1\)} \)=\bigO\(\overline \mu_\ve^{3\delta_0-2\delta_0^2+\smallo\(1\)}\)= \smallo\(\frac{\delta_\ve \ln\overline{\gamma}_\ve}{\overline{\gamma}_\ve^5}\),
\end{equation}
where in the last line, we used that $\delta_\ve=\overline \mu_\ve^{\delta_1+1/2}$ and $3\delta_0-2\delta_0^2>\delta_1+1/2$.

We now move on to the estimate of $(B)$. Integration by parts and using that $U_{\ve,\gamma,\tau}$ is an exact solution outside the balls $B\(\overline{\tau_m},r_\ve+r_\ve^2\)$ give
\begin{align}\label{EqBinitial}
(B)_j&= \theta_j\int_\Omega \[\Delta Z_{0,i}- f_\ve'\(U_{\ve,\gamma,\theta}\)Z_{0,i}\]B_{\ve,\gamma_j,\tau_j}dx\nonumber\\
&=\theta_j\int_\Omega \partial_{\gamma_i}\[\Delta U_{\ve,\gamma,\tau}- f_\ve\(U_{\ve,\gamma,\theta}\)\]B_{\ve,\gamma_j,\tau_j}dx\nonumber\allowdisplaybreaks\\
&=\theta_j\sum_{m=1}^k\int_{B\(\overline{\tau_m},r_\ve\)} \partial_{\gamma_i}\[\Delta U_{\ve,\gamma,\tau}- f_\ve\(U_{\ve,\gamma,\theta}\)\]B_{\ve,\gamma_j,\tau_j}dx\nonumber\allowdisplaybreaks\\
&\quad +\theta_j\sum_{m=1}^k\int_{\Omega^m_\ve}\partial_{\gamma_i}\[\Delta U_{\ve,\gamma,\tau}- f_\ve\(U_{\ve,\gamma,\theta}\)\]B_{\ve,\gamma_j,\tau_j}dx\nonumber\\
&=:\sum_{m=1}^k\[(B)_{jm}+(B')_{jm}\].
\end{align}
Using the same notations as in \eqref{EqUve}--\eqref{EqRive} and using \eqref{expUR}, which gives
\begin{equation}\label{f'BR}
f'\(B+R\)=f'\(B\)+\bigO\(B^3\left|R\right|\exp\(B^2\)\)
\end{equation}
with
$$B=\overline B_m=\overline B_{\ve, \gamma_m,\tau_m},\quad R=R_m= E^{\(m\)}_{\ve,\gamma,\tau}+F^{\(m\)}_{\ve,\gamma,\tau},$$
on $B\(\overline{\tau_m},r_\ve\)$, we can now write
\begin{align}\label{f'BR2}
&\partial_{\gamma_i}\[\Delta U_{\ve,\gamma,\tau}- f_\ve\(U_{\ve,\gamma,\tau}\)\]=\partial_{\gamma_i}\[\Delta \overline B_m\]-f_\ve'\(U_{\ve,\gamma,\tau}\)\partial_{\gamma_i}\[U_{\ve,\gamma,\tau}\]\nonumber\\
&\qquad=\partial_{\gamma_i}\[\Delta \overline B_m\]-\big[ f_\ve'\(\overline B_m\)+\bigO\big( \gammae^3\exp\big(\overline B_m^2\big) \left|R_m\right|\big)\big]\partial_{\gamma_i}\[\overline B_m + R_m\]\nonumber\allowdisplaybreaks\\
&\qquad=\partial_{\gamma_i}\underbrace{\[\Delta \overline B_m-\lambda_\ve h_\ve\(\overline{\tau_m}\)f\(\overline B_m\)\]}_{=0}+\bigO\big( \gammae^3\exp\big(\overline B_m^2\big)\(\left|x-\overline{\tau_m}\right|+ \left|R_m\right|\)\partial_{\gamma_i}\[B_m\]\big)\nonumber\\
&\hspace{157pt}-\big[f'_\ve\(\overline B_m\) + \bigO\big( \gammae^3\exp\big(\overline B_m^2\big) \left|R_m\right|\big)\big]\partial_{\gamma_i}\[R_m\],
\end{align}
where we have also bound $h_\ve-h_\ve\(\overline{\tau_m}\)=\bigO\(\left|x-\overline{\tau_m}\right|\)$.
Expanding $\partial_{\gamma_i}\[R_m\]$ as in \eqref{EqZ0i1}, we then get
\begin{multline*}
(B)_{jm}=-\theta_j \int_{B\(\overline{\tau_m},r_\ve\)} f'_\ve\(\overline B_{\ve, \gamma_m,\overline{\tau_m}}\)B_{\ve,\gamma_j,\tau_j}\partial_{\gamma_i}\[R_m\]dx\nonumber\\
+\bigO\(\left|\theta_j\right|\gammae^4\int_{B\(\overline{\tau_m},r_\ve\)} \exp\big(\overline B_{\ve, \gamma_m,\overline{\tau_m}}^2\big)\(\delta_\ve+\left|x-\overline{\tau_m}\right|\)dx\),
\end{multline*}
hence 
$$(B)_{jm}=-\theta_j \partial_{\gamma_i}\big[E^{\(m\)}_{\ve,\gamma,\tau}\big]\int_{B\(\overline{\tau_m},r_\ve\)} f'_\ve\(\overline B_{\ve, \gamma_m,\overline{\tau_m}}\) B_{\ve,\gamma_j,\tau_j} dx  + \smallo\(\frac{\delta_\ve \ln\overline{\gamma}_\ve}{\overline{\gamma}_\ve^5}\).$$

Together with \eqref{Pr12Eq1}, for $j=m$, we obtain
\begin{align*}
(B)_{jj}&=-8\pi \gammae \theta_j \partial_{\gamma_i}\big[E^{\(j\)}_{\ve,\gamma,\tau}\big] +\smallo\(\frac{\delta_\ve \ln\overline{\gamma}_\ve}{\overline{\gamma}_\ve^5}\),
\end{align*}
while, observing that $B_{\ve,\gamma_j,\tau_j}=\bigO\(1\)$ on $B\(\overline{\tau_m},r_\ve\)$ if $j\ne m$, we get
$$(B)_{jm}= \smallo\(\frac{\delta_\ve \ln\overline{\gamma}_\ve}{\overline{\gamma}_\ve^5}\),\quad\text{for }j\ne m.$$
As for $(B')_{jm}$, similarly as in \eqref{EqA'final}, we can bound with \eqref{Estannuli} and Proposition~\ref{Pr13}
\begin{multline}\label{EqB'}
(B')_{jm}=\int_{\Omega^m_\ve}\(\chi_{\ve,\tau}-1\)\partial_{\gamma_i}\[f_\ve\(U_{\ve,\gamma,\tau}\)\]B_{\ve,\gamma_j,\tau_j}dx\\
=\bigO\(\gammae \int_{\Omega^m_\ve}\left|f_\ve'\(U_{\ve,\gamma,\tau}\)\right|\left|Z_{0,j,\ve,\gamma,\tau}\right|dx\)=\bigO\(\overline\mu_\ve^{3\delta_0-2\delta_0^2+\smallo\(1\)}\)=\smallo\(\frac{\delta_\ve \ln\overline{\gamma}_\ve}{\overline{\gamma}_\ve^5}\).
\end{multline}
Hence, finally, summing over $m$ and $j$, we obtain
\begin{equation}\label{EqBfinal}
(B)=\sum_{j=1}^k(B)_j=-8\pi \gammae\sum_{j=1}^k \theta_j \partial_{\gamma_i}\big[E^{\(j\)}_{\ve,\gamma,\tau}\big] +\smallo\(\frac{\delta_\ve \ln\overline{\gamma}_\ve}{\overline{\gamma}_\ve^5}\).
\end{equation}

We now estimate the term $(C)$. Similar to \eqref{EqBinitial}, integration by parts and Proposition~\ref{Pr4} give
\begin{align*}
(C)&=\int_\Omega \[\Delta Z_{0,i}- f_\ve'\(U_{\ve,\gamma,\theta}\)Z_{0,i}\]\Phi_{\ve,\gamma,\tau,\theta}dx\nonumber\allowdisplaybreaks\\
&=\sum_{j=1}^k \int_{B\(\overline{\tau_j},r_\ve\)}\partial_{\gamma_i}\[\Delta U_{\ve,\gamma,\tau}- f_\ve\(U_{\ve,\gamma,\tau}\)\]\Phi_{\ve,\gamma,\tau,\theta}dx\nonumber\\
&\quad + \sum_{j=1}^k \int_{\Omega^j_\ve}\partial_{\gamma_i}\[\Delta U_{\ve,\gamma,\tau}- f_\ve\(U_{\ve,\gamma,\tau}\)\]\Phi_{\ve,\gamma,\tau,\theta}dx=:\sum_{j=1}^k\[(C)_j+(C')_j\].
\end{align*}
We can now use \eqref{f'BR2}, and with the same notations, we write
\begin{multline*}
(C)_j=-\int_{B\(\overline{\tau_j},r_\ve\)} f_\ve'\(\overline{B}_{\ve,\gamma_j,\overline{\tau_j}}\) \partial_{\gamma_i}\big[R_j\big]\Phi_{\ve,\gamma,\tau,\theta}dx+\bigO\Bigg(\int_{B\(\overline{\tau_j},r_\ve\)}\gammae^3\exp\big(\overline B_{\ve,\gamma_j,\overline{\tau_j}}^2\big)\\
\times\(\(\left|x-\overline{\tau_j}\right|+\left|R_j\right|\)\left|\partial_{\gamma_i}\[\overline{B}_j\]\right|+\left|R_j\right| \left|\partial_{\gamma_i}\[R_j\]\right|\)\left|\Phi_{\ve,\gamma,\tau,\theta}\right|dx\Bigg).
\end{multline*}
The main term in $(C)_j$ will be
\begin{align*}
(C_1)_j&=-\partial_{\gamma_i}\big[E^{\(j\)}_{\ve,\gamma,\tau}\big]\int_{B\(\overline{\tau_j},r_\ve\)} f_\ve'\(\overline B_{\ve,\gamma_j,\overline{\tau_j}}\) \Phi_{\ve,\gamma,\tau,\theta}dx\\
&=-\partial_{\gamma_i}\big[E^{\(j\)}_{\ve,\gamma,\tau}\big]\int_{B\(\overline{\tau_j},r_\ve\)} 2\gamma_j f_\ve\(\overline B_{\ve,\gamma_j,\overline{\tau_j}}\) \Phi_{\ve,\gamma,\tau,\theta}dx\\
&\quad +\bigO\(\frac{\ln\gammae}{\gammae^2}\int_{B\(\overline{\tau_j},r_\ve\)} |2\gamma_j f_\ve\(\overline B_{\ve,\gamma_j,\overline{\tau_j}}\)-f'_\ve\(\overline B_{\ve,\gamma_j,\overline{\tau_j}}\)| \left| \Phi_{\ve,\gamma,\tau,\theta}\right|dx \)\allowdisplaybreaks\\
&=-2\gamma_j \partial_{\gamma_i}\big[E^{\(j\)}_{\ve,\gamma,\tau}\big]\int_{B\(\overline{\tau_j},r_\ve\)} \Delta B_{\ve,\gamma_j,\tau_j} \Phi_{\ve,\gamma,\tau,\theta}dx\\
&\quad +\bigO\(\gammae \int_{B\(\overline{\tau_j},r_\ve\)} \left|x-\overline\tau_j\right| f\(\overline{B}_{\ve,\gamma_j,\overline{\tau_j}}\) \left| \Phi_{\ve,\gamma,\tau,\theta}\right|dx\)\\
&+\bigO\(\frac{\ln \gammae}{\gammae^2}\int_{B\(\overline{\tau_j},r_\ve\)}\exp\big(\overline B_{\ve,\gamma_j,\overline{\tau_j}}^2\big) |2\gamma_j \overline B_{\ve,\gamma_j,\overline{\tau_j}}-2\overline B_{\ve,\gamma_j,\overline{\tau_j}}^2-1| \left| \Phi_{\ve,\gamma,\tau,\theta}\right|dx \),
\end{align*}
where we used that
$$\lambda_{\ve,j}f\(\overline{B}_{\ve,\gamma_j,\overline{\tau_j}}\)=\Delta \overline B_{\ve,\gamma_j,\overline{\tau_j}}=\Delta B_{\ve,\gamma_j,\tau_j},\quad \text{in }B\(\overline{\tau_j},r_\ve\).$$ Since $\Phi_{\ve,\gamma,\tau,\theta}\perp B_{\ve,\gamma_j,\tau_j}$ in $H^1_0\(\Omega\)$ and $\Delta B_{\ve,\gamma_j,\tau_j}=0$ in $\Omega\setminus B\(\overline{\tau_j},r_\ve\)$, we have
\begin{equation}\label{orthog}
\int_{B\(\overline{\tau_j},r_\ve\)} \Delta B_{\ve,\gamma_j,\tau_j} \Phi_{\ve,\gamma,\tau,\theta}dx=\int_\Omega\<\nabla B_{\ve,\gamma_j,\tau_j},\nabla \Phi_{\ve,\gamma,\tau,\theta}\>dx =0,
\end{equation}
and by Proposition~\ref{Pr10},
$$\gamma_j \overline B_{\ve,\gamma_j,\overline{\tau_j}}-\overline B_{\ve,\gamma_j,\overline{\tau_j}}^2= \bigO\(1+t_{\gamma_j}\(\cdot -\overline{\tau_j}\)\),\quad \text{in }B\(\overline{\tau_j},r_\ve\).$$
so that with a change of variables and Propositions~\ref{Pr8} and~\ref{Pr14}, we get
\begin{align*}
(C_1)_j&=\bigO\(\gammae^2 r_\ve \int_{B\(\overline{\tau_j},r_\ve\)} \exp\big(\overline{B}_{\ve,\gamma_j,\overline{\tau_j}}^2\big) \left| \Phi_{\ve,\gamma,\tau,\theta}\right|dx\)\\
&\quad + \bigO\(\frac{\ln \gammae}{\gammae^2}\int_{B\(\overline{\tau_j},r_\ve\)}\exp\big(\overline B_{\ve,\gamma_j,\overline{\tau_j}}^2\big)\(1+t_{\gamma_j}\(x-\overline{\tau_j}\)\)\left| \Phi_{\ve,\gamma,\tau,\theta}\right|dx \)\allowdisplaybreaks\\
&=\bigO\(r_\ve \left\|\nabla \Phi_{\ve,\gamma,\tau,\theta}\right\|_{L^2}\)+\bigO\(\frac{\ln \gammae}{\gammae^4}\|\nabla \Phi_{\ve,\gamma,\tau\theta}\|_{L^2}\)=\smallo\(\frac{\delta_\ve \ln\gammae}{\gammae^5}\).
\end{align*}
Note that we crucially used the orthogonality condition \eqref{orthog} to gain a factor $\gammae^{-2}$. Again, with a change of variables and Proposition~\ref{Pr14}, we bound
\begin{align*}
(C_2)_j&=-\int_{B\(\overline{\tau_j},r_\ve\)} f_\ve'\(\overline B_{\ve,\gamma_j,\overline{\tau_j}}\)\partial_{\gamma_i}\big[F^{\(j\)}_{\ve,\gamma,\tau}\big] \Phi_{\ve,\gamma,\tau,\theta}dx\\
&=\bigO\(\int_{B\(\overline{\tau_j},r_\ve\)} f_\ve'\(\overline B_{\ve,\gamma_j,\overline{\tau_j}}\)\left|x-\overline{\tau_j}\right| \left|\Phi_{\ve,\gamma,\tau,\theta}\right|dx\)\allowdisplaybreaks\\
&=\bigO\(r_\ve \int_{B\(\overline{\tau_j},r_\ve\)} f_\ve'\(\overline B_{\ve,\gamma_j,\overline{\tau_j}}\) \left|\Phi_{\ve,\gamma,\tau,\theta}\right|dx\)=\bigO\(r_\ve \|\nabla \Phi _{\ve,\gamma,\tau,\theta}\|_{L^2}\)=\smallo\(\frac{\delta_\ve \ln\gammae}{\gammae^5}\).
\end{align*}
Similarly, for some exponent $a>0$ (which plays no role),
\begin{align}\label{EqC}
(C_3)_j&=\bigO\Bigg(\int_{B\(\overline{\tau_j},r_\ve\)}\gammae^3\exp\big(\overline B_{\ve,\gamma_j,\overline{\tau_j}}^2\big)\(\(\left|x-\overline{\tau_j}\right|+\left|R_j\right|\)\left|\partial_{\gamma_i}\[\overline{B}_j\]\right|+\left|R_j\right| \left|\partial_{\gamma_i}\[R_j\]\right|\)\nonumber\\
&\qquad \times\left|\Phi_{\ve,\gamma,\tau,\theta}\right|dx\Bigg)=\bigO\(\(\delta_\ve+r_\ve\)\gammae^a \int_{B\(\overline{\tau_j},r_\ve\)}\exp\big(\overline B_{\ve,\gamma_j,\overline{\tau_j}}^2\big)\left|\Phi_{\ve,\gamma,\tau,\theta}\right|dx\)\nonumber\\
&=\bigO\(\(\delta_\ve+r_\ve\)\gammae^{a-2}\|\nabla \Phi_{\ve,\gamma,\tau,\theta}|\|_{L^2}\)=\smallo\(\frac{\delta_\ve \ln\gammae}{\gammae^5}\).
\end{align}
As for $(C')_j$, in analogy with \eqref{EqB'} (with $\Phi_{\ve,\gamma,\tau}$ instead of $B_{\ve,\gamma_j,\tau_j}$), using \eqref{Estannuli} together with the H\"older and Poincar\'e inequalities, we bound
\begin{align*}
(C')_j&=\bigO\(\int_{\Omega^j_\ve}f'\(U_{\ve,\gamma,\tau}\)\left|\Phi_{\ve,\gamma,\tau,\theta}\right|dx \)
= \bigO\(\|f'\(U_{\ve,\gamma,\tau}\)\|_{L^2(\Omega^j_\ve)}  \|\Phi_{\ve,\gamma,\tau,\theta}\|_{L^2(\Omega)}\)\\
&=\bigO\(\overline{\mu}_\ve^{\frac{1}{2}(3\delta_0-2\delta_0^2)+\smallo(1)} \|\nabla \Phi_{\ve,\gamma,\tau,\theta}\|_{L^2(\Omega)}\)=\smallo\(\frac{\delta_\ve \ln\gammae}{\gammae^5}\).
\end{align*}
Summing over $j$, we arrive at
\begin{equation}\label{EqCfinal}
(C)=\sum_{j=1}^k\[(C_1)_j+(C_2)_j+(C_3)_j+(C')_j\]=\smallo\(\frac{\delta_\ve \ln\gammae}{\gammae^5}\).
\end{equation}

As for $(D)$, recalling that $\left|Z_{0,i}\right|=\bigO\(1\)$, we bound
\begin{align*}
(D)&=\bigO\(\int_{\Omega}\left|U_{\ve,\gamma,\tau}\right|^3 \exp\big(U_{\ve,\gamma,\tau}^2\big)\(\Phi_{\ve,\gamma,\tau,\theta}^2+ \sum_{i=1}^k\theta_i^2\gammae^2\) dx\)\allowdisplaybreaks\\
&=\bigO\(\int_{\Omega}\gammae^3 \exp\big(U_{\ve,\gamma,\tau}^2\big) \Phi_{\ve,\gamma,\tau,\theta}^2dx \) + \sum_{i=1}^k \bigO\( \theta_i^2\gammae^2 \int_{\Omega}\left|U_{\ve,\gamma,\tau}\right|^3 \exp\big(U_{\ve,\gamma,\tau}^2\big) dx\)\\
&=:(D_1)+(D_2).
\end{align*}
We first claim that
\begin{equation}\label{EqD11}
\int_{\Omega} \exp\big(U^2_{\ve,\gamma,\tau}\big)dx =\bigO\(1\).
\end{equation}
Indeed, with the usual decomposition given by \eqref{EqUve}, we get
$$\int_{B\(\overline{\tau_j},r_\ve\)}\exp\big(U_{\ve,\gamma,\tau}^2\big)dx=\bigO\(\int_{B\(\overline{\tau_j},r_\ve\)}\exp\big(\overline B_{\ve,\gamma_j,\overline{\tau_j}}^2\big)dx\)=\bigO\(\frac{1}{\gammae^2}\),$$
thanks to the usual change of variables and Proposition~\ref{Pr12}. Then, summing over $j=1,\dots,k$ and also using \eqref{Pr4Eq3b}, we obtain \eqref{EqD11}. Then we immediately estimate
$$(D_2)= \bigO\(\theta_i^2 \gammae^5\)=\smallo\(\frac{\delta_\ve \ln\gammae}{\gammae^5}\).$$
As for $(D_1)$, from H\"older's inequality and \eqref{Pr4Eq3b}, we have
\begin{align*}
(D_1'):=\int_{\Omega_{r_\ve,\tau}}\gammae^3 \exp\big(U_{\ve,\gamma,\tau}^2\big) \Phi_{\ve,\gamma,\tau,\theta}^2dx&=\bigO\(\gammae^3\big\|\exp\big( U_{\ve,\gamma,\tau}^2\big)\mathbf{1}_{\Omega_{r_\ve,\tau}}\big\|_{L^p}\left\|\Phi_{\ve,\gamma,\tau,\theta}\right\|_{L^{2p'}}^2\)\\
&=\bigO\(\gammae^3 \left\|\nabla \Phi_{\ve,\gamma,\tau,\theta}\right\|_{L^2}^2\)=\smallo\(\frac{\delta_\ve \ln\gammae}{\gammae^5}\),
\end{align*}
where $p$ is sufficiently small and $p'$ is the conjugate exponent of $p$. Moreover, with Proposition~\ref{Pr14}, and the same change of variables used to estimate $(C_1)_j$, we obtain
\begin{align}\label{EqD1''}
(D_1)_j:=\int_{B\(\overline{\tau_j},r_\ve\)} \gammae^3 \exp\big(U_{\ve,\gamma,\tau}^2\big) \Phi_{\ve,\gamma,\tau,\theta}^2dx&=\bigO\(\gammae^3 \int_{B\(\overline{\tau_j},r_\ve\)}\exp\big(\overline B_{\ve,\gamma_j,\overline{\tau_j}}^2\big) \Phi_{\ve,\gamma,\tau,\theta}^2dx \)\nonumber\\
&=\bigO\(\gammae \left\|\nabla \Phi_{\ve,\gamma,\tau,\theta}\right\|_{L^2}^2\)=\smallo\(\frac{\delta_\ve \ln\gammae}{\gammae^5}\).
\end{align}
Summing up, we conclude
\begin{equation}\label{EqDfinal}
(D)=\sum_{j=1}^k(D_1)_j+(D_1')+(D_2)=\smallo\(\frac{\delta_\ve \ln\gammae}{\gammae^5}\).
\end{equation}
Now, putting together \eqref{EqAfinal}, \eqref{EqA'final}, \eqref{EqBfinal}, \eqref{EqCfinal} and \eqref{EqDfinal}, we conclude.
\endproof

\proof[Proof of \eqref{Pr9LemEq2}] We consider now \eqref{bigexpansion} with $Z= B_{\ve,\gamma_i,\tau_i}$ and estimate the terms from $(A)$ to $(D)$. 
Using \eqref{Taylor}, we get
\begin{align*}
(A)_j&=\int_{B\(\overline{\tau_j},r_\ve\)}\underbrace{\[\Delta\overline B_{\ve,\gamma_j,\overline{\tau_j}}-\lambda_{\ve,j}f\(\overline B_{\ve,\gamma_j,\overline{\tau_j}}\)\]}_{=0}B_{\ve,\gamma_i,\tau_i}dx\\
&\quad   - \int_{B\(\overline{\tau_j},r_\ve\)}\lambda_\ve\(h_\ve-h_\ve\(\overline{\tau_j}\)\)f\(\overline B_{\ve,\gamma_j,\overline{\tau_j}}\) B_{\ve,\gamma_i,\tau_i}dx\\
&\quad - \int_{B\(\overline{\tau_j},r_\ve\)} \lambda_\ve h_\ve f'\(\overline B_{\ve,\gamma_j,\overline{\tau_j}}\) R_j B_{\ve,\gamma_i,\tau_i} dx+\bigO\(\int_{B\(\overline{\tau_j},r_\ve\)} \gammae^4\exp\big(\overline B_{\ve,\gamma_j,\overline{\tau_j}}^2\big) R_j^2 \),
\end{align*}
where $R_j$ is as in \eqref{EqRive}. Similarly as in \eqref{eqAib}, we reduce to 
\begin{align*}
(A)_j&=- \int_{B\(\overline{\tau_j},r_\ve\)} \lambda_{\ve,j}f'\(\overline B_{\ve,\gamma_j,\overline{\tau_j}}\) E^{\(j\)}_{\ve,\gamma,\tau} B_{\ve,\gamma_i,\tau_i} dx\\
&\quad+ \bigO\(\int_{B\(\overline{\tau_j},r_\ve\)}\overline\gamma_\ve^4\exp\big(\overline{B}_{\ve,\gamma_i,\overline{\tau_i}}^2\big) \(\left|x-\overline{\tau_j}\right|+\delta_\ve^2\)dx \).
\end{align*}
In the case $j=i$ we use that
\begin{equation}\label{EqAbis1}
B _{\ve,\gamma_i,\tau_i}=\overline{B}_{\ve,\gamma_i,\tau_i}\(1+\bigO\(\frac{\ln\gammae}{\gammae^2}\)\),\quad  \text{in } B\(\overline{\tau_i},r_\ve\),
\end{equation}
(see Claim \ref{Claim1}) and with the usual change of variables, taking \eqref{EqBeBj} and Propositions~\ref{Pr12} and~\ref{Pr13} into account, we obtain
\begin{align*}
(A)_i&=-\(1+\bigO\(\frac{\ln\gammae}{\gammae^2}\)\)E^{\(i\)}_{\ve,\gamma,\tau}\int_{B\(\overline{\tau_i},r_\ve\)} \lambda_{\ve,j}f'\(\overline B_{\ve,\gamma_i,\overline{\tau_i}}\)  \overline{B}_{\ve,\gamma_i,\tau_i}dx +\smallo\(\frac{\delta_\ve}{\gammae^2}\)\allowdisplaybreaks\\
&=-\(1+\bigO\(\frac{\ln\gammae}{\gammae^2}\)\) E^{\(i\)}_{\ve,\gamma,\tau}\int_{B(0,\sqrt{\lambda_{\ve,i}} r_\ve)} f'\(\overline B_{\gamma_i}\) \overline B_{\gamma_i}dx+\smallo\(\frac{\delta_\ve}{\gammae^2}\)\\
&=-\(8\pi+\bigO\(\frac{\ln\gammae}{\gammae^2}\)\) \gammae E^{\(i\)}_{\ve,\gamma,\tau}+\smallo\(\frac{\delta_\ve}{\gammae^2}\)= -8\pi\gammae E^{\(i\)}_{\ve,\gamma,\tau}+\smallo\(\frac{\delta_\ve}{\gammae^2}\).
\end{align*}
For the case $j\ne i$, we use that $B_{\ve,\gamma_i,\tau_i}=\bigO\(\(\ln\gammae\)/\gammae\)$ in $B(\overline{\tau_j},r_\ve)$ and with the same computations, we obtain
$$(A)_j=\bigO\(\frac{\ln\gammae}{\gammae^2}\) |E^{\(j\)}_{\ve,\gamma,\tau}|\int_{B(0,\sqrt{\lambda_{\ve,j}} r_\ve)} f'\(\overline B_{\gamma_j}\)dx+\smallo\(\frac{\delta_\ve}{\gammae^2}\)=\smallo\(\frac{\delta_\ve}{\gammae^2}\),$$
so that summing up we conclude
\begin{equation}\label{EqAbisfinal}
(A)=\sum_{j=1}^k(A)_j=-8\pi \gammae E^{\(i\)}_{\ve,\gamma,\tau} +\smallo\(\frac{\delta_\ve}{\gammae^2}\).
\end{equation}
As for the annuli, similarly as in \eqref{EqA'final}, we bound
\begin{equation}\label{EqA'bisfinal}
(A')=\int_{\cup_{j=1}^k\Omega^j_\ve} \(\chi_{\ve,\tau}-1\) f_\ve\(U_{\ve,\gamma,\tau}\)B_{\ve,\gamma,\tau}dx=\bigO\(\overline{\mu}_\ve^{3\delta_0-2\delta_0^2+\smallo\(1\)}\)=\smallo\(\frac{\delta_\ve}{\gammae^2}\).
\end{equation}

We now turn to the estimate of $(B)$. Using a Taylor expansion, together with \eqref{expUR}, \eqref{EqUve} and \eqref{EqRive}, we write
\begin{align*}
(B)_{jm}&:=\theta_j \int_{B\(\overline{\tau_m},r_\ve\)} \[\Delta  B_{\ve,\gamma_j,\tau_j}-f_\ve'\(U_{\ve,\gamma,\tau}\)B_{\ve,\gamma_j,\tau_j}\] B_{\ve,\gamma_i,\tau_i}dx\allowdisplaybreaks\\
&=\theta_j \int_{B\(\overline{\tau_m},r_\ve\)} \[\delta_{jm}\Delta  \overline{B}_{\ve,\gamma_j,\overline{\tau_j}}-f_\ve'\(\overline B_{\ve,\gamma_m,\overline{\tau_m}}\)B_{\ve,\gamma_j,\tau_j}\] B_{\ve,\gamma_i,\tau_i}dx\\
&\quad +\bigO\(\left|\theta_j\right|\int_{B\(\overline{\tau_m},r_\ve\)}\gammae^3 \exp\big(\overline B_{\ve,\gamma_m,\overline{\tau_m}}^2\big)R_mB_{\ve,\gamma_i,\tau_i}dx\).
\end{align*}
With Propositions~\ref{Pr12} and~\ref{Pr13}, we estimate the last term as $\smallo\(\delta_\ve/\gammae^2\)$.
For $j=m=i$, still with Proposition~\ref{Pr12}, we compute, keeping \eqref{EqAbis1} in mind
\begin{multline*}
\theta_i \int_{B\(\overline{\tau_i},r_\ve\)} \[\lambda_{\ve,i}f\big(\overline{B}_{\ve,\gamma_i,\overline{\tau_i}}\big)-f_\ve'\(\overline B_{\ve,\gamma_i,\overline{\tau_i}}\)B_{\ve,\gamma_i,\tau_i}\] B_{\ve,\gamma_i,\tau_i}dx \\
=-8\pi \theta_i\gammae^2+\bigO\(\left|\theta_i\right|\gammae\)=-8\pi \theta_i\gammae^2+\smallo\(\frac{\delta_\ve}{\gammae^2}\),
\end{multline*}
while for $j\ne m$, or $j\ne i$ a similar computation based on Proposition~\ref{Pr12} and \eqref{EqAbis1} gives $(B)_{jm}= \smallo\(\delta_\ve/\gammae^2\).$
Considering the integral in $\Omega_{r_\ve,\tau}$, where $\Delta B_{\ve,\gamma_j,\tau_j}=0$ for every $j=1,\dots, k$, we estimate with the help of \eqref{Pr4Eq3b},
\begin{align*}
(B')_{jm}&:=\theta_j \int_{B\(\overline{\tau_m},R_\ve\)\setminus B\(\overline{\tau_m},r_\ve\)} f_\ve'\(U_{\ve,\gamma,\tau}\)B_{\ve,\gamma_j,\tau_j} B_{\ve,\gamma_i,\tau_i}dx\\
&=\bigO\(\left|\theta_j\right|\gammae^3\int_{B\(\overline{\tau_m},R_\ve\)\setminus B\(\overline{\tau_m},r_\ve\)}  \exp\big(U_{\ve,\gamma,\tau}^2\big)dx\)=\smallo\(\frac{\delta_\ve}{\gammae^2}\),
\end{align*}
where $R_\ve=\exp\(-\gammae\)$ and, since $B_{\ve,\gamma_j,\tau_j}=\bigO\(1\)$ in $\Omega\setminus B\(\overline{\tau_j},r_\ve\)$, still with \eqref{Pr4Eq3b}, we get
\begin{align*}
(B'')_{jm}&:=\theta_j \int_{\Omega_{R_\ve,\tau}} f_\ve'\(U_{\ve,\gamma,\tau}\)B_{\ve,\gamma_j,\tau_j} B_{\ve,\gamma_i,\tau_i}dx=\bigO\(\left|\theta_j\right|\)=\smallo\(\frac{\delta_\ve}{\gammae^2}\),
\end{align*}
In conclusion, we have proven that
\begin{equation}\label{EqBbisfinal}
(B)=\sum_{j=1}^k\[(B)_{jm}+(B')_{jm}+(B'')_{jm}\]=-8\pi \theta_i\gammae^2+\smallo\(\frac{\delta_\ve}{\gammae^2}\).
\end{equation}

To estimate $(C)$, we integrate by parts to obtain
\begin{align*}
(C)&=\int_\Omega \[\Delta B_{\ve,\gamma_i,\tau_i}-f_\ve'\(U_{\ve,\gamma,\tau}\)B_{\ve,\gamma_i,\tau_i} \]\Phi_{\ve,\gamma,\tau,\theta} dx=-\int_\Omega f_\ve'\(U_{\ve,\gamma,\tau}\)B_{\ve,\gamma_i,\tau_i} \Phi_{\ve,\gamma,\tau,\theta} dx,
\end{align*}
where we also used that $\Phi_{\ve,\gamma,\tau,\theta} \perp B_{\ve,\gamma_i,\tau_i}$ in $H^1_0$. Using \eqref{f'BR}, we write
\begin{align*}
(C_1)_j&:=-\int_{B\(\overline{\tau_j},r_\ve\)} f_\ve'\(U_{\ve,\gamma,\tau}\)B_{\ve,\gamma_i,\tau_i} \Phi_{\ve,\gamma,\tau,\theta} dx\allowdisplaybreaks\\
&=-\int_{B\(\overline{\tau_j},r_\ve\)} f_\ve'\(\overline B_{\ve,\gamma_j,\overline{\tau_j}}\)B_{\ve,\gamma_i,\tau_i} \Phi_{\ve,\gamma,\tau,\theta} dx\\
&\quad +\underbrace{\bigO\(\int_{B\(\overline{\tau_j},r_\ve\)}\gammae^3 \exp\big(\overline B_{\ve,\gamma_j,\overline{\tau_j}}^2\big)|R_jB_{\ve,\gamma_i,\tau_i}\Phi_{\ve,\gamma,\tau,\theta}|\) }_{=\smallo\(\delta_\ve/\gammae^2\)\text{ as in \eqref{EqC}}}.
\end{align*}
Then, recalling that $f'\(s\)=\(1+2s^2\)\exp\(s^2\)=2sf\(s\)+\exp\(s^2\)$, which gives 
\begin{align*}
(C_1)_j&=-2\int_{B\(\overline{\tau_j},r_\ve\)}\overline B_{\ve,\gamma_j,\overline{\tau_j}}f_\ve\(\overline B_{\ve,\gamma_j,\overline{\tau_j}}\)B_{\ve,\gamma_i,\tau_i}\Phi_{\ve,\gamma,\tau,\theta}dx\\
&\quad +\bigO\(\gammae \int_{B\(\overline{\tau_j},r_\ve\)} \exp\big(\overline B_{\ve,\gamma_j,\overline{\tau_j}}^2\big)\left|\Phi_{\ve,\gamma,\tau,\theta}\right|dx \) +\smallo\(\frac{\delta_\ve}{\gammae^2}\),
\end{align*}
and using Proposition~\ref{Pr14}, we obtain
$$\gammae \int_{B\(\overline{\tau_j},r_\ve\)} \exp\big(\overline B_{\ve,\gamma_j,\overline{\tau_j}}^2\big)\left|\Phi_{\ve,\gamma,\tau,\theta}\right|dx= \bigO\(\frac{\left\|\nabla \Phi_{\ve,\gamma,\tau,\theta}\right\|_{L^2}}{\gammae}\)=\smallo\(\frac{\delta_\ve}{\gammae^2}\).$$
Using Proposition~\ref{Pr14} again, we simplify
\begin{align*}
&(C_1)_i=2\gamma_i^2 \int_{B\(\overline{\tau_i},r_\ve\)}\lambda_\ve h_\ve \(\overline{\tau_i}\)f\(\overline B_{\ve,\gamma_i,\overline{\tau_i}}\)\Phi_{\ve,\gamma,\tau,\theta}dx\\
&\quad\qquad +\bigO\(\gammae^2\int_{B\(\overline{\tau_i},r_\ve\)}\left|x-\overline{\tau_i}\right|f\(\overline B_{\ve,\gamma_i,\overline{\tau_i}}\)\left|\Phi_{\ve,\gamma,\tau,\theta}\right|dx\)\allowdisplaybreaks\\
&\quad\qquad +\bigO\(\int_{B\(\overline{\tau_i},r_\ve\)}\(1+t_{\gamma_i}\(x-\overline{\tau_i}\)\)f\(\overline B_{\ve,\gamma_i,\overline{\tau_i}}\)\left|\Phi_{\ve,\gamma,\tau,\theta}\right|dx\) +\smallo\(\frac{\delta_\ve}{\gammae^2}\)\\
&\quad=2\gamma_i^2 \int_{B\(\overline{\tau_i},r_\ve\)}\Delta \overline B_{\ve,\gamma_i,\overline{\tau_i}}\Phi_{\ve,\gamma,\tau,\theta}dx+\bigO\(  \(r_\ve +\frac{1}{\gammae^2}\) \left\|\nabla \Phi_{\ve,\gamma,\tau,\theta}\right\|_{L^2}\) +\smallo\(\frac{\delta_\ve}{\gammae^2}\),
\end{align*}
where the last integral vanishes thanks to $\Delta \overline B_{\ve,\gamma_i,\overline{\tau_i}}= \Delta B_{\ve,\gamma_i,\tau_i}$ and to the condition $B_{\ve,\gamma_i,\tau_i}\perp \Phi_{\ve,\gamma,\tau,\theta}$. A similar computation holds on $B\(\overline{\tau_j},r_\ve\)$, where we can use that $B_{\ve,\gamma_i,\tau_i}=\bigO\(1\)$ if $j\ne i$. Hence
$$(C_1)=\sum_{j=1}^k(C_1)_j= \smallo\(\frac{\delta_\ve}{\gammae^2}\).$$
With \eqref{Pr4Eq3b} and the H\"older and Poincar\'e inequalities, we now bound
\begin{align*}
(C_2)_j&:=-\sum_{j=1}^k\int_{B\(\overline{\tau_j},R_\ve\)\setminus B\(\overline{\tau_j},r_\ve\)}f_\ve'\(U_{\ve,\gamma,\tau}\)B_{\ve,\gamma_i,\tau_i} \Phi_{\ve,\gamma,\tau,\theta} dx\allowdisplaybreaks\\
&=\sum_{j=1}^k\bigO\(\gammae^3\big\|\exp\big(U_{\ve,\gamma,\tau}^2\big)\mathbf{1}_{B\(\overline{\tau_j},R_\ve\)\setminus B\(\overline{\tau_j},r_\ve\)}\big\|_{L^p}\left\| \Phi_{\ve,\gamma,\tau,\theta}\right\|_{L^{p'}} \)\\
&=\bigO\( \frac{\left\|\nabla \Phi_{\ve,\gamma,\tau,\theta}\right\|_{L^2}}{\gammae^{a-3}}\)=\smallo\(\frac{\delta_\ve}{\gammae^2}\),
\end{align*}
upon choosing $a>3$. Again with  \eqref{Pr4Eq3b} and the H\"older and Poincar\'e inequalities,  and observing that $U_{\ve,\gamma,\tau}=\bigO\(1\)$ in $\Omega_{R_\ve,\tau}$,
we get
\begin{multline*}
(C_3):=-\int_{\Omega_{R_\ve,\tau}}f_\ve'\(U_{\ve,\gamma,\tau}\)B_{\ve,\gamma_i,\tau_i} \Phi_{\ve,\gamma,\tau,\theta} dx\\
=\bigO\(\big\|\exp\big(U_{\ve,\gamma,\tau}^2\big)B_{\ve,\gamma_i,\tau_i}\mathbf{1}_{\Omega_{R_\ve,\tau}}\big\|_{L^p}\left\|\Phi_{\ve,\gamma,\tau,\theta}\right\|_{L^{p'}}\)=\bigO\(\frac{\left\|\nabla \Phi_{\ve,\gamma,\tau,\theta}\right\|_{L^2}}{\gammae}\)=\smallo\(\frac{\delta_\ve}{\gammae^2}\),
\end{multline*}
where $p$ is sufficiently small and $p'$ is the conjugate exponent of $p$. Adding up, we conclude
\begin{equation}\label{EqCbisfinal}
(C)=(C_1)+(C_2)+(C_3)=\smallo\(\frac{\delta_\ve}{\gammae^2}\).
\end{equation}

Finally the estimate
\begin{equation}\label{EqDbisfinal}
(D)=\smallo\(\frac{\delta_\ve}{\gammae^2}\)
\end{equation}
follows exactly as the analog estimate in the proof of \eqref{Pr9LemEq1} (replacing $Z_{0,i,\ve,\gamma,\tau}$ by $B_{\ve,\gamma_i,\tau_i})$, since all the terms contain $\theta_i^2$ or $\left\|\nabla \Phi_{\ve,\gamma,\tau,\theta}\right\|_{L^2}^2$, which actually allows an estimate of the form $(D)=\bigO\(\delta_\ve/\gammae^a\)$ for every $a\ge 0$. Now, putting together \eqref{EqAbisfinal}, \eqref{EqA'bisfinal}, \eqref{EqBbisfinal}, \eqref{EqCbisfinal} and \eqref{EqDbisfinal}, we conclude.
\endproof

\proof[Proof of \eqref{Pr9LemEq3}] 
We now use \eqref{bigexpansion} with $Z=Z_{1,i,\ve,\gamma,\tau}$, and again we need to estimate the terms from $(A)$ to $(D)$. 

We start with some estimates of $Z_{1,i}:=Z_{1,i,\ve,\gamma,\tau}$. From Claim~\ref{Claim1}, we have
\begin{align*}
\partial_{\tau_i}\[B_{\ve,\gamma_i,\tau_i}\]&=\partial_{\tau_i}\[\overline B_{\ve,\gamma_i,\overline{\tau_i}}\]-\partial_{\tau_i}\[C_{\ve,\gamma_i,\tau_i}\]+ \partial_{\tau_i}\[A_{\ve,\gamma_i,\tau_i}H(\cdot,\overline{\tau_i})\]\\
&=\partial_{\tau_i}\[\overline B_{\ve,\gamma_i,\overline{\tau_i}}\]+\bigO\(\frac{1}{\gammae}\),\quad \text{in }  B\(\overline \tau_i,r_\ve\).
\end{align*}
Now, recalling \eqref{EqUegt} and using \eqref{Pr4Eq3}, we write
\begin{equation}\label{EqZ1iar}
Z_{1,i}=\partial_{\tau_i}\[B_{\ve,\gamma_i,\tau_i}\] +\partial_{\tau_i}\[\Psi_{\ve,\gamma,\tau}\]
=\partial_{\tau_i}\[\overline B_{\ve,\gamma_i,\overline{\tau_i}}\]+\bigO\(\frac{1}{\gammae}\)=:Z_{1,i}^a+Z_{1,i}^r,
\end{equation}
in $B\(\overline{\tau_i},r_\ve\)$, and with \eqref{EqBeBj} and Proposition~\ref{Pr10}, we estimate
\begin{equation}\label{EqZ1i}
Z^a_{1,i}=\frac{2}{\gamma_i} \frac{\lambda_{\ve,i}\(x_1-\tau_i\)}{\mu_i^2+\lambda_{\ve,i}\left|x-\overline{\tau_i}\right|^2}+\bigO\(\frac{1}{\gammae^3}\frac{1}{\mu_i+\left|x-\overline{\tau_i}\right|}\), \quad\text{in }B\(\overline{\tau_i},r_\ve\),
\end{equation}
where $\mu_i:=\mu_{\gamma_i}=\overline{\mu}_\ve^{1+\smallo\(1\)}$ is given by \eqref{defmugamma},
while directly from the definition of $B_{\ve,\gamma_i,\tau_i}$, Claim~\ref{Claim1} and \eqref{Pr4Eq3}, we also obtain
\begin{equation}\label{EqZ1i2}
Z_{1,i}= \frac{(2+\smallo(1))\(x_1-\tau_i\)}{\gammae \left|x-\overline{\tau_i}\right|^2}+\bigO\(\frac{1}{\gammae}\),\quad \text{in }\Omega\setminus B\(\overline{\tau_i},r_\ve\),
\end{equation}
which can be specialized to 
\begin{equation}\label{EqZ1i3}
Z_{1,i}= \bigO\(\frac{1}{\gammae d_\ve}\),\quad \text{in } B\(\overline{\tau_j},r_\ve\),\quad\text{for }j\ne i.
\end{equation}
Let us also write from \eqref{Pr6Eq0},
\begin{equation}\label{EqFiegt}
F^{\(i\)}_{\ve,\gamma,\tau}\(x\)=\Lambda^{\(i\)}_{\ve,\tau}\(x_1-\tau_i\)+\smallo\(\frac{\left|x-\overline{\tau_i}\right|}{\gammae d_\ve}\)\quad\text{in } B\(\overline{\tau_i},r_\ve\),
\end{equation}
where
$$\Lambda^{\(i\)}_{\ve,\tau}:= a_0 l \tau_i^{l-1}-\frac{2}{\gammae}\sum_{j\ne i}\frac{1}{\tau_i-\tau_j}=\bigO\(\frac{1}{\gammae d_\ve}\).$$

With the help of \eqref{EqUve}, as in \eqref{EqAj2}, we can write
\begin{align*}
&(A)_j= -\int_{B\(\overline{\tau_j},r_\ve\)}\lambda_\ve\(h_\ve -h_\ve\(\overline{\tau_j}\)\)f\(\overline{B}_{\ve,\gamma_j,\overline{\tau_j}}\)Z_{1,i}dx \\
&-\int_{B\(\overline{\tau_j},r_\ve\)}\lambda_{\ve,j} f'\(\overline{B}_{\ve,\gamma_j,\overline{\tau_j}}\)R_j Z_{1,i}dx-\int_{B\(\overline{\tau_j},r_\ve\)}\lambda_\ve \(h_\ve-h_\ve\(\overline{\tau_j}\)\) f'\(\overline{B}_{\ve,\gamma_j,\overline{\tau_j}}\)R_j Z_{1,i}dx \\
&+\bigO\(\int_{B\(\overline{\tau_j},r_\ve\)}\overline\gamma_\ve^3\exp \big(\overline{B}_{\ve,\gamma_j,\overline{\tau_j}}^2\big) R_j^2\left|Z_{1,i}\right|dx \)=:(A_1)_j+(A_2)_j+(A_3)_j+(A_4),
\end{align*}
where $R_j$ is as in \eqref{EqRive}. 

We start with the main order term, which turns out to be the one involving $F^{\(i\)}_{\ve,\gamma,\tau}$ and which we write, using \eqref{EqZ1i} and \eqref{EqFiegt}, as
\begin{align*}
(A_2^F)_i&:=-\int_{B\(\overline{\tau_i},r_\ve\)}  f'_\ve\big(\overline{B}_{\ve,\gamma_i,\overline{\tau_i}}\big)F^{\(i\)}_{\ve,\gamma,\tau} Z_{1,i}dx\allowdisplaybreaks\\
&=-\frac{\Lambda_{\ve,\tau}^{\(i\)}}{\gamma_i} \int_{B\(\overline{\tau_i},r_\ve\)} \lambda_{\ve,i}f'\big(\overline{B}_{\ve,\gamma_i,\overline{\tau_i}}\big)\bigg(\frac{2\lambda_{\ve,i}\(x_1-\tau_i\)^2}{\mu_i^2+\lambda_{\ve,i}\left|x-\overline{\tau_i}\right|^2}+\bigO\(\frac{1}{\gammae}\) \bigg)dx \\
&\quad+\smallo\(\frac{1}{\gammae^2d_\ve}\int_{B\(\overline{\tau_i},r_\ve\)} \lambda_{\ve,i}f'\big(\overline{B}_{\ve,\gamma_i,\overline{\tau_i}}\big) dx\).
\end{align*}
With the usual change of variables $\sqrt{\lambda_{\ve,i}}\(x-\overline{\tau_i}\)=y$ and using \eqref{EqBeBj} and Proposition~\ref{Pr12} together with $\gamma_i=\gammae\(1+\smallo\(1\)\)$ and $\Lambda_{\ve,\tau}^{\(i\)}=\bigO\(1/\(\gammae d_\ve\)\)$, we get
\begin{align*}
(A_2^F)_i&=-\frac{\Lambda_{\ve,\tau}^{\(i\)}}{\gamma_i} \int_{B(0,\sqrt{\lambda_{\ve,i}}r_\ve)}f'\(\overline B_{\gamma_i}\)\bigg(\frac{2y_1^2}{\mu_i^2+\left|y\right|^2}+\smallo\(1\) \bigg)dy+ \smallo\(\frac{1}{\gammae^2d_\ve}\)\\
&=-\frac{(4\pi+\smallo(1))\Lambda_{\ve,\tau}^{\(i\)}}{\gamma_i}+\smallo\(\frac{1}{\gammae^2d_\ve}\)=-\frac{4\pi\Lambda_{\ve,\tau}^{\(i\)}}{\gammae} +\smallo\(\frac{1}{\gammae^2 d_\ve}\), 
\end{align*}
For $j\ne i$, using \eqref{EqZ1i3} and Proposition~\ref{Pr13}, we get
\begin{align*}
(A_2^F)_j&= \bigO\(\int_{B\(\overline{\tau_j},r_\ve\)} f'\(\overline{B}_{\ve,\gamma_j,\overline{\tau_j}}\)\big|F^{\(j\)}_{\ve,\gamma,\tau}\big|\left|Z_{1,i}\right|dx\)\allowdisplaybreaks\\
&=\bigO\(\int_{B\(\overline{\tau_j},r_\ve\)} f'\(\overline{B}_{\ve,\gamma_j,\overline{\tau_j}}\) \frac{\left|x-\overline{\tau_j}\right|}{\gammae d_\ve}\frac{1}{\gammae d_\ve}dx\)\\
&=\bigO\(\frac{1}{\gammae^2d_\ve^2}\int_{B(0,\sqrt{\lambda_{\ve,j}}r_\ve)}f'\(\overline{B}_{\gamma_j}\) \left|y\right|dy \)=\smallo\(\frac{1}{\gammae^2 d_\ve}\).
\end{align*}
Using \eqref{EqZ1iar}, canceling the integral of the anti-symmetric term and using Proposition~\ref{Pr12}, we get
\begin{align*}
(A_2^E)_i&:=-\int_{B\(\overline{\tau_i},r_\ve\)}  f'_\ve\big(\overline{B}_{\ve,\gamma_i,\overline{\tau_i}}\big)E^{\(i\)}_{\ve,\gamma,\tau} (Z_{1,i}^a+Z_{1,i}^r)dx\allowdisplaybreaks \\
&=\bigO\(\frac{\big|E^{\(i\)}_{\ve,\gamma,\tau}\big|}{\gammae}\int_{B\(\overline{\tau_i},r_\ve\)}  f'_\ve\big(\overline{B}_{\ve,\gamma_i,\overline{\tau_i}}\big) dx\)\\
&=\bigO\(\frac{\delta_\ve \ln\gammae}{\gammae^4}\int_{B(0,\sqrt{\lambda_{\ve,i}}r_\ve)}  f'_\ve\(\overline{B}_{\gamma_i}\) dx\)=\smallo\(\frac{1}{\gammae^2 d_\ve}\).
\end{align*}
When $j\ne i$, we have thanks to \eqref{EqZ1i3},
\begin{align*}
(A_2^E)_j&=\bigO\(\int_{B\(\overline{\tau_j},r_\ve\)}  f'_\ve\(\overline{B}_{\ve,\gamma_j,\overline{\tau_j}}\) |E^{\(j\)}_{\ve,\gamma,\tau}|\left|Z_{1,i}\right|dx\)\\
&=\bigO\(\frac{\delta_\ve \ln\gammae}{\gammae^3}\frac{1}{\gammae d_\ve}\int_{B(0,\sqrt{\lambda_{\ve,j}}r_\ve)}  f'_\ve\(\overline{B}_{\gamma_j}\) dx\)=\smallo\(\frac{1}{\gammae^2 d_\ve}\).
\end{align*}

We now estimate $(A_1)$. Using that $h_\ve-h_\ve\(\overline{\tau_i}\)=\bigO\(\left|x-\overline{\tau_i}\right|\)$ in $B\(\overline{\tau_i},r_\ve\)$, by \eqref{EqZ1i}, we have
\begin{equation}\label{EqheZ1i}
\left|(h_\ve-h_\ve\(\overline{\tau_i}\))Z_{1,i}\right|=\bigO\(\frac{1}{\gammae}\), \quad \text{in } B\(\overline{\tau_i},r_\ve\),
\end{equation}
and with Proposition~\ref{Pr12}, we estimate
\begin{align*}
(A_1)_i=\bigO\(\frac{1}{\gammae} \int_{B\(\overline{\tau_i},r_\ve\)} f\big(\overline{B}_{\ve,\gamma_i,\overline{\tau_i}}\big)dx \)&=\bigO\(\frac{1}{\gammae} \int_{B(0,\sqrt{\lambda_{\ve,i}}r_\ve)} f\(\overline{B}_{\gamma_i}\)dy \)\\
&=\bigO\(\frac{1}{\gammae^2}\)=\smallo\(\frac{1}{\gammae^2 d_\ve}\).
\end{align*}
Observe that this says that thanks to $d_\ve=\smallo\(1\)$, the term $\nabla h_\ve\(\overline\tau_i\)$ does not play a role, contrary to what happens when the blow-up points are separated by a finite distance. For $j\ne i$, with \eqref{EqZ1i3}, Proposition~\ref{Pr13} and the usual change of variables we obtain
$$(A_1)_j=\bigO\(\frac{1}{\gammae d_\ve} \int_{B\(\overline{\tau_j},r_\ve\)} f\(\overline{B}_{\ve,\gamma_j,\overline{\tau_j}}\) \left|x-\overline{\tau_j}\right|dx \)=\smallo\(\frac{1}{\gammae^2 d_\ve}\).$$

Similarly, one can bound with \eqref{EqZ1i},
\begin{align*}
(A_3)_i&=\bigO\(\int_{B\(\overline{\tau_i},r_\ve\)}\left|x-\overline{\tau_i}\right|f'\(\overline{B}_{\ve,\gamma_i,\overline{\tau_i}}\)\left|R_i\right|\frac{1}{\gammae \left|x-\overline{\tau_i}\right|}dx \)\\
&=\bigO\(\frac{1}{\gammae}\int_{B\(\overline{\tau_i},r_\ve\)}f'\(\overline{B}_{\ve,\gamma_i,\overline{\tau_i}}\) \(\frac{\delta_\ve\ln\gammae}{\gammae^3}+\frac{\left|x-\overline{\tau_i}\right|}{\gammae d_\ve}\) dx \)=\smallo\(\frac{1}{\gammae^2 d_\ve}\),
\end{align*}
where we also used Propositions~\ref{Pr12} and~\ref{Pr13}. For $j\ne i$, an easier estimate holds, using \eqref{EqZ1i3} instead of \eqref{EqZ1i}, and $h_\ve-h_\ve\(\tau_i\)=\bigO\(r_\ve\)$, so that
$$(A_3)_j=\bigO\(\frac{r_\ve}{\gammae d_\ve}\int_{B\(\overline{\tau_j},r_\ve\)}f'\(\overline{B}_{\ve,\gamma_j,\overline{\tau_j}}\) \(\frac{\delta_\ve\ln\gammae}{\gammae^3}+\frac{\left|x-\overline{\tau_j}\right|}{\gammae d_\ve}\) dx \)=\smallo\(\frac{1}{\gammae^2 d_\ve}\),$$

As for $(A_4)_i$, using that
$\big|E^{\(i\)}_{\ve,\gamma,\tau}\big|=\smallo\(\delta_\ve^2\(\ln\gammae\)^2/\gammae^6\)=\bigO\big(\overline{\mu}_\ve^{2\delta_1+1+\smallo\(1\)}\big),$
and $\gammae^a=\bigO\big(\overline{\mu}_\ve^{\smallo\(1\)}\big)$ for every $a\in \R$, we bound
\begin{align*}
(A_4^E)_i&=\bigO\(\overline{\mu}_\ve^{2\delta_1+1+\smallo\(1\)} \int_{B\(\overline{\tau_i},r_\ve\)}\exp\big(\overline{B}_{\ve,\gamma_i,\overline{\tau_i}}^2\big) \frac{1}{\mu_i+\left|x-\overline{\tau_i}\right|}  dx \)\\
&=\bigO\(\overline\mu_\ve^{2\delta_1+\smallo\(1\)}\int_{B(0,\sqrt{\lambda_{\ve,i}}r_\ve)}\exp\big(\overline B_{\gamma_i}^2\big) dx\)= \bigO\big(\overline\mu_\ve^{2\delta_1+\smallo\(1\)}\big)= \smallo\(\frac{1}{\gammae^2 d_\ve}\),
\end{align*}
and, similarly, for $j\ne i$,
$$(A_4^E)_j=\bigO\(\overline{\mu}_\ve^{2\delta_1+1+\smallo\(1\)}\int_{B\(\overline{\tau_j},r_\ve\)}\exp\big(\overline{B}_{\ve,\gamma_i,\overline{\tau_i}}^2\big)\frac{1}{\gammae d_\ve} dx\) =\smallo\(\frac{1}{\gammae^2 d_\ve}\).$$
Using that $F^{\(j\)}_{\ve,\gamma,\tau}=\bigO\(r_\ve/\(\gammae d_\ve\)\)$, similarly as in the case of $(A_4^E)_j$, we obtain
$(A_4^F)_j=\smallo\(\frac{1}{\gammae^2 d_\ve}\),$
including the case $j=i$. Summing over $j$, we obtain
\begin{equation}\label{EqAterfinal}
(A)=\sum_{j=1}^k (A)_j= -\frac{4\pi \Lambda^{\(i\)}_{\ve,\tau}}{\gammae}+ \smallo\(\frac{1}{\gammae^2 d_\ve}\).
\end{equation}
It remains to show that all the remaining terms are $\smallo\(1/(\gammae^2 d_\ve)\)$.

Let us now estimate $(A')$. By \eqref{Pr4Eq4b},  \eqref{Estannuli}, \eqref{EqZ1i2} and \eqref{EqZ1i3}, we have
\begin{align}\label{EqA'terfinal}
(A')&=\bigO\(\sum_{j=1}^k\int_{\Omega_\ve^j} \left|f_\ve\(U_{\ve,\gamma,\tau}\)\right|\left|Z_{1,i}\right|dx\)=\bigO\(\frac{\overline{\mu}_\ve^{-2\delta_0^2+\smallo\(1\)}r_\ve^2}{\gammae}\)\nonumber\\ &=\bigO\(\overline\mu_\ve^{2\delta_0-2\delta_0^2+\smallo\(1\)}\)=\smallo\(\frac{1}{\gammae^2 d_\ve}\),
\end{align}
absorbing powers of $\gammae$ in the term $\overline\mu_\ve^{\smallo\(1\)}$ and using that $2\delta_0-2\delta_0^2>0$.

We now estimate $(B)$. Since $\Delta B_{\ve,\gamma_j,\tau_j}=\Delta \overline{B}_{\ve,\gamma_j,\overline{\tau_j}}\textbf{1}_{B\(\overline{\tau_j},r_\ve\)}$, we have
$$(B^\dagger)_j:=\theta_j\int_\Omega \Delta B_{\ve,\gamma_j,\tau_j} Z_{1,i}dx =\theta_j\int_{B\(\overline{\tau_j},r_\ve\)}\Delta \overline{B}_{\ve,\gamma_j,\overline{\tau_j}} Z_{1,i}dx.$$
Then for $j\ne i$, together with \eqref{EqZ1i3} we obtain
$$\theta_j\int_{B\(\overline{\tau_j},r_\ve\)}\Delta \overline{B}_{\ve,\gamma_j,\overline{\tau_j}} Z_{1,i}dx =\bigO\(\frac{\left|\theta_j\right|}{\gammae d_\ve}\int_{B\(\overline{\tau_j},r_\ve\)}f\(\overline{B}_{\ve,\gamma_j,\overline{\tau_j}}\)dx\)=\smallo\(\frac{1}{\gammae^2d_\ve}\). $$
For $j=i$, we use the anti-symmetry to obtain 
\begin{multline*}
\theta_i\int_{B\(\overline{\tau_i},r_\ve\)}\Delta \overline{B}_{\ve,\gamma_i,\overline{\tau_i}} Z_{1,i}dx=\theta_i\int_{B\(\overline{\tau_i},r_\ve\)}\lambda_{\ve,i} f\big(\overline{B}_{\ve,\gamma_i,\overline{\tau_i}}\big)\overline{B}_{\ve,\gamma_i,\overline{\tau_i}}\( Z_{1,i}^a+Z_{1,i}^r\)dx\\
=\theta_i\int_{B\(\overline{\tau_i},r_\ve\)}\lambda_{\ve,i} f\big(\overline{B}_{\ve,\gamma_i,\overline{\tau_i}}\big) \overline{B}_{\ve,\gamma_i,\overline{\tau_i}}D_{\tau_i}\Psi_{\ve,\gamma,\tau}dx=\smallo\(\frac{1}{\gammae^2d_\ve}\).
\end{multline*}
In order to estimate the second term in the integral in $(B)_j$, we start with the integral away from the blow-up points, and using \eqref{Pr4Eq3c}, the definition of $B_{\ve,\gamma_j,\tau_j}$ and \eqref{Sec32Eq3}, we get 
\begin{multline*}
(B'_j)=\theta_j\int_{\Omega_{r_\ve,\tau}} f'_\ve\(U_{\ve,\gamma,\tau}\)B_{\ve,\gamma_j,\tau_j}Z_{1,i} dx=\bigO\big(\left|\theta_j\right|\big\|f'_\ve\(U_{\ve,\gamma,\tau}\)Z_{1,i}\mathbf{1}_{\Omega_{r_\ve,\tau}}\big\|_{L^p}\\
\times\big\|B_{\ve,\gamma_j,\tau_j}\mathbf{1}_{\Omega\backslash B\(\overline{\tau_j},r_\ve\)}\big\|_{L^{p'}}\big)=\bigO\(\left|\theta_j\right|\gammae\left\|\frac{1}{\gammae}\ln\frac{C}{|\cdot-\overline{\tau_j}|}\right\|_{L^{p'}}\)=\smallo\(\frac{1}{\gammae^2d_\ve}\),
\end{multline*}
where $p$ is sufficiently small and $p'$ is the conjugate exponent of $p$. It remains to estimate
$$(B)_{jm}:=\theta_j\int_{B\(\overline{\tau_m},r_\ve\)}B_{\ve,\gamma_j,\tau_j}f_\ve'\(U_{\ve,\gamma,\tau}\)Z_{1,i}dx.$$
For $m\ne i$, it easily follows from \eqref{EqZ1i3} and Proposition~\ref{Pr12} that
$$(B)_{jm}=\bigO\(\frac{\left|\theta_j\right|}{d_\ve}  \int_{B\(\overline{\tau_m},r_\ve\)}f_\ve'\(U_{\ve,\gamma,\tau}\)dx \) =\smallo\(\frac{1}{\gammae^2d_\ve}\).$$
The case $m=i$ is more subtle.
Using \eqref{f'BR} to split
\begin{multline}\label{EqBji}
(B)_{ji}=\theta_j\int_{B\(\overline{\tau_i},r_\ve\)}B_{\ve,\gamma_j,\tau_j}f_\ve'\big(\overline{B}_{\ve,\gamma_i,\overline{\tau_i}}\big)Z_{1,i}dx\\
+\bigO\(\left|\theta_j\right|\gammae^3 \int_{B\(\overline{\tau_i},r_\ve\)} \exp\big(\overline{B}_{\ve,\gamma_i,\overline{\tau_i}}^2\big)\left|R_i\right|\left|Z_{1,i}\right| dx\)=:(B_1)_{ji}+(B_2)_{ji}.
\end{multline}
Now, writing
$$B_{\ve,\gamma_i,\tau_i}=B_{\ve,\gamma_i,\tau_i}^s+B_{\ve,\gamma_i,\tau_i}^r,$$
where
\begin{align*}
B_{\ve,\gamma_i,\tau_i}^s&=\overline{B}_{\ve,\gamma_i,\overline{\tau_i}}- C_{\ve,\gamma_i,\tau_i} + A_{\ve,\gamma_i,\tau_i}H\(\overline{\tau_i},\overline{\tau_i}\),\\
B_{\ve,\gamma_i,\tau_i}^r&=A_{\ve,\gamma_i,\tau_i}\(H\(\cdot,\overline{\tau_i}\)-H\(\overline{\tau_i},\overline{\tau_i}\)\)=\bigO\(\frac{|\cdot-\overline{\tau_i}|}{\gammae}\),\quad \text{in }B\(\overline{\tau_i},r_\ve\),
\end{align*}
and also using \eqref{EqZ1iar} and \eqref{EqheZ1i}, we get
\begin{multline}\label{EqBii}
(B_1)_{ii}=\theta_i\int_{B\(\overline{\tau_i},r_\ve\)}B_{\ve,\gamma_i,\tau_i}^s \lambda_{\ve,i} f'\big(\overline{B}_{\ve,\gamma_i,\overline{\tau_i}}\big)Z_{1,i}^adx\\
+\bigO\bigg(\left|\theta_i\right|\int_{B\(\overline{\tau_i},r_\ve\)}f'\big(\overline{B}_{\ve,\gamma_i,\overline{\tau_i}}\big)\big(1+\gammae\underbrace{\left|x-\overline{\tau_j}\right|\left|Z_{1,i}\right|}_{=\bigO\(1/\gammae\)}\big)dx \bigg) =\smallo\(\frac{1}{\gammae^2d_\ve}\),
\end{multline}
where we used skew-symmetry to cancel the first integral, and the usual change of variables to estimate the second one with Proposition~\ref{Pr12}. When $j\ne i$, a similar approach gives analog results, with the splitting of $B_{\ve,\gamma_j,\tau_j}=B_{\ve,\gamma_j,\tau_j}^s+B_{\ve,\gamma_j,\tau_j}^r\(x\)$, where
\begin{align*}
B_{\ve,\gamma_j,\tau_j}^s&=A_{\ve,\gamma_j,\tau_j}G\(\overline{\tau_j},\overline{\tau_i}\),\\
B_{\ve,\gamma_j,\tau_j}^r&=A_{\ve,\gamma_j,\tau_j}\(G\(\overline{\tau_j},x\)-G\(\overline{\tau_j},\overline{\tau_i}\)\) =\bigO\(\frac{\left|x-\overline{\tau_i}\right|}{\gammae d_\ve}\),\quad \text{in }B\(\overline{\tau_i},r_\ve\),
\end{align*}
which allows to cancel the symmetric term and obtain
$$(B_1)_{ji}=\smallo\(\frac{1}{\gammae^2d_\ve}\)$$ 
As for $(B_2)_{ji}$, the term involving $F^{\(i\)}_{\ve,\gamma,\tau}$ can be estimated using a similar approach as in \eqref{EqBii}, since $F^{\(i\)}_{\ve,\gamma,\tau}\(x\)=\bigO\(\left|x-\overline{\tau_i}\right|/(\gammae d_\ve)\)$ in $B\(\overline{\tau_i},r_\ve\)$. In the term involving $E^{\(i\)}_{\ve,\gamma,\tau}=\bigO\(\delta_\ve \ln\gammae / \gammae^3\)$, we use the estimate
\begin{equation}\label{EqZ1iexp}
\left|Z_{1,i}\right|=\bigO\(\frac{1}{\gammae \(\mu_{\gamma_i} +\left|x-\overline{\tau_i}\right|\)}\)=\bigO\(\frac{1}{\overline{\mu}_\ve^{1+\smallo\(1\)}}\)
\end{equation}
to finally obtain 
$$(B_2)_{ji}=\bigO\(\frac{\left|\theta_j\right|\delta_\ve}{\overline\mu_\ve^{1+\smallo\(1\)}}\)=\smallo\(\frac{1}{\gammae^2d_\ve}\).$$
Summing up, we conclude
\begin{equation}\label{EqBterfinal}
(B)=\sum_{j=1}^k\[(B^\dagger)_j+(B')_j+\sum_{m=1}^k(B)_{jm}\]=\smallo\(\frac{1}{\gammae^2 d_\ve}\).
\end{equation}

To bound the term $(C)$, let us start by observing that $\Phi_{\ve,\gamma,\tau,\theta}\perp Z_{1,i}$ implies
$$(C_1):=\int_{\Omega}\Delta \Phi_{\ve,\gamma,\tau,\theta} Z_{1,i}dx=0,$$
so that it remains to bound 
$$(C_2):=- \int_{\Omega}f'_\ve\(U_{\ve,\gamma,\tau}\) \Phi_{\ve,\gamma,\tau,\theta} Z_{1,i}dx.$$
Observe that a rough estimate on $B\(\overline{\tau_i},r_\ve\)$ using $\left|Z_{1,i}\right|=\bigO\( 1/\overline{\mu}_\ve\)$ would lead to an exponentially large error term. Therefore we have to be more subtle and use again the Sobolev--Poincar\'e estimates which follow from $\Phi_{\ve,\gamma,\tau,\theta}\perp  B_{\ve,\gamma_i,\tau_i}$. 
We start by noticing that by \eqref{Pr4Eq3c} and the Sobolev embedding, we have
\begin{multline*}
(C_2^*):=- \int_{\Omega_{r_\ve,\tau}}f'_\ve\(U_{\ve,\gamma,\tau}\) \Phi_{\ve,\gamma,\tau,\theta} Z_{1,i}dx=\bigO\big(\big\|f'\(U_{\ve,\gamma,\tau}\) Z_{1,i}\mathbf{1}_{\Omega_{r_\ve,\tau}}\big\|_{L^p}\\
\times\left\|\Phi_{\ve,\gamma,\tau,\theta}\right\|_{L^{p'}}\big)=\bigO\(\gammae\left\|\nabla \Phi_{\ve,\gamma,\tau,\theta}\right\|_{L^2}\)=\bigO\(\frac{\delta_\ve \ln\gammae}{\gammae}\)=\smallo\(\frac{1}{\gammae^2d_\ve}\).
\end{multline*}
For $j\ne i$, we bound with \eqref{EqZ1i3} and Proposition~\ref{Pr14} (which we can use thanks to \eqref{orthog}),
\begin{align}\label{EqC2j}
(C_2^\dagger)_j&:=-\int_{B\(\overline{\tau_j},r_\ve\)}f'_\ve\(U_{\ve,\gamma,\tau}\)Z_{1,i}\Phi_{\ve,\gamma,\tau,\theta}dx\nonumber\allowdisplaybreaks\\
&=\bigO\(\frac{1}{\gammae d_\ve}\int_{B\(\overline{\tau_j},r_\ve\)}f'_\ve\(U_{\ve,\gamma,\tau}\)\left|\Phi_{\ve,\gamma,\tau,\theta}\right|dx\)\nonumber\\
&=\bigO\(\frac{\left\|\nabla \Phi_{\ve,\gamma,\tau,\theta}\right\|_{L^2}}{\gammae d_\ve}\)=\bigO\(\frac{\delta_\ve\ln\gammae}{\gammae^3 d_\ve}\)=\smallo\(\frac{1}{\gammae^2d_\ve}\).
\end{align}
We are left with $(C_2^\dagger)_i$ which we expand as in \eqref{f'BR}, giving
\begin{multline}\label{EqC2i}
(C_2^\dagger)_i=-\int_{B\(\overline{\tau_i},r_\ve\)}f'_\ve\big(\overline{B}_{\ve,\gamma_i,\overline{\tau_i}}\big) Z_{1,i}\Phi_{\ve,\gamma,\tau,\theta}dx\\
+\bigO\(\gammae^3\int_{B\(\overline{\tau_i},r_\ve\)} \exp\big(\overline{B}_{\ve,\gamma_i,\overline{\tau_i}}^2\big)\big|\big(E^{(i)}_{\ve,\gamma,\tau}+F^{(i)}_{\ve,\gamma,\tau}\big)Z_{1,i}\Phi_{\ve,\gamma,\tau,\theta}\big|dx\)
\end{multline}
The remainder term in \eqref{EqC2i} can be estimated as follows. By \eqref{EqZ1i} and \eqref{EqFiegt}, we get $|F^{(i)}_{\ve,\gamma,\tau}Z_{1,i}|=\bigO\(1/\(\gammae^2 d_\ve\)\)$, and we use Proposition~\ref{Pr14} to obtain an error term of order 
$$\bigO\(\frac{\|\nabla \Phi_{\ve,\gamma,\tau,\theta}\|_{L^2}}{\gammae d_\ve}\)=\smallo\(\frac{1}{\gammae^2d_\ve}\).$$ 
As regards the term involving $E^{(i)}_{\ve,\gamma,\tau}$, using \eqref{EqZ1iexp} and Proposition~\ref{Pr14}, we obtain an error term of order
$$\bigO\(\frac{\delta_\ve \|\nabla \Phi_{\ve,\gamma,\tau,\theta}\|_{L^2}}{\overline{\mu}_\ve^{1+\smallo(1)}}\)=\bigO\(\overline{\mu}_\ve^{2\delta_1+\smallo(1)}\)=\smallo\(\frac{1}{\gammae^2 d_\ve}\).$$
The first integral in \eqref{EqC2i} can be estimated by using \eqref{EqZ1iar} together with the estimate $\big(h_\ve-h_\ve\(\tau_i\)\big)\partial_{\tau_i}\[\overline B_{\ve,\gamma_i,\overline{\tau_i}}\]=\bigO\(1/\gamma_\ve\)$ to obtain
\begin{multline*}
\int_{B\(\overline{\tau_i},r_\ve\)}f'_\ve\big(\overline{B}_{\ve,\gamma_i,\overline{\tau_i}}\big) Z_{1,i}\Phi_{\ve,\gamma,\tau,\theta}dx=\int_{B\(\overline{\tau_i},r_\ve\)}\lambda_{\ve,i}f'\big(\overline{B}_{\ve,\gamma_i,\overline{\tau_i}}\big) \partial_{\tau_i}\[\overline{B}_{\ve,\gamma_i,\overline{\tau_i}}\]\Phi_{\ve,\gamma,\tau,\theta}dx\\
+\bigO\(\frac{1}{\gammae} \int_{B\(\overline{\tau_i},r_\ve\)}f'\big(\overline{B}_{\ve,\gamma_i,\overline{\tau_i}}\big) \left|\Phi_{\ve,\gamma,\tau,\theta}\right|dx\)=:-(C_2^\ddagger)_i+(C_2^r)_i.
\end{multline*}
The remainder term $(C_2^r)_i$ can be handled as in \eqref{EqC2j}, giving
$$(C_2^r)_i=\bigO\(\frac{\left\|\nabla\Phi_{\ve,\gamma,\tau,\theta}\right\|_{L^2}}{\gammae}\)=\smallo\(\frac{1}{\gammae^2d_\ve}\).$$

Then we are left with the term $(C_2^\ddagger)_i$, which is actually more subtle to bound. Let us first rewrite it as
$$(C_2^\ddagger)_i = - \int_{B\(\overline{\tau_i},r_\ve\)}\lambda_{\ve,i}\Delta Z_{1,i}\Phi_{\ve,\gamma,\tau,\theta}dx,$$
using that
$$\Delta Z_{1,i}=\partial_{\tau_i}\[\Delta U_{\ve,\gamma,\tau}\]=\partial_{\tau_i}\[\Delta \overline{B}_{\ve,\gamma_i,\overline{\tau_i}}\]=\lambda_{\ve,i}\partial_{\tau_i}\[f\(\overline{B}_{\ve,\gamma,\overline{\tau_i}}\)\],\quad \text{in }B\(\overline{\tau_i},r_\ve\).$$
In order to estimate $(C_2^\ddagger)_i$ we start by observing that the orthogonality condition $Z_{1,i}\perp \Phi_{\ve,\gamma,\tau,\theta}$ and integration by parts imply
\begin{multline}\label{EqZ1iint}
0=\int_{\Omega}\<\nabla Z_{1,i},\nabla \Phi_{\ve,\gamma,\tau,\theta}\>dx=\int_{B\(\overline{\tau_i},r_\ve\)}\Delta Z_{1,i}\, \Phi_{\ve,\gamma,\tau,\theta}dx\\
+\int_{\Omega\backslash B\(\overline{\tau_i},r_\ve\)}\Delta Z_{1,i}\, \Phi_{\ve,\gamma,\tau,\theta}dx+\int_{\partial B\(\overline{\tau_i},r_\ve\)}\(\partial_\nu Z_{1,i}^{int}-\partial_\nu Z_{1,i}^{ext}\) \Phi_{\ve,\gamma,\tau,\theta}d\sigma.
\end{multline}
Here $\nu$ denotes the exterior normal to $\partial B\(\overline{\tau_i},r_\ve\)$ and
$$Z_{1,i}^{int}:= Z_{1,i}|_{\overline{B\(\overline{\tau_i},r_\ve\)}},\quad Z_{1,i}^{ext}:= Z_{1,i}|_{\Omega\setminus B\(\overline{\tau_i},r_\ve\)}.$$
Note that the boundary integral in \eqref{EqZ1iint} is in general non-zero because $B_{\ve,\gamma_i,\tau_i}$ is $C^1$ but not smooth across $\partial B\(\overline{\tau_i},r_\ve\)$.
Now we reduced the estimate of  $(C_2^\ddagger)_i$ to 
$$(C_2^\ddagger)_i =\int_{\Omega\backslash B\(\overline{\tau_i},r_\ve\)}\Delta Z_{1,i}\, \Phi_{\ve,\gamma,\tau,\theta}dx+\int_{\partial B\(\overline{\tau_i},r_\ve\)}\(\partial_\nu Z_{1,i}^{int}-\partial_\nu Z_{1,i}^{ext}\) \Phi_{\ve,\gamma,\tau,\theta}d\sigma.$$
Now using that
\begin{align*}
\Delta Z_{1,i}&=\partial_{\tau_i}\[\Delta U_{\ve,\gamma,\tau}\]=\partial_{\tau_i}\[\chi_{\ve,\tau}f_\ve\(U_{\ve,\gamma,\tau}\)\]\\
&=\partial_{\tau_i}\[\chi_{\ve,\tau}\]f_\ve\(U_{\ve,\gamma,\tau}\)+\chi_{\ve,\tau}f'_\ve\(U_{\ve,\gamma,\tau}\)\partial_{\tau_i}\[U_{\ve,\gamma,\tau}\],\quad \text{in }B\(\overline{\tau_i},r_\ve\).
\end{align*}
from \eqref{Pr4Eq3c}, we obtain
$$\big\|\Delta Z_{1,i}\mathbf{1}_{\Omega\backslash B\(\overline{\tau_i},r_\ve\)}\big\|_{L^p}=\bigO\(\gammae\) $$
for some $p>1$, hence with the H\"older and Sobolev inequalities
\begin{align*}
\int_{\Omega\backslash B\(\overline{\tau_i},r_\ve\)}\Delta Z_{1,i}\, \Phi_{\ve,\gamma,\tau,\theta}dx&=\bigO\(\big\|\Delta Z_{1,i}\mathbf{1}_{\Omega\backslash B\(\overline{\tau_i},r_\ve\)}\big\|_{L^p}\left\|\Phi_{\ve,\gamma,\tau,\theta}\right\|_{L^{p'}} \)\\
&=\bigO\(\gammae\left\|\nabla\Phi_{\ve,\gamma,\tau,\theta}\right\|_{L^2}\)=\smallo\(\frac{1}{\gammae d_\ve}\).
\end{align*}
Observe that $D_\tau\[\Psi_{\ve,\gamma,\tau}\]\in C^1\(\overline{\Omega}\)$, by elliptic estimates (the function $\chi_{\ve,\tau}$ in Proposition~\ref{Pr4} is smooth), hence we get
\begin{equation}\label{EqZ1iint-ext}
\partial_\nu Z_{1,i}^{int}-\partial_\nu Z_{1,i}^{ext}=\partial_\nu\partial_{\tau_i}\[B_{\ve,\gamma_i,\tau_i}^{int}\]-\partial_\nu\partial_{\tau_i}\[B_{\ve,\gamma_i,\tau_i}^{ext}\]
\end{equation}
Using the definition of $B_{\ve,\gamma_i,\tau_i}$ and \eqref{Sec32Eq5}, we compute
\begin{align*}
\partial_\nu\partial_{\tau_i}\[B_{\ve,\gamma_i,\tau_i}^{int}\]&=\partial_{\tau_i}\[\partial_\nu B_{\ve,\gamma_i,\tau_i}^{int}\]=\partial_{\tau_i}\[\partial_{\nu}\overline{B}_{\ve,\gamma_i,\overline{\tau_i}}\] +\partial_{\tau_i}\[\partial_\nu(A_{\ve,\gamma_i,\tau_i} H\(\cdot, \overline{\tau_i}\))\]\\
&=\partial_{\tau_i}\[\partial_{\nu} \overline{B}_{\ve,\gamma_i,\overline{\tau_i}}\]+\bigO\(\frac{1}{\gammae}\).
\end{align*}
Similarly,
\begin{align*}
\partial_\nu\partial_{\tau_i}\[B_{\ve,\gamma_i,\tau_i}^{ext}\]&=\partial_{\tau_i}\[\frac{A_{\ve,\gamma_i,\tau_i}}{2\pi}\partial_\nu\(\ln\frac{1}{\left|x-\overline{\tau_i}\right|}\)\]+\partial_{\tau_i}\[\partial_\nu\(A_{\ve,\gamma_i,\tau_i} H\(\cdot, \overline{\tau_i}\)\)\]\\
&=-\partial_{\tau_i}\[\frac{A_{\ve,\gamma_i,\tau_i}}{2\pi\left|x-\overline{\tau_i}\right|}\]+\bigO\(\frac{1}{\gammae}\).
\end{align*}
Now in order to compute the difference of the two terms in \eqref{EqZ1iint-ext}, set
$$v_{\ve,\gamma_i,\tau_i}\(x\):=\frac{A_{\ve,\gamma_i,\tau_i}}{2\pi}\ln \frac{1}{\left|x-\overline{\tau_i}\right|}$$
and Note that $v_{\ve,\gamma_i,\tau_i}'\(r_\ve\)=\overline{B}_{\ve,\gamma_i,\overline{\tau_i}}'\(r_\ve\)$ by the definitions in Section~\ref{Sec31}, where with a little abuse of notation, we use the prime to denote the radial derivative from $\overline{\tau_i}$.
Then, with a similar abuse of notation
\begin{align*}
-\overline{B}_{\ve,\gamma_i,\overline{\tau_i}}''\(r_\ve\)-\frac{\overline{B}_{\ve,\gamma_i,\overline{\tau_i}}'\(r_\ve\)}{r_\ve}&=\Delta \overline{B}_{\ve,\gamma_i,\overline{\tau_i}}\(r_\ve\)=\lambda_{\ve,i} f\(\overline{B}_{\ve,\gamma_i,\overline{\tau_i}}\(r_\ve\)\),\\
-v_{\ve,\gamma_i,\tau_i}''\(r_\ve\)-\frac{v_{\ve,\gamma_i,\tau_i}'\(r_\ve\)}{r_\ve}&=\Delta v_{\ve,\gamma_i,\tau_i}\(r_\ve\)=0,
\end{align*}
and subtracting we finally estimate
\begin{align}\label{EqZ1iint-ext2}
\left| \partial_\nu Z_{1,i}^{int}-\partial_\nu Z_{1,i}^{ext}\right|&=\bigO\(\left|\overline{B}_{\ve,\gamma_i,\overline{\tau_i}}''\(r_\ve\)- v_{\ve,\gamma_i,\tau_i}''\(r_\ve\)\right|\)+\bigO\(\frac{1}{\gammae}\)\nonumber\\
&=\bigO\(f\(\overline{B}_{\ve,\gamma_i,\overline{\tau_i}}\(r_\ve\)\)\)+\bigO\(\frac{1}{\gammae}\)=\bigO\(\frac{1}{\overline{\mu}_\ve^{2\delta_0^2+\smallo\(1\)}}\).
\end{align}
We now claim that
\begin{equation}\label{EqPhibordo}
\left\|\Phi_{\ve,\gamma,\tau,\theta}\right\|_{L^1\(\partial B\(\overline{\tau_i},r_\ve\)\)}=\bigO\(r_\ve \sqrt{\ln\frac{1}{r_\ve}}+\left\|\nabla\Phi_{\ve,\gamma,\tau,\theta}\right\|_{L^2}\)=\bigO\big(\overline{\mu}_\ve^{\delta_0+\smallo\(1\)}\big).
\end{equation}
This, together with \eqref{EqZ1iint-ext2} allows to bound
\begin{align*}
\int_{\partial B\(\overline{\tau_i},r_\ve\)}\(\partial_\nu Z_{1,i}^{int}-\partial_\nu Z_{1,i}^{ext}\) \Phi_{\ve,\gamma,\tau,\theta}d\sigma&=\bigO\(\frac{\left\|\Phi_{\ve,\gamma,\tau,\theta}\right\|_{L^1\(\partial B\(\overline{\tau_i},r_\ve\)\)}}{\overline{\mu}_\ve^{2\delta_0^2+\smallo\(1\)}}\)\\
&=\bigO\(\overline{\mu}_\ve^{\delta_0-2\delta_0^2+\smallo\(1\)}\)=\smallo\(\frac{1}{\gammae^2 d_\ve}\).
\end{align*}
This completes the estimates of $(C_2^\ddagger)_i$, hence
\begin{equation}\label{EqCterfinal}
(C)=(C_1)+(C_2^\ast)+\sum_{j=1}^k\[(C_2^\dagger)_j+(C_2^\ddagger)_j\]=\smallo\(\frac{1}{\gammae^2 d_\ve}\).
\end{equation}
In order to prove \eqref{EqPhibordo}, set $\widetilde{\Phi}\(y\):=\Phi_{\ve,\gamma,\tau,\theta}\(\overline{\tau_i}+r_\ve y\)$. We then have
$$\big\|\nabla \widetilde{\Phi}\big\|_{L^2\(B\(0,1\)\)}=\left\|\nabla \Phi_{\ve,\gamma,\tau,\theta}\right\|_{L^2\(B\(\overline{\tau_i},r_\ve\)\)}.$$
By the trace inequality, we get
\begin{equation}\label{EqPhibordo2}
\frac{\left\|\Phi_{\ve,\gamma,\tau,\theta}\right\|_{L^1\(B\(\overline{\tau_i},r_\ve\)\)}}{r_\ve}=\big\|\widetilde{\Phi}\big\|_{L^1\(B\(0,1\)\)}=\bigO\(\big\|\nabla \widetilde{\Phi}\big\|_{L^2\(B\(0,1\)\)}+\left|\frac{1}{\left|B\(0,1\)\right|}\int_{B\(0,1\)} \widetilde{\Phi} dy \right| \),
\end{equation}
since the right-hand side contains a norm equivalent to the $H^1$-norm. Now, by the Jensen and Moser--Trudinger inequalities, we have
\begin{align*}
\exp\(\(\frac{1}{|B\(0,1\)|}\int_{B\(0,1\)} \widetilde{\Phi} dy\)^2 \) &\le \frac{1}{|B\(0,1\)|}\int_{B\(0,1\)}\exp\big(\widetilde{\Phi}^2\big)dy \\
&=\frac{1}{\pi r_\ve} \int_{B\(\overline{\tau_i},r_\ve\)}\exp\big(\Phi_{\ve,\gamma,\tau,\theta}^2\big)dx\le \frac{1}{\pi r_\ve}
\end{align*}
It follows that 
$$\left|\frac{1}{|B\(0,1\)|}\int_{B\(0,1\)} \widetilde{\Phi} dy \right| \le \sqrt{\ln\frac{1}{\pi r_\ve}}$$
and \eqref{EqPhibordo} follows at once from \eqref{EqPhibordo2}. This completes the proof of \eqref{EqCterfinal}.

We finally estimate
\begin{align*}
(D)&=\bigO\(\int_\Omega\left|U_{\ve,\gamma,\tau}\right|^3\exp\big(U_{\ve,\gamma,\tau}^2\big)\(\Phi_{\ve,\gamma,\tau,\theta}^2+ \sum_{j=1}^k \theta_j^2 B_{\ve,\gamma_j,\tau_j}^2\)\left|Z_{1,i}\right|dx\)\allowdisplaybreaks\\
&=\bigO\(\int_{\Omega}\left|U_{\ve,\gamma,\tau}\right|^3 \exp\big(U_{\ve,\gamma,\tau}^2\big) \Phi_{\ve,\gamma,\tau,\theta}^2\left|Z_{1,i}\right|dx \)\\
&\quad+ \sum_{j=1}^k \bigO\( \theta_j^2\gammae^2 \int_{\Omega}\left|U_{\ve,\gamma,\tau}\right|^3 \exp\big(U_{\ve,\gamma,\tau}^2\big)\left|Z_{1,i}\right|dx\)=:(D_1)+\sum_{j=1}^k(D_2)_j.
\end{align*}
For every $j\in\left\{1,\dotsc,k\right\}$, with the rough estimate
$$Z_{1,i}=\bigO\(\mu_{\gamma_i}^{-1}\)=\bigO\big(\overline{\mu}_\ve^{-1+\smallo\(1\)}\big),\quad \text{in }B\(\overline{\tau_j},r_\ve\),$$
we obtain as in \eqref{EqD1''}
\begin{align*}
(D_1)_j&:=\int_{B\(\overline{\tau_j},r_\ve\)}\left|U_{\ve,\gamma,\tau}\right|^3 \exp\big(U_{\ve,\gamma,\tau}^2\big)\Phi_{\ve,\gamma,\tau,\theta}^2\left|Z_{1,i}\right|dx\allowdisplaybreaks\\
&=\bigO\(\frac{\gammae^3}{\overline{\mu}_\ve^{1+\smallo\(1\)}} \(\int_{B\(\overline{\tau_j},r_\ve\)} \exp\big(\overline B_{\ve,\gamma_j,\tau_j}^2\big)\Phi_{\ve,\gamma,\tau,\theta}^2dx\)\)\\
&=\bigO\( \frac{\left\|\nabla \Phi_{\ve,\gamma,\tau,\theta}\right\|_{L^2}^2}{\overline{\mu}_\ve^{1+\smallo\(1\)}}\)=\bigO\(\frac{\delta_\ve^2}{\overline{\mu}_\ve^{1+\smallo\(1\)}}\)=\bigO\big(\overline{\mu}_\ve^{2\delta_1+\smallo\(1\)}\big)=\smallo\(\frac{1}{\gammae^2d_\ve}\).
\end{align*}
Similarly, with \eqref{Pr4Eq3c},
\begin{multline*}
(D_1'):=\int_{\Omega_{r_\ve,\tau}}\left|U_{\ve,\gamma,\tau}\right|^3 \exp\big(U_{\ve,\gamma,\tau}^2\big)\Phi_{\ve,\gamma,\tau,\theta}^2\left|Z_{1,i}\right|dx=\bigO\Big(\gammae\big\|f'\(U_{\ve,\gamma,\tau}\)Z_{1,i} \mathbf{1}_{\Omega_{r_\ve,\tau}}\big\|_{L^p}\\
\times\left\|\Phi_{\ve,\gamma,\tau,\theta}\right\|^2_{L^{2p'}}\Big)=\bigO\(\gammae^2\left\|\nabla\Phi_{\ve,\gamma,\tau,\theta}\right\|_{L^2}^2\)=\smallo\(\frac{1}{\gammae^2 d_\ve}\).
\end{multline*}
As for the terms involving $\theta_j$, we can use the rough estimate $Z_{1,i}=\bigO\big(\overline \mu_\ve^{-1+\smallo\(1\)}\big)$ to get
\begin{align*}
(D_2)_j&:=\theta_j^2\gammae^2\int_{B\(\overline{\tau_m},r_\ve\)}\left|U_{\ve,\gamma,\tau}\right|^3\exp\big(U_{\ve,\gamma,\tau}^2\big)\left|Z_{1,i}\right|dx\\
&=\bigO\(\frac{\theta_j^2\gammae^5}{\overline{\mu}_\ve^{1+\smallo\(1\)}} \(\int_{B\(\overline{\tau_m},r_\ve\)} \exp\big(\overline B_{\ve,\gamma_m,\tau_m}^2\big)dx\)\)
=\bigO\big(\overline{\mu}_\ve^{2\delta_1+\smallo\(1\)}\big)=\smallo\(\frac{1}{\gammae^2d_\ve}\),
\end{align*}
while in $\Omega_{r_\ve,\tau}$, we can use \eqref{Pr4Eq3c} with $p=1$ to obtain 
$$(D_2')_j:=\theta_j^2\gammae^2\int_{\Omega_{r_\ve,\tau}} \left|U_{\ve,\gamma,\tau}\right|^3\exp\big(U_{\ve,\gamma,\tau}^2\big)\left|Z_{1,i}\right|dx=\bigO(\theta_j^2\gammae^4)=\smallo\(\frac{1}{\gammae^2 d_\ve}\).$$
Summing up, we obtain
\begin{equation}\label{EqDterfinal}
(D)=\sum_{j=1}^k(D_1)_j+(D_1')+\sum_{j=1}^k\[(D_2)_j+(D_2')_j\]=\smallo\(\frac{1}{\gammae^2 d_\ve}\).
\end{equation}
Now, \eqref{EqAterfinal}, \eqref{EqA'terfinal}, \eqref{EqBterfinal}, \eqref{EqCterfinal} and \eqref{EqDterfinal} allow us to conclude.
\endproof

\proof[Proof of Proposition~\ref{Pr9}]
We claim that for $\delta$ and  $\ve$ small enough, we can find $(\gamma_\ve,\theta_\ve,\tau_\ve)\in P^k_\ve\(\delta\)$ such that
\begin{align}\label{Pr9Eq1b}
\<\widetilde{R}_{\ve,\gamma_\ve,\tau_\ve,\theta_\ve},Z_{0,i,\varepsilon,\gamma_\ve,\tau_\ve}\>_{H^1_0}=\<\widetilde{R}_{\ve,\gamma_\ve,\tau_\ve,\theta_\ve},B_{\ve,\gamma_\ve,\tau_\ve}\>_{H^1_0}&=\<\widetilde{R}_{\ve,\gamma_\ve,\tau_\ve,\theta_\ve},Z_{1,i,\varepsilon,\gamma_\ve,\tau_\ve}\>_{H^1_0}=0,
\end{align}
for $i=1,\dots,k$, so that
$$\Pi_{\ve,\gamma_\ve,\tau_\ve}\(U_{\ve,\gamma_\ve,\tau_\ve,\theta_\ve}+\Phi_{\ve,\gamma_\ve,\tau_\ve,\theta_\ve}-\Delta^{-1}\(\lambda_\ve h_\ve f\(U_{\ve,\gamma_\ve,\tau_\ve,\theta_\ve}+\Phi_{\ve,\gamma_\ve,\tau_\ve,\theta_\ve}\)\)\)=0,$$
hence, together with Proposition~\ref{Pr8}, $U_{\ve,\gamma_\ve,\tau_\ve,\theta_\ve}+\Phi_{\ve,\gamma_\ve,\tau_\ve,\theta_\ve}$ is a solution to \eqref{Pr9Eq1}.
For every $\tau\in T^k_\ve\(\delta\)$ and $\gamma\in \widehat{\Gamma}^k_\ve\(\tau\)$, let us set $\hat\gamma:=\gamma-\overline{\gamma}_\ve\(\tau\)$ and 
$$\widehat{\Gamma}^k_\ve:=\left\{\hat\gamma=\(\hat\gamma_1,\dotsc,\hat\gamma_k\)\in(0,\infty)^k:\,\left|\hat\gamma_i\right|<\frac{\delta_\ve}{\gammae},\,\forall i\in\{i,\dots,k\}\right\},$$
and define $\(L_\ve, M_\ve, N_\ve\):\widehat P^k_\ve\(\delta\) :=\widehat{\Gamma}^k_\ve \times \Theta^k_\ve\times T^k_\ve\(\delta\) \to \R^{3k}$ as
\begin{align*}
L_{\ve}^i\(\hat\gamma,\tau,\theta\)&:= -\frac{1}{8\pi}\<\widetilde{R}_{\ve,\gamma,\tau,\theta},Z_{0,i,\varepsilon,\gamma,\tau}\>_{H^1_0}\\
&= \sum_{j=1}^k \partial_{\gamma_i}\big[E^{\(j\)}_{\ve,\gamma,\tau}\big]\big(E^{\(j\)}_{\ve,\gamma,\tau} +\theta_j \gammae \big)- \frac{E^{\(i\)}_{\ve,\gamma,\tau}}{2\gammae^2}+\smallo\(\frac{\delta_\ve \ln\overline{\gamma}_\ve}{\overline{\gamma}_\ve^5}\),\allowdisplaybreaks\\
M_{\ve}^i\(\hat\gamma,\tau,\theta\)&:=-\frac{1}{8\pi\gammae}\<\widetilde{R}_{\ve,\gamma,\tau,\theta},B_{\varepsilon,\gamma_i,\tau_i}\>_{H^1_0}=E^{\(i\)}_{\ve,\gamma,\tau}+\theta_i\gammae +\smallo\(\frac{\delta_\ve}{\gammae^3}\),\allowdisplaybreaks\\
N_{\ve}^i\(\hat \gamma,\tau,\theta\)&:=-\frac{\gammae^2 d_\ve}{4\pi}\<\widetilde{R}_{\ve,\gamma,\tau,\theta},Z_{1,i,\varepsilon,\gamma,\tau}\>_{H^1_0}=a_0l\(\frac{\tau_i}{d_\ve}\)^{l-1}-\sum_{j\ne i}\frac{2d_\ve}{\tau_i-\tau_j}+\smallo\(1\)
\end{align*}
for $i=1,\dots k$, where the error terms in the right-hand sides are uniform for $\(\hat\gamma,\tau,\theta\)\in \widehat P^k_\ve\(\delta\)$. (Note that we wrote $\hat\gamma$ in the left-hand side and $\gamma$ instead of $\hat\gamma +\gammae\(\tau\)$ in the right-hand side for simplicity, so that for instance the terms $E^{\(j\)}_{\ve,\gamma,\tau}$ should be read as $E^{\(j\)}_{\ve,\hat \gamma+\gammae\(\tau\),\tau}$)

We claim that
\begin{equation}\label{Pr9Eq3}
\deg\big(\(L_{\ve},M_{\ve},N_{\ve}\),  \widehat P^k_\ve\(\delta\),0\big)\ne 0
\end{equation}
for $\delta$ and $\ve$ small (to be fixed), where $\deg$ denotes the Brouwer degree.
Let us consider the homotopy $\(L_\ve^t, M_\ve^t, N_\ve^t\):\widehat P^k_\ve\(\delta\)\to \R^{3k}$ with $L_\ve^t=\(L_\ve^{t,1},\dotsc,L_\ve^{t,k}\)$, etc. 
defined by
\begin{align*}
L_{\ve}^{t,i}&=\(1-t\) \overline{L}_{\ve}^i+ t L_{\ve}^i,\quad \overline{L}_{\ve}^i:=\sum_{j=1}^k \partial_{\gamma_i}\big[E^{\(j\)}_{\ve,\gamma,\tau}\big]\big(E^{\(j\)}_{\ve,\gamma,\tau} +\theta_j \gammae \big)- \frac{E^{\(i\)}_{\ve,\gamma,\tau}}{2\gammae^2},\allowdisplaybreaks\\
M_{\ve}^{t,i}&=\(1-t\)\overline{M}_{\ve}^i +t M_{\ve}^i,\quad \overline{M}_{\ve}^i:=E^{\(i\)}_{\ve,\gamma,\tau}+\theta_i\gammae,\allowdisplaybreaks\\
N_\ve^{t,i}&=\(1-t\)\overline{N}_{\ve}^i+tN_\ve^i,\quad \overline{N}_\ve^i:=a_0l\(\frac{\tau_i}{d_\ve}\)^{l-1}-\sum_{j\ne i}\frac{2d_\ve}{\tau_i-\tau_j}.
\end{align*}
for $i=1,\dots k$ and $t\in\[0,1\]$. We first show that $\(L_{\ve}^t, M_{\ve}^t,N_\ve^t\)\ne 0$ on $\partial \widehat P^k_\ve\(\delta\)$ for any $t\in [0,1]$ if $\ve>0$ is sufficiently small. Otherwise there would be a sequence $\ve_n \downarrow 0$ (which we still denote by $\ve$), $t_\ve\in [0,1]$ and
\begin{equation}\label{Pr9Eq3b}
\(\hat\gamma_\ve,\theta_\ve,\tau_\ve\)\in \partial \widehat P^k_\ve,\quad\text{i.e. }\hat\gamma_\ve\in\partial\widehat{\Gamma}^k_\ve,\text{ or }\theta_\ve\in\partial\Theta^k_\ve,\text{ or  }\tau_\ve\in\partial T^k_\ve\(\delta\),
\end{equation}
such that
$$\(L_{\ve}^{t_\ve}\(\hat \gamma_\ve,\theta_\ve,\tau_\ve\),M_{\ve}^{t_\ve}\(\hat \gamma_\ve,\theta_\ve,\tau_\ve\), N^{t_\ve}\(\hat\gamma_\ve,\theta_\ve,\tau_\ve\)\) =0.$$
Then, multiplying $M^{t_\ve,j}_{\ve}$ by $\partial_{\gamma_i}E^{\(j\)}_{\ve,\gamma_\ve,\tau_\ve}$, subtracting it from $L^{t_\ve,i}_{\ve}$ for $j=1,\dots,k$ and using Proposition~\ref{PrEive}, we obtain (upon multiplication by $2\gammae^2$)
\begin{equation}\label{Pr9Eq4}
E^{\(i\)}_{\ve,\gamma_\ve,\tau_\ve}= \smallo\(\frac{\delta_\ve \ln\overline{\gamma}_\ve}{\overline{\gamma}_\ve^3}\),\quad \text{for }i=1,\dots,k.
\end{equation}
Plugging \eqref{Pr9Eq4} into the equation for $M^{t_\ve,i}_{\ve}$, we then obtain
\begin{equation}\label{Pr9Eq4b}
\theta_\ve= \smallo\(\frac{\delta_\ve \ln\overline{\gamma}_\ve}{\overline{\gamma}_\ve^4}\),
\end{equation}
hence $\theta_{\ve}\not\in \partial \Theta^k_\ve$. Now \eqref{Pr9Eq4b}, the equation for $L^{t_\ve,i}_{\ve}$ and Proposition~\ref{PrEive} yield
\begin{equation}\label{Pr9Eq5}
\sum_{j=1}^k \partial_{\gamma_i}\big[E^{\(j\)}_{\ve,\gamma_\ve,\tau_\ve}\big]E^{\(j\)}_{\ve,\gamma_\ve,\tau_\ve} =\smallo\(\frac{\delta_\ve \ln^2\overline{\gamma}_\ve}{\overline{\gamma}_\ve^5}\),\quad \text{for }i=1,\dots,k.
\end{equation}
We can rewrite \eqref{Pr9Eq5} as
\begin{equation}\label{Pr9Eq6}
Q_\ve E_{\ve,\gamma_\ve,\tau_\ve} =\smallo\(\frac{\delta_\ve \ln\overline{\gamma}_\ve}{\overline{\gamma}_\ve^3}\),
\end{equation}
where, taking Proposition~\ref{PrEive} into account, $Q_\ve=\(Q_{\ve,ij}\)_{1\le i,j\le k}$ is a $k\times k$ matrix with $Q_{\ve,ij}=\mathcal{Q}_{ij}+\smallo\(1\)$ as $\ve\to 0$ and  
\begin{equation}\label{matrix}
\mathcal{Q}=\(\mathcal{Q}_{ij}\)_{1\le i,j\le k}:=\begin{pmatrix}
1 &1/l &\dots & 1/l\\
1/l & 1&\dots&1/l\\
\vdots & &\ddots & \vdots\\
1/l &1/l&\dots  &1
\end{pmatrix},
\quad 
E_{\ve,\gamma_\ve,\tau_\ve}=
\begin{pmatrix}
E_{\ve,\gamma_\ve,\tau_\ve}^{\(1\)}\\
\vdots\\
E_{\ve,\gamma_\ve,\tau_\ve}^{(k)}
\end{pmatrix}
\end{equation}
Now, since
$$\det Q_\ve=\det \mathcal{Q}+\smallo\(1\)=\(1+\frac{k-1}{l}\)\(1-\frac{1}{l}\)^{k-1} +\smallo\(1\)>0$$
for $\ve>0$ sufficiently small, uniformly with respect to $(\gamma,\theta,\tau)\in P^k_\ve\(\delta\)$, we can invert $Q_\ve$ in \eqref{Pr9Eq6} and get
\begin{equation}\label{Pr9Eq7}
E_{\ve,\gamma_\ve,\tau_\ve}^{\(i\)} =\smallo\(\frac{\delta_\ve \ln\overline{\gamma}_\ve}{\overline{\gamma}_\ve^3}\),\quad\text{for }i=1,\dots,k.
\end{equation}
On the other hand, still by Proposition~\ref{PrEive}, we have
$$E_{\ve,\gamma_\ve,\tau_\ve} =-2\frac{\ln\gammae}{\gammae^2} Q_\ve \hat \gamma_\ve +\smallo\(\left|\hat \gamma_\ve\right|\frac{\ln\gammae}{\gammae^2}\),$$
where we recall that $\hat \gamma=\gamma-\gammae\(\tau\)$.
Then, inverting $Q_\ve$ and using \eqref{Pr9Eq7} we end up with 
$\hat \gamma_\ve=\smallo\(\delta_\ve/\gammae\),$
which, for $\ve>0$ sufficiently small implies that $\hat \gamma_\ve \not\in\partial \overline\Gamma_\ve^k$.

Finally, writing $\hat \tau_\ve :=\tau_\ve/d_\ve$, we have
$N^i\(\hat \tau_\ve\) =\smallo\(1\)$,
where $N=(N^1,\dots,N^k)$ is as in \eqref{LehEq2}. On the other hand, $\tau_\ve \in \partial T^k_\ve\(\delta\)$ implies $\hat \tau_\ve \in \partial \widehat{T}^k\(\delta\)$, where
\begin{multline}\label{LehEq1}
\widehat{T}^k\(\delta\):=\Big\{y=\(y_1,\dotsc,y_k\)\in \R^k:\,-\frac{k}{\delta}<y_1<y_2<\dots<y_k<\frac{k}{\delta}\\
\text{and }|y_i-y_j|>\delta,\, \forall i,j\in\left\{1,\dotsc,k\right\},\,i\ne j\Big\},
\end{multline}
which is compact and this contradicts Lemma~\ref{Lemmah} for $\delta=\delta\(a_0,l,k\)>0$ sufficiently small such that $y^*\in \widehat{T}^k(\delta)$. Then we also have $\tau_\ve\not\in \partial T^k_\ve\(\delta\)$, which contradicts \eqref{Pr9Eq3b}.

We have therefore proven that $(L^{t}_{\ve}, M^{t}_{\ve},N^{t}_\ve)\ne 0$ on  $\partial \widehat{P}^k_\ve\(\delta\)$, for $\ve>0$ sufficiently small,  hence by homotopy invariance of the degree
\begin{equation}\label{Pr9Eq8}
\deg\big(\(L_{\ve},M_{\ve}, N_\ve\),  \widehat{P}^k_\ve\(\delta\),0\big)=\deg\big(\(\overline{L}_{\ve},\overline{M}_{\ve},\overline N_\ve\), \widehat{P}^k_\ve\(\delta\),0\big).
\end{equation}
The degree of $\(\overline{L}_{\ve},\overline{M}_{\ve},\overline{N}_\ve\)$ does not change upon multiplication by an invertible matrix with determinant $1$, namely if we consider
$$\begin{pmatrix}
\widetilde{L}_{\ve}\\
\overline{M}_{\ve}\\
\overline N_{\ve}\\
\end{pmatrix}
:=
\begin{pmatrix}
I_{k} & -D_\gamma\[ E_{\ve,\gamma,\tau}\]&0\\
0 & I_k&0\\
0&0& I_k
\end{pmatrix}
\begin{pmatrix}
\overline{L}_{\ve}\\
\overline{M}_{\ve}\\
\overline{N}_{\ve}
\end{pmatrix},
$$
where $D_\gamma\[ E_{\ve,\gamma,\tau}\]=\big(\partial_{\gamma_j}\big[E^{\(i\)}_{\ve,\gamma,\tau}\big]\big)_{1\le i,j\le k}$, $I_k$ is the $k\times k$ identity matrix and $\widetilde{L}_{\ve}:\widehat P^k_\ve\(\delta\)\to \R^k$, is defined by
$\widetilde{L}_{\ve}^{i} =-E^{\(i\)}_{\ve,\gamma,\tau}$,
for $i=1,\dots k$, we get
$$\deg\big(\(\overline{L}_{\ve},\overline{M}_{\ve},\overline N_\ve\), \widehat{P}^k_\ve\(\delta\),0\big)=\deg\big(\big(\widetilde{L}_{\ve},\overline{M}_{\ve},\overline N_\ve\big), \widehat{P}^k_\ve\(\delta\),0\big).$$
Expanding $E^{\(i\)}_{\ve,\gamma,\tau}$ as in Proposition~\ref{PrEive}, we do a final homotopy between $\big(\widetilde{L}_{\ve},\overline{M}_{\ve}\big)$ and $\big(L_{\ve}^*,M_{\ve}^*\big):\widehat P^k_\ve\(\delta\)\to \R^{2k}$, where 
$$L_{\ve}^{*i}=-\frac{2\ln \overline{\gamma}_{\ve}}{\overline{\gamma}_{\ve}^2}\(\hat\gamma_i+\sum_{j\ne i}\frac{\hat \gamma_j}{l}\),\quad M_{\ve}^{*i}=L_{\ve}^{*i}+\theta_i\gammae$$
for $i=1,\dots k$ (with the same method as above to prevent zeroes on $\partial \widehat P^k_\ve\(\delta\)$), so that
$$\deg\big(\big(\widetilde{L}_{\ve},\overline{M}_{\ve},\overline N_\ve\big), \widehat{P}^k_\ve\(\delta\),0\big)=\deg\big(\(L^*_{\ve},M^*_{\ve},\overline N_\ve\), \widehat{P}^k_\ve\(\delta\),0\big) $$
Using the matrix $\mathcal{Q}$ defined in \eqref{matrix}, we see that
\begin{equation}\label{Pr9Eq8b}
\begin{pmatrix}
L_{\ve}^*\\
M_{\ve}^*\\
\overline{N}_\ve
\end{pmatrix}
=
\begin{pmatrix}
-2\frac{\ln\gammae}{\gammae^2} \mathcal{Q}& 0&0\\
-2\frac{\ln\gammae}{\gammae^2} \mathcal{Q} &\gammae I_k&0\\
0&0&I_k
\end{pmatrix}
\begin{pmatrix}
\hat \gamma\\
\theta\\
\overline{N}_\ve
\end{pmatrix}.
\end{equation}
Since $\mathcal{Q}$ has positive determinant, if we call $\mathcal{A}$ the $3k\times 3k$ matrix on the right-hand side of \eqref{Pr9Eq8b} we have $\sign\(\det\mathcal{A}\)=(-1)^k$, and noticing that $\overline{N}_\ve$ only depends on $\tau$, we obtain
$$\mathcal{A}^{-1}\begin{pmatrix}
L_{\ve}^*\\
M_{\ve}^*\\
\overline{N}_\ve
\end{pmatrix}
= \mathrm{Id}\times \mathrm{Id}\times \overline{N}_\ve :\widehat{\Gamma}^k_\ve \times \Theta^k_\ve\times T^k_\ve\(\delta\) \to \R^{k}\times \R^k\times \R^k,$$
and using the product formula for the degree, we finally obtain
\begin{align}\label{Pr9Eq9}
&\deg\big(\(L_{\ve},M_{\ve}, N_\ve\), \widehat P^k_\ve\(\delta\),0\big)= \deg\big(\(L^*_{\ve,},M^*_{\ve},\overline{N}_\ve\),  \widehat{P}^k_\ve\(\delta\),0\big)\nonumber\\
&\qquad=(-1)^k\deg\big(\text{Id}, \widehat\Gamma^k_\ve,0\big) \deg\(\text{Id},\Theta^k_\ve,0\)\deg\(\overline{N}_\ve,T^k_\ve\(\delta\),0\)\nonumber\\
&\qquad=(-1)^k\deg\(\overline{N}_\ve,T^k_\ve\(\delta\),0\).
\end{align}
In order to compute the degree of $\overline{N}_\ve$, observe that $\overline N_\ve^i\(\tau\)=N^i\(\hat\tau_\ve\)$, where $\hat \tau_\ve=\tau/d_\ve$ and $N=\(N^1,\dotsc,N^k\)$ is as in \eqref{LehEq2}.
Moreover, since $\delta\in (0,1)$ was chosen such that $y^*\in \widehat{T}^k(\delta)$, with $y^*$ as in Lemma \ref{Lemmah}, it follows that 
$$\deg\big(\overline{N}_\ve, T^k_\ve\(\delta\),0\big)=\deg\big(N,\widehat{T}^k\(\delta\),0\big)=1.$$
We then conclude with \eqref{Pr9Eq9} that there exists $\(\hat \gamma_\ve,\theta_\ve,\tau_\ve\)\in \widehat{P}^k_\ve\(\delta\)$ such that $\(\gamma_\ve,\theta_\ve,\tau_\ve\)=\(\hat\gamma_\ve+\gammae\(\tau\),\theta_\ve,\tau_\ve\)\in P^k_\ve\(\delta\)$ solves \eqref{Pr9Eq1b}.
\endproof

\begin{lemma}\label{Lemmah} The function 
$$N:\widehat{T}^k(0):=\left\{y=\(y_1,\dotsc,y_k\)\in \R^k:\,y_1<y_2<\dots<y_k\right\}\to\R^k$$
given by
\begin{equation}\label{LehEq2}
N^i\(y_1,\dots, y_k\):=a_0ly_i^{l-1}-2\sum_{j\ne i}\frac{1}{y_i-y_j},\quad\text{for }i=1,\dots,k,
\end{equation}
has exactly one zero, which we call $y^*$. Moreover $\deg(H,\widehat{T}^k(0),0)=1$.
\end{lemma}

\begin{proof}
We have that $N=\nabla J$, with
$$J\(y\)=a_0\sum_{i=1}^ky_i^l  +\frac{1}{2}\sum_{i\ne j}\ln \frac{1}{\(y_i-y_j\)^2},\quad\forall y=\(y_1,\dots,y_k\)\in \widehat{T}^k\(0\).$$
The Hessian $\nabla^2 J$ is positive definite on $\widehat{T}^k\(0\)$, since
\begin{align*}
\partial_{y_i}^2 J&=\partial_{y_i} N^i=a_0l(l-1)y_i^{l-2}+\sum_{j\ne i}\frac{2}{\(y_i-y_j\)^2},\allowdisplaybreaks\\
\partial_{y_i}\partial_{y_j} J&=\partial_{y_i} N^j=-\frac{2}{\(y_i-y_j\)^2},\quad\text{for }i\ne j,
\end{align*}
so that for every $\xi\in \R^k\setminus\{0\}$, using that $\xi_i^2+\xi_j^2\ge 2\xi_i\xi_j$, we get
\begin{align*}
\xi^T \nabla^2 J \xi&=\sum_{i=1}^k\xi_i^2\(a_0l\(l-1\)y_i^{l-2}+\sum_{j\ne i}\frac{2}{\(y_i-y_j\)^2}\) -\sum_{i=1}^k \sum_{j\ne i}\frac{2\xi_j\xi_i}{\(y_i-y_j\)^2}\\
&\ge \sum_{i=1}^k \xi_i^2a_0 l\(l-1\)y_i^{l-2},
\end{align*}
and, using that $y\in \widehat T^k\(0\)$ and $l\in2\N^*$, the right-hand side is positive, unless $\xi=\(0,\dots, \xi_{i_0},\dots 0\)$ and $y_{i_0}=0$ for some $i_0\in \left\{1,\dotsc,k\right\}$, in which case
$$\xi^T \nabla^2 J \xi = \sum_{j\ne i_0}\frac{2\xi_{i_0}^2}{(y_{i_0}-y_j)^2}>0.$$
Then $J$ is strictly convex in $\widehat{T}^k(0)$ and since $|J(y)|\to \infty$ as $y\to \partial \widehat{T}^k(0)$ or $|y|\to\infty$, $J$ has a minimum $y^*$, which is its only critical point and the only zero of $N$. Moreover $\det(\nabla N(y^*))=\det (\nabla^2J(y^*))>0$, hence $\deg(N,\widehat{T}^k(0),0)=1$.
\end{proof}

Finally, we can now conclude the proof of Theorems~\ref{Th1} and \ref{Th2}.

\proof[End of proof of Theorems~\ref{Th1} and \ref{Th2}]
It follows from Proposition~\ref{Pr9} that for small $\varepsilon>0$, $u_\varepsilon:=U_{\varepsilon,\gamma_\varepsilon,\tau_\varepsilon,\theta_\varepsilon}+\Phi_{\varepsilon,\gamma_\varepsilon,\tau_\varepsilon,\theta_\varepsilon}\in E_{h_\varepsilon,\beta_\varepsilon}$, where $\beta_\varepsilon:=\big\|\nabla u_\varepsilon\big\|_{L^2}^2$. We denote $\gamma_\varepsilon=\(\gamma_{1,\varepsilon},\dotsc,\gamma_{k,\varepsilon}\)$, $\tau_{\varepsilon}=\(\tau_{1,\varepsilon},\dotsc,\tau_{k,\varepsilon}\)$, $\theta_{\varepsilon}=\(\theta_{1,\varepsilon},\dotsc,\theta_{k,\varepsilon}\)$.
By using \eqref{Pr4Eq2} and \eqref{Pr8Eq2}, we obtain that $\Psi_{\varepsilon,\gamma_\varepsilon,\tau_\varepsilon},\Phi_{\varepsilon,\gamma_\varepsilon,\tau_\varepsilon,\theta_\varepsilon}\to0$ in $H^1_0\(\Omega\)$ as $\varepsilon\to0$. Since moreover $w_\varepsilon\to w_0$ in $C^1\(\overline\Omega\)$, $H\in C^1\(\overline\Omega\times \overline{\Omega}\)$, $\theta_{i,\varepsilon}\to0$ and $A_{\varepsilon,\gamma_{i,\varepsilon},\tau_{i,\varepsilon}}\to0$ for all $i\in\left\{1,\dotsc,k\right\}$, we obtain
\begin{multline}\label{Th1Eq1}
\big\|\nabla{u}_\varepsilon\big\|_{L^2}=\Big\|\nabla w_0+\sum_{i=1}^k\(1+\theta_{i,\varepsilon}\)\big(\nabla\overline{B}_{\varepsilon,\gamma_{i,\varepsilon},\overline{\tau_{i,\varepsilon}}}\mathbf{1}_{B\(\overline{\tau_{i,\varepsilon}},r_\varepsilon\)}\\
+A_{\varepsilon,\gamma_{i,\varepsilon},\tau_{i,\varepsilon}}\nabla G\(\cdot,\overline{\tau_{i,\varepsilon}}\)\mathbf{1}_{\Omega\backslash B\(\overline{\tau_{i,\varepsilon}},r_\varepsilon\)}\big)\Big\|_{L^2}+\smallo\(1\)
\end{multline}
as $\varepsilon\to0$. By integrating by parts, we obtain
\begin{equation}\label{Th1Eq2}
\left\|\nabla G\(\cdot,\overline{\tau_{i,\varepsilon}}\)\mathbf{1}_{\Omega\backslash B\(\overline{\tau_{i,\varepsilon}},r_\varepsilon\)}\right\|_{L^2}^2=-\int_{\partial B\(\overline{\tau_{i,\varepsilon}},r_\varepsilon\)}G\(\cdot,\overline{\tau_{i,\varepsilon}}\)\partial_\nu G\(\cdot,\overline{\tau_{i,\varepsilon}}\)d\sigma\sim\frac{1}{2\pi}\ln\frac{1}{r_\varepsilon}\sim\frac{\delta_0\overline\gamma_\varepsilon^2}{4\pi},
\end{equation}
\begin{align}\label{Th1Eq3}
\left<\nabla G\(\cdot,\overline{\tau_{i,\varepsilon}}\)\mathbf{1}_{\Omega\backslash B\(\overline{\tau_{i,\varepsilon}},r_\varepsilon\)},\nabla w_0\right>_{L^2}=\int_{\partial B\(\overline{\tau_{i,\varepsilon}},r_\varepsilon\)}w_0\partial_\nu G\(\cdot,\overline{\tau_{i,\varepsilon}}\)d\sigma&=\bigO\big(\left\|w_0\right\|_{C^0\(\partial B\(\overline{\tau_{i,\varepsilon}},r_\varepsilon\)\)}\big)\nonumber\\
&=\smallo\(1\)
\end{align}
and
\begin{multline}\label{Th1Eq4}
\Big<\nabla G\(\cdot,\overline{\tau_{i,\varepsilon}}\)\mathbf{1}_{\Omega\backslash B\(\overline{\tau_{i,\varepsilon}},r_\varepsilon\)},\nabla G\(\cdot,\overline{\tau_{j,\varepsilon}}\)\mathbf{1}_{\Omega\backslash B\(\overline{\tau_{j,\varepsilon}},r_\varepsilon\)}\Big>_{L^2}\\
=\int_{\partial B\(\overline{\tau_{i,\varepsilon}},r_\varepsilon\)\cup\partial B\(\overline{\tau_{j,\varepsilon}},r_\varepsilon\)}G\(\cdot,\overline{\tau_{i,\varepsilon}}\)\partial_\nu G\(\cdot,\overline{\tau_{j,\varepsilon}}\)d\sigma=\bigO\(\ln\frac{1}{d_\varepsilon}\)=\bigO\(\ln\overline\gamma_\varepsilon\)
\end{multline}
as $\varepsilon\to0$ for all $i,j\in\left\{1,\dotsc,k\right\}$, $i\ne j$, where $\nu$ and $d\sigma$ are the outward unit normal vector and volume element of $\partial B\(\overline{\tau_{i,\varepsilon}},r_\varepsilon\)\cup\partial B\(\overline{\tau_{j,\varepsilon}},r_\varepsilon\)$, respectively. On the other hand, since $\gamma_{i,\varepsilon}\sim\overline\gamma_{\varepsilon}$, by using \eqref{Pr10Eq1}, we obtain
\begin{align}\label{Th1Eq5}
\left\|\nabla\overline{B}_{\varepsilon,\gamma_{i,\varepsilon},\overline{\tau_{i,\varepsilon}}}\mathbf{1}_{B\(\overline{\tau_{i,\varepsilon}},r_\varepsilon\)}\right\|_{L^2}^2&=2\pi\int_0^{\sqrt{\lambda_\varepsilon h_\varepsilon\(\overline{\tau_{i,\varepsilon}}\)}r_\varepsilon}\big(\overline{B}_{\gamma_{i,\varepsilon}}'\(r\)\big)^2r\,dr\nonumber\\
&\sim\frac{8\pi}{\gamma_{i,\varepsilon}^2}\ln\(\frac{\sqrt{\lambda_\varepsilon h_\varepsilon\(\overline{\tau_{i,\varepsilon}}\)}r_\varepsilon}{\mu_{i,\varepsilon}}\)\sim4\pi\(1-\delta_0\)
\end{align}
for all $i\in\left\{1,\dotsc,k\right\}$, where $\mu_{i,\varepsilon}$ is defined by $\mu_{i,\varepsilon}^2:=4\gamma_{i,\varepsilon}^{-2}\exp\(-\gamma_{i,\varepsilon}^2\)$. For every $j\ne i$, by remarking that $B\(\overline{\tau_{i,\varepsilon}},r_\varepsilon\)\cap B\(\overline{\tau_{j,\varepsilon}},r_\varepsilon\)=\emptyset$ for small $\varepsilon>0$ and
$$\left\|A_{\varepsilon,\gamma_{j,\varepsilon},\tau_{j,\varepsilon}}\nabla G\(\cdot,\overline{\tau_{j,\varepsilon}}\)\mathbf{1}_{B\(\overline{\tau_{i,\varepsilon}},r_\varepsilon\)}\right\|_{C^0}=\bigO\(\frac{1}{\overline\gamma_\varepsilon d_\varepsilon}\)=\smallo\(1\),$$
we obtain
\begin{multline}\label{Th1Eq6}
\Big<\nabla\overline{B}_{\varepsilon,\gamma_{i,\varepsilon},\overline{\tau_{i,\varepsilon}}}\mathbf{1}_{B\(\overline{\tau_{i,\varepsilon}},r_\varepsilon\)},\nabla w_0+\sum_{j\ne i}\(1+\theta_{i,\varepsilon}\)\big(\nabla\overline{B}_{\varepsilon,\gamma_{i,\varepsilon},\overline{\tau_{i,\varepsilon}}}\mathbf{1}_{B\(\overline{\tau_{i,\varepsilon}},r_\varepsilon\)}+A_{\varepsilon,\gamma_{i,\varepsilon},\tau_{i,\varepsilon}}\nabla G\(\cdot,\overline{\tau_{i,\varepsilon}}\)\\
\times\mathbf{1}_{\Omega\backslash B\(\overline{\tau_{i,\varepsilon}},r_\varepsilon\)}\big)\Big>_{L^2}=\bigO\(\left\|\nabla\overline{B}_{\varepsilon,\gamma_{i,\varepsilon},\overline{\tau_{i,\varepsilon}}}\mathbf{1}_{B\(\overline{\tau_{i,\varepsilon}},r_\varepsilon\)}\right\|_{L^1}\)=\bigO\(\frac{r_\varepsilon}{\overline\gamma_\varepsilon}\)=\smallo\(1\)
\end{multline}
as $\varepsilon\to0$. Since moreover $\theta_\varepsilon\to0$, it follows from \eqref{Sec32Eq3} and \eqref{Th1Eq1}--\eqref{Th1Eq6} that
$$\big\|\nabla{u}_\varepsilon\big\|^2_{L^2}\longrightarrow\left\|\nabla w_0\right\|^2_{L^2}+4k\pi=\beta_0+4k\pi$$
as $\varepsilon\to0$. Standard elliptic theory gives that $\big\|{u}_\varepsilon\big\|_{L^\infty}\to\infty$ as $\varepsilon\to0$. Since $h_\varepsilon\to h_0$ in $C^2\(\overline\Omega\)$, we then obtain that $\beta_0+4k\pi$ is an unstable energy level of $I_{h_0}$. This ends the proof of Theorem~\ref{Th1}.

The above construction also proves Theorem~\ref{Th2} if $w_0$ is non-degenerate. Otherwise we apply a diagonal procedure. More precisely, thanks to Proposition \ref{Pr0}, for $\kappa\in (0,\kappa_0)$ and $\ve=\ve(\kappa)$ sufficiently small we construct $w_{\kappa,\ve(\kappa)}\in E_{\beta_{\kappa,\ve(\kappa)},h_{\kappa,\ve(\kappa)}}$, with $w_{\kappa,\ve(\kappa)}\to w_0$, $h_{\kappa,\ve(\kappa)}\to h$  in $C^2(\overline\Omega)$ and $\beta'_\kappa:=\beta_{\kappa,\ve(\kappa)}\to \beta_0$ as $\kappa\to 0$; we further construct
$$u_\kappa = w_\kappa +\sum_{i=1}^k(1+\theta_{\kappa,i}) B_{\kappa,\gamma_{\kappa,i},\tau_{\kappa,i}}+\Psi_{\kappa,\gamma_\kappa,\tau_\kappa}+\Phi_{\kappa,\gamma_\kappa,\theta_\kappa,\tau_\kappa}\in E_{h_\kappa,\beta_\kappa},
 $$
where each subscript $\kappa$ on the right-hand side actually means $(\kappa,\ve(\kappa))$, with $\ve(\kappa)>0$ sufficiently small so that
$$\|\nabla \Psi_{\kappa,\ve(\kappa),\gamma_{\kappa,\ve(\kappa)},\tau_{\kappa,\ve(\kappa)}}\|_{L^2}+ \|\nabla \Phi_{\kappa,\ve(\kappa),\gamma_{\kappa,\ve(\kappa)},\theta_{\kappa,\ve(\kappa)},\tau_{\kappa,\ve(\kappa)}}\|_{L^2}=\smallo(1) \quad \text{as }\kappa\to 0.$$
Up to renaming the indices, we conclude.
\endproof

\begin{rem}[Stable vs. positively stable energy levels]\label{RemStable} As in Definition \ref{Def1}, let $(u_\varepsilon)$ be a family of functions such that $u_\varepsilon\in E_{h_\varepsilon,\beta_\varepsilon}$ with $h_\varepsilon\to h>0$ in $C^2(\bar{\Omega})$ and $\beta_\varepsilon\to \beta>0$. In particular, $u_\varepsilon$ solves $(\mathcal{E}_{h_\varepsilon,\beta_\varepsilon})$ with $\lambda=\lambda_\varepsilon>0$ obtained from $h_\varepsilon$, $\beta_\varepsilon$ and $u_\varepsilon$ thanks to \eqref{Sec1Eq2}. As a simple claim, testing $(\mathcal{E}_{h_\varepsilon,\beta_\varepsilon})$ against $v>0$, first eigenfuntion of $\Delta$ with zero Dirichlet condition on $\partial\Omega$, the bound $\lambda_\varepsilon=O(1)$ is automatic when defining a \emph{positively} stable energy level in Definition \ref{Def1}. In the sign changing case, however, let us consider the following unstable situation: $u_\varepsilon$ goes uniformly to $0\not \in E_{h,\beta}$, while looking like a ($h$-weighted) Dirichlet eigenfunction associated to some large eigenvalue $\bar{\lambda}_\varepsilon\sim \lambda_\varepsilon\to +\infty$, but still having the given energy $\beta_\varepsilon\sim \beta>0$ as $\varepsilon\to 0$. Then, in order not to have an empty notion of \emph{stable energy level}, we further assume the bound $\lambda_\varepsilon=O(1)$ in Definition \ref{Def1}.
\end{rem}

\section{Proof of Proposition~\ref{Pr4}}\label{ProofPr4}

We fix $\varepsilon\in\(0,1\)$ and $\delta'\in\(0,1-\sqrt{2\delta_0}\)$. For every $p>1$, we define $\mathcal{W}^{2,p}_0\(\Omega\):=W^{2,p}\(\Omega\)\cap H^1_0\(\Omega\)$. Note that we have a compact embedding of $W^{2,p}\(\Omega\)$ into $H^1\(\Omega\)$ and $C^0\(\overline\Omega\)$ when $p>1$ and into $C^1\(\overline\Omega\)$ when $p>2$. For every $\varepsilon\in\(0,\varepsilon_0\)$ and $\tau\in T^k_{\varepsilon}\(\delta\)$, we let $L_{\varepsilon,\tau}:\mathcal{W}^{2,p}_0\(\Omega\)\to \mathcal{W}^{2,p}_0\(\Omega\)$ be the operator defined as
\begin{equation}\label{Pr4Eq4}
L_{\varepsilon,\tau}\(\Psi\)=\Psi-\Delta^{-1}\[\lambda_\varepsilon h_\varepsilon\chi_{\varepsilon,\tau}f'\(w_\varepsilon\)\Psi\]\qquad\forall\Psi\in\mathcal{W}^{2,p}_0\(\Omega\).
\end{equation}
As a first step, we prove that there exists a constant $C=C\(p,\delta\)>0$ such that 
\begin{equation}\label{Pr4Eq5}
\left\|\Psi\right\|_{W^{2,p}}\le C\left\|L_{\varepsilon,\tau}\(\Psi\)\right\|_{W^{2,p}}\qquad\forall\Psi\in\mathcal{W}^{2,p}_0\(\Omega\)
\end{equation}
so that in particular $L_{\varepsilon,\tau}$ is an isomorphism. We assume by contradiction that there exist sequences $\(\varepsilon_n,\tau_n,\Psi_n\)_{n}$ such that $\varepsilon_n\to0$, $\tau_n\in T_{\varepsilon_n}^k\(\delta\)$, $\Psi_n\in\mathcal{W}^{2,p}_0\(\Omega\)$ and 
\begin{equation}\label{Pr4Eq6}
\left\|\Psi_n\right\|_{W^{2,p}}=1\quad\text{and}\quad\left\|L_{\varepsilon_n,\tau_n}\(\Psi_n\)\right\|_{W^{2,p}}\to0
\end{equation}
as $n\to\infty$. In particular, we obtain that $\(\Psi_n\)_{n}$ converges, up to a subsequence, weakly in $W^{2,p}\(\Omega\)$ and strongly in $H^1_0\(\Omega\)$ and $C^0\(\overline\Omega\)$ to a function $\Psi_0$. By using the second part of \eqref{Pr4Eq6}, we obtain
\begin{equation}\label{Pr4Eq7}
\int_\Omega\left<\nabla\Psi_n,\nabla\phi\right>dx-\lambda_{\varepsilon_n}\int_\Omega h_{\varepsilon_n}\chi_{\varepsilon_n,\tau_n}f'\(u_{\varepsilon_n}\)\Psi_n\phi\,dx=\smallo\(1\)
\end{equation}
for all $\phi\in C^\infty_c\(\Omega\)$. Since $\lambda_{\varepsilon_n}h_{\varepsilon_n}\chi_{\varepsilon_n,\tau_n}f'\(u_{\varepsilon_n}\)$ is uniformly bounded and converges pointwise to $\lambda_0h_0f'\(u_0\)$ in $\Omega$ and $\Psi_n\to\Psi_0$ in $H^1_0\(\Omega\)$ and $C^0\(\overline\Omega\)$, by passing to the limit into \eqref{Pr4Eq7}, we obtain that $\Psi_0$ is a solution of the problem
\begin{equation*}
\left\{\begin{aligned}&\Delta\Psi_0=\lambda_0 h_0f'\(u_0\)\Psi_0&&\text{in }\Omega\\&\Psi_0=0&&\text{on }\partial\Omega.\end{aligned}\right.
\end{equation*}
Since $u_0$ is non-degenerate, it follows that $\Psi_0\equiv0$. By using \eqref{Pr4Eq6} together with standard  $L^p$-estimates for the Dirichlet problem (see Lemma~9.17 of Gilbarg--Trudinger~\cite{GilTru}), we then obtain 
\begin{align*}
\left\|\Psi_n\right\|_{W^{2,p}}&\le\left\|L_{\varepsilon_n,\tau_n}\(\Psi_n\)\right\|_{W^{2,p}}+\left\|\Delta^{-1}\[\lambda_{\varepsilon_n}h_{\varepsilon_n}\chi_{\varepsilon_n,\tau_n}f'\(u_{\varepsilon_n}\)\Psi_n\]\right\|_{W^{2,p}}\\
&=\smallo\(1\)+\bigO\(\left\|\lambda_{\varepsilon_n}h_{\varepsilon_n}\chi_{\varepsilon_n,\tau_n}f'\(u_{\varepsilon_n}\)\Psi_n\right\|_{L^p}\)=\smallo\(1\)
\end{align*}
as $n\to\infty$, which is in contradiction with \eqref{Pr4Eq6}. This ends the proof of \eqref{Pr4Eq5}.

Now, for every $\varepsilon\in\(0,\varepsilon_0\)$ and $\(\gamma,\tau\)\in T^k_{\varepsilon}\(\delta\)\times\Gamma_\varepsilon^k\(\delta'\)$, we let $N_{\varepsilon,\gamma,\tau},T_{\varepsilon,\gamma,\tau}:\mathcal{W}^{2,p}_0\(\Omega\)\to \mathcal{W}^{2,p}_0\(\Omega\)$ be the operators defined as
\begin{align*}
N_{\varepsilon,\gamma,\tau}\(\Psi\)&:=\Delta^{-1}\big[\lambda_\varepsilon h_\varepsilon\chi_{\varepsilon,\tau}\big(f\big(\widetilde{U}_{\varepsilon,\gamma,\tau}+\Psi\big)-f\(w_\varepsilon\)-f'\(w_\varepsilon\)\Psi\big)\big],\allowdisplaybreaks\\
T_{\varepsilon,\gamma,\tau}\(\Psi\)&:=L^{-1}_{\varepsilon,\tau}\(N_{\varepsilon,\gamma,\tau}\(\Psi\)-R_{\varepsilon,\tau}\)
\end{align*}
for all $\Psi\in \mathcal{W}^{2,p}_0\(\Omega\)$, where
$$R_{\varepsilon,\tau}:=w_\varepsilon-\Delta^{-1}\[\lambda_\varepsilon h_\varepsilon\chi_{\varepsilon,\tau}f\(w_\varepsilon\)\]=\Delta^{-1}\[\lambda_\varepsilon h_\varepsilon(1-\chi_{\varepsilon,\tau})f\(w_\varepsilon\)\].$$
Note that the problem \eqref{Pr4Eq1} can be rewritten as the fixed point equation $T_{\varepsilon,\gamma,\tau}\(\Psi\)=\Psi$. For every $C>0$ and $\varepsilon\in\(0,\varepsilon_0\)$, we define
$$V_\varepsilon\(C\):=\left\{\Psi\in \mathcal{W}^{2,p}_0\(\Omega\):\,\left\|\Psi\right\|_{W^{2,p}}\le C/\overline{\gamma}_\varepsilon\right\}.$$
We will prove that if $C$ is chosen large enough, then $T_{\varepsilon,\gamma,\tau}$ has a fixed point in $V_\varepsilon\(C\)$ for small $\varepsilon>0$. By using a standard $L^p$-estimate and since $\lambda_\varepsilon\to\lambda_0$, $h_\varepsilon\to h_0$ and $w_\varepsilon\to w_0$ in $C^0\(\overline\Omega\)$, we obtain
\begin{equation}\label{Pr4Eq8}
\left\|R_{\varepsilon,\tau}\right\|_{W^{2,p}}=\bigO\(\left\|\lambda_\varepsilon h_\varepsilon\(1-\chi_{\varepsilon,\tau}\)f\(w_\varepsilon\)\right\|_{L^p}\)=\bigO\(\left\|1-\chi_{\varepsilon,\tau}\right\|_{L^p}\)=\smallo\(1/\overline{\gamma}_\varepsilon\)
\end{equation}
as $\varepsilon\to0$, uniformly in $\tau\in T_\varepsilon^k\(\delta\)$. Similarly, for every $\Psi,\Psi_1,\Psi_2\in V_\varepsilon\(C\)$, we obtain
\begin{align}
&\left\|N_{\varepsilon,\gamma,\tau}\(\Psi\)\right\|_{W^{2,p}}=\bigO\big(\big\|\chi_{\varepsilon,\tau}\big(f\big(\widetilde{U}_{\varepsilon,\gamma,\tau}+\Psi\big)-f\(w_\varepsilon\)-f'\(w_\varepsilon\)\Psi\big)\big\|_{L^p}\big),\label{Pr4Eq9}\allowdisplaybreaks\\
&\left\|N_{\varepsilon,\gamma,\tau}\(\Psi_1\)-N_{\varepsilon,\gamma,\tau}\(\Psi_2\)\right\|_{W^{2,p}}=\bigO\big(\big\|\chi_{\varepsilon,\tau}\big(f\big(\widetilde{U}_{\varepsilon,\gamma,\tau}+\Psi_1\big)-f\big(\widetilde{U}_{\varepsilon,\gamma,\tau}+\Psi_2\big)\nonumber\\
&\hspace{250pt}-f'\(w_\varepsilon\)\(\Psi_1-\Psi_2\)\big)\big\|_{L^p}\big).\label{Pr4Eq10}
\end{align}
By applying the mean value theorem together with H\"older's inequality, it follows from \eqref{Pr4Eq9} and \eqref{Pr4Eq10} that
\begin{align}
&\left\|N_{\varepsilon,\gamma,\tau}\(\Psi\)\right\|_{W^{2,p}}=\bigO\bigg(\bigg\|\chi_{\varepsilon,\tau}f'\Big(w_\varepsilon+t_1\sum_{i=1}^kB_{\varepsilon,\gamma_i,\tau_i}+\Psi\Big)\sum_{i=1}^kB_{\varepsilon,\gamma_i,\tau_i}\bigg\|_{L^p}\nonumber\\
&\hspace{214pt}+\left\|\chi_{\varepsilon,\tau}f''\(w_\varepsilon+s_1\Psi\)\right\|_{L^p}\left\|\Psi\right\|_{C^0}^2\bigg),\label{Pr4Eq11}\allowdisplaybreaks\\
&\left\|N_{\varepsilon,\gamma,\tau}\(\Psi_1\)-N_{\varepsilon,\gamma,\tau}\(\Psi_2\)\right\|_{W^{2,p}}\nonumber\\
&\quad=\bigO\big(\big\|\chi_{\varepsilon,\tau}\big(f'\big(\widetilde{U}_{\varepsilon,\gamma,\tau}+\(1-s_2\)\Psi_1+s_2\Psi_2\big)-f'\(w_\varepsilon\)\big)\big\|_{L^p}\left\|\Psi_1-\Psi_2\right\|_{C^0}\big)\nonumber\allowdisplaybreaks\\
&\quad=\bigO\bigg(\bigg\|\chi_{\varepsilon,\tau}f''\Big(w_\varepsilon+t_2\sum_{i=1}^kB_{\varepsilon,\gamma_i,\tau_i}+t_2\(1-s_2\)\Psi_1+t_2s_2\Psi_2\Big)\nonumber\\
&\qquad\qquad\times\Big(\sum_{i=1}^kB_{\varepsilon,\gamma_i,\tau_i}+\(1-s_2\)\Psi_1+s_2\Psi_2\Big)\bigg\|_{L^p}\left\|\Psi_1-\Psi_2\right\|_{C^0}\bigg)\label{Pr4Eq12}
\end{align}
for some functions $s_1,s_2,t_1,t_2:\Omega\to\[0,1\]$. Since $0\le\chi_{\varepsilon,\tau}\le1$ in $\Omega$, $w_\varepsilon\to w_0$ in $C^0\(\overline\Omega\)$ and $\Psi\in V_\varepsilon\(C\)$, we obtain
\begin{equation}\label{Pr4Eq13}
\left\|\chi_{\varepsilon,\tau}f''\(w_\varepsilon+s_1\Psi\)\right\|_{L^p}=\bigO\(1\).
\end{equation}
For every $j\in\left\{1,\dotsc,k\right\}$, by using \eqref{Sec32Eq3}, we obtain  
\begin{equation}\label{Pr4Eq14}
B_{\varepsilon,\gamma_j,\tau_j}\(x\)=\frac{2}{\gamma_j}\(\ln\frac{1}{\left|x-\overline{\tau_j}\right|}+\bigO\(1\)\)
\end{equation}
uniformly in $x\in \Omega\backslash B\(\overline{\tau_j},r_\varepsilon\)$. We let $R_\varepsilon:=\exp\(-\overline\gamma_\varepsilon\)\gg r_\varepsilon$. Since $\chi_{\varepsilon,\tau}\equiv0$ in $B\(\overline{\tau_j},r_\varepsilon\)$, $0\le\chi_{\varepsilon,\tau}\le1$ in $\Omega$, $w_\varepsilon\to w_0$ in $C^0\(\overline\Omega\)$, $u_0\(0\)=0$, $\Psi,\Psi_1,\Psi_2\in V_\varepsilon\(C\)$ and $0\le s_1,s_2,t_1,t_2\le1$, it follows from \eqref{Pr4Eq14} that 
\begin{align}
&\bigg\|\chi_{\varepsilon,\tau}f'\Big(w_\varepsilon+t_1\sum_{i=1}^kB_{\varepsilon,\gamma_i,\tau_i}+\Psi\Big)\sum_{i=1}^kB_{\varepsilon,\gamma_i,\tau_i}\mathbf{1}_{B\(\overline{\tau_j},R_\varepsilon\)}\bigg\|_{L^p}^p\nonumber\\
&\qquad=\bigO\(\overline\gamma_\varepsilon^p\int_{r_\varepsilon}^{R_\varepsilon}f'\(\frac{2t_1}{\gamma_j}\ln\frac{1}{r}+\smallo\(1\)\)^prdr\)\nonumber\allowdisplaybreaks\\
&\qquad=\bigO\(\overline\gamma_\varepsilon^{p+2}\int_{2\overline\gamma_\varepsilon/\gamma_j^2}^{\delta_0\overline\gamma_\varepsilon^2/\gamma_j^2}f'\(t_1\gamma_js+\smallo\(1\)\)^{p}\exp\(-\gamma_j^2s\)ds\)\nonumber\allowdisplaybreaks\\
&\qquad=\bigO\(\overline\gamma_\varepsilon^{3p+2}\int_{2\overline\gamma_\varepsilon/\gamma_j^2}^{\delta_0\overline\gamma_\varepsilon^2/\gamma_j^2}\exp\(\(pt_1^2s-1\)s\gamma_j^2+\smallo\(\gamma_j\)\)ds\)\nonumber\\
&\qquad=\bigO\(\overline\gamma_\varepsilon^{3p+2}\int_{\delta_0/[\overline\gamma_\varepsilon\(1+\delta'\)^2]}^{\delta_0/\(1-\delta'\)^2}\exp\(\(pt_1^2s-1\)s\gamma_j^2+\smallo\(\gamma_j\)\)ds\)=\smallo\(1\),\label{Pr4Eq15}\allowdisplaybreaks\\
&\bigg\|\chi_{\varepsilon,\tau}f''\Big(w_\varepsilon+t_2\sum_{i=1}^kB_{\varepsilon,\gamma_i,\tau_i}+t_2\(1-s_2\)\Psi_1+t_2s_2\Psi_2\Big)\nonumber\\
&\qquad\quad\times\Big(\sum_{i=1}^kB_{\varepsilon,\gamma_i,\tau_i}+\(1-s_2\)\Psi_1+s_2\Psi_2\Big)\mathbf{1}_{B\(\overline{\tau_j},R_\varepsilon\)}\bigg\|_{L^p}^p\nonumber\\
&\qquad=\bigO\(\overline\gamma_\varepsilon^p\int_{r_\varepsilon}^{R_\varepsilon}f''\(\frac{2t_2}{\gamma_j}\ln\frac{1}{r}+\smallo\(1\)\)^prdr\)\nonumber\allowdisplaybreaks\\
&\qquad=\bigO\(\overline\gamma_\varepsilon^{p+2}\int_{2\overline\gamma_\varepsilon/\gamma_j^2}^{\delta_0\overline\gamma_\varepsilon^2/\gamma_j^2}f''\(t_2\gamma_js+\smallo\(1\)\)^{p}\exp\(-\gamma_j^2s\)ds\)\nonumber\allowdisplaybreaks\\
&\qquad=\bigO\(\overline\gamma_\varepsilon^{4p+2}\int_{2\overline\gamma_\varepsilon/\gamma_j^2}^{\delta_0\overline\gamma_\varepsilon^2/\gamma_j^2}\exp\(\(pt_2^2s-1\)s\gamma_j^2+\smallo\(\gamma_j\)\)ds\)\nonumber\\
&\qquad=\bigO\(\overline\gamma_\varepsilon^{4p+2}\int_{\delta_0/[\overline\gamma_\varepsilon\(1+\delta'\)^2]}^{\delta_0/\(1-\delta'\)^2}\exp\(\(pt_2^2s-1\)s\gamma_j^2+\smallo\(\gamma_j\)\)ds\)=\smallo\(1\)\label{Pr4Eq16}
\end{align}
as $\varepsilon\to0$, uniformly in $(\gamma,\tau)\in \Gamma_\varepsilon^k\(\delta'\)\times T_\varepsilon^k\(\delta\)$ and $\Psi,\Psi_1,\Psi_2\in V_\varepsilon(C)$, provided we choose $p$ such that $p\delta_0/\(1-\delta'\)^2-1<0$, i.e. $p<\(1-\delta'\)^2/\delta_0$. By using \eqref{Pr4Eq14}, we obtain 
\begin{multline}\label{Pr4Eq17}
\bigg\|\chi_{\varepsilon,\tau}f'\Big(w_\varepsilon+t_1\sum_{i=1}^kB_{\varepsilon,\gamma_i,\tau_i}+\Psi\Big)\sum_{i=1}^kB_{\varepsilon,\gamma_i,\tau_i}\mathbf{1}_{\Omega_{R_\varepsilon,\tau}}\bigg\|_{L^p}^p\\
=\bigO\bigg(\frac{1}{\overline\gamma_\varepsilon^p}\sum_{i=1}^k\int_{\Omega_{R_\varepsilon,\tau}}\left|\ln\left|x-\overline{\tau_i}\right|+\bigO\(1\)\right|^pdx\bigg)=\smallo\(1\)
\end{multline}
and, similarly,
\begin{multline}\label{Pr4Eq18}
\bigg\|\chi_{\varepsilon,\tau}f''\Big(w_\varepsilon+t_2\sum_{i=1}^kB_{\varepsilon,\gamma_i,\tau_i}+t_2\(1-s_2\)\Psi_1+t_2s_2\Psi_2\Big)\\
\times\Big(\sum_{i=1}^kB_{\varepsilon,\gamma_i,\tau_i}+\(1-s_2\)\Psi_1+s_2\Psi_2\Big)\mathbf{1}_{\Omega_{R_\varepsilon,\tau}}\bigg\|_{L^p}=\smallo\(1\)
\end{multline}
as $\varepsilon\to0$, uniformly in $(\gamma,\tau)\in \Gamma_\varepsilon^k\(\delta'\)\times T_\varepsilon^k\(\delta\)$ and $\Psi,\Psi_1,\Psi_2\in V_\varepsilon(C)$.

Note that similar estimates as in \eqref{Pr4Eq15}--\eqref{Pr4Eq17} yield \eqref{Pr4Eq3b}.

By putting together \eqref{Pr4Eq12}--\eqref{Pr4Eq18} and using the continuity of the embedding $W^{2,p}\(\Omega\)\hookrightarrow  C^0\(\overline{\Omega}\)$, we obtain
\begin{align}
\left\|N_{\varepsilon,\gamma,\tau}\(\Psi\)\right\|_{W^{2,p}}&=\smallo\big(\left\|\Psi\right\|^2_{W^{2,p}}\big),\label{Pr4Eq19}\allowdisplaybreaks\\
\left\|N_{\varepsilon,\gamma,\tau}\(\Psi_1\)-N_{\varepsilon,\gamma,\tau}\(\Psi_2\)\right\|_{W^{2,p}}&=\smallo\(\left\|\Psi_1-\Psi_2\right\|_{W^{2,p}}\)\label{Pr4Eq20}
\end{align}
as $\varepsilon\to0$, uniformly in $(\gamma,\tau)\in \Gamma_\varepsilon^k\(\delta'\)\times T_\varepsilon^k\(\delta\)$ and $\Psi,\Psi_1,\Psi_2\in V_\varepsilon(C)$. It follows from \eqref{Pr4Eq5}, \eqref{Pr4Eq8}, \eqref{Pr4Eq19} and \eqref{Pr4Eq20} that there exist $\varepsilon_1\(p,\delta,\delta'\)\in\(0,\varepsilon_0\)$ and $C=C\(p,\delta,\delta'\)>0$ (here we do not specify the dependence in $\delta_0$ as this number is considered to be fixed) such that for every $\varepsilon\in\(0,\varepsilon_1\(\delta\)\)$ and $(\gamma,\tau)\in \Gamma_\varepsilon^k\(\delta'\)\times T_\varepsilon^k\(\delta\)$, $T_{\varepsilon,\gamma,\tau}$ is a contraction mapping on $V_\varepsilon\(C\)$. By the fixed point theorem, we then obtain that there exists a unique solution $\Psi_{\varepsilon,\gamma,\tau}\in V_\varepsilon\(C\)$ to the problem \eqref{Pr4Eq1}. By fixing a number $p$ such that $2<p<\(1-\delta_0\)^2/\delta_0$, the first inequality in \eqref{Pr4Eq2} then follows from the continuity of the embedding $W^{2,p}\(\Omega\)\hookrightarrow C^1\(\overline\Omega\)$. By using the Moser--Trudinger inequality together with standard elliptic regularity theory, we obtain that $\Psi_{\varepsilon,\gamma,\tau}\in C^{l,\alpha}\(\Omega\)\cap C^2\(\overline\Omega\)$. Furthermore, by symmetry of $\Omega$, $w_\varepsilon$, $h_\varepsilon$, $\chi_{\varepsilon,\tau}$ and $\widetilde{U}_{\varepsilon,\gamma,\tau}$, we obtain that $\Psi_{\varepsilon,\gamma,\tau}$ is even in $x_2$ and by using the continuous differentiability of $\widetilde{U}_{\varepsilon,\gamma,\tau}$ and $\chi_{\varepsilon,\tau}$ in $\(\gamma,\tau\)$, we obtain that $\Psi_{\varepsilon,\gamma,\tau}$ is continuously differentiable in $\(\gamma,\tau\)$.

Now, we prove the second inequality in \eqref{Pr4Eq2}. For $i\in\left\{1,\dotsc,k\right\}$, by differentiating \eqref{Pr4Eq1} in $\gamma_i$, we obtain
\begin{multline}\label{Pr4Eq21}
\Delta\[L_{\varepsilon,\tau}\(\partial_{\gamma_i}\[\Psi_{\varepsilon,\gamma,\tau}\]\)\]=\lambda_\varepsilon h_\varepsilon\chi_{\varepsilon,\tau}f'\big(\widetilde{U}_{\varepsilon,\gamma,\tau}+\Psi_{\varepsilon,\gamma,\tau}\big)\partial_{\gamma_i}\big[\widetilde{U}_{\varepsilon,\gamma,\tau}\big]\\
+\lambda_\varepsilon h_\varepsilon\chi_{\varepsilon,\tau}\big(f'\big(\widetilde{U}_{\varepsilon,\gamma,\tau}+\Psi_{\varepsilon,\gamma,\tau}\big)-f'\(w_\varepsilon\)\big)\partial_{\gamma_i}\[\Psi_{\varepsilon,\gamma,\tau}\],
\end{multline}
where $L_{\varepsilon,\tau}$ is as in \eqref{Pr4Eq4}. By using \eqref{Pr4Eq5} and \eqref{Pr4Eq21} together with a standard $L^p$-estimate and since $\lambda_\varepsilon\to\lambda_0$ and $h_\varepsilon\to h_0$ in $C^0\(\overline\Omega\)$, we then obtain
\begin{multline}\label{Pr4Eq22}
\left\|\partial_{\gamma_i}\[\Psi_{\varepsilon,\gamma,\tau}\]\right\|_{W^{2,p}}=\bigO\bigg(\left\|\chi_{\varepsilon,\tau}f'\big(\widetilde{U}_{\varepsilon,\gamma,\tau}+\Psi_{\varepsilon,\gamma,\tau}\big)\partial_{\gamma_i}\big[\widetilde{U}_{\varepsilon,\gamma,\tau}\big]\right\|_{L^p}\\
+\left\|\chi_{\varepsilon,\tau}\big(f'\big(\widetilde{U}_{\varepsilon,\gamma,\tau}+\Psi_{\varepsilon,\gamma,\tau}\big)-f'\(w_\varepsilon\)\big)\partial_{\gamma_i}\[\Psi_{\varepsilon,\gamma,\tau}\]\right\|_{L^p}\bigg).
\end{multline}
By using \eqref{Sec32Eq4}, we obtain  
\begin{equation}\label{Pr4Eq23}
\partial_{\gamma_i}\big[\widetilde{U}_{\varepsilon,\gamma,\tau}\big]=\frac{2}{\gamma_i^2}\(\ln\left|x-\overline{\tau_i}\right|+\bigO\(1\)\)
\end{equation}
uniformly in $x\in \Omega\backslash B\(\overline{\tau_i},r_\varepsilon\)$. By using \eqref{Pr4Eq14} and \eqref{Pr4Eq23} and proceeding as in \eqref{Pr4Eq15}--\eqref{Pr4Eq18}, we obtain 
\begin{multline}\label{Pr4Eq24}
\bigg\|\chi_{\varepsilon,\tau}f'\big(\widetilde{U}_{\varepsilon,\gamma,\tau}+\Psi_{\varepsilon,\gamma,\tau}\big)\partial_{\gamma_i}\big[\widetilde{U}_{\varepsilon,\gamma,\tau}\big]\bigg\|_{L^p}^p=\bigO\bigg(\int_{r_\varepsilon}^{R_\varepsilon}f'\(\frac{2}{\gamma_i}\ln\frac{1}{r}+\smallo\(1\)\)^prdr\\
+\frac{1}{\gamma_i^{2p}}\int_{\Omega_{R_\varepsilon,\tau}}\left|\ln\left|x-\overline{\tau_i}\right|+\bigO\(1\)\right|^pdx\bigg)=\bigO\(\frac{1}{\overline\gamma_\varepsilon^{2p}}\)
\end{multline}
uniformly in $(\gamma,\tau)\in \Gamma_\varepsilon^k\(\delta'\)\times T_\varepsilon^k\(\delta\)$, provided $p$ is chosen so that $p<\(1-\delta'\)^2/\delta_0$. On the other hand, by applying the mean value theorem together with H\"older's inequality, we obtain
\begin{multline}\label{Pr4Eq25}
\left\|\chi_{\varepsilon,\tau}\big(f'\big(\widetilde{U}_{\varepsilon,\gamma,\tau}+\Psi_{\varepsilon,\gamma,\tau}\big)-f'\(w_\varepsilon\)\big)\partial_{\gamma_i}\[\Psi_{\varepsilon,\gamma,\tau}\]\right\|_{L^p}\le\bigg\|\partial_{\gamma_i}\[\Psi_{\varepsilon,\gamma,\tau}\]\bigg\|_{C^0}\\
\times\bigg\|\chi_{\varepsilon,\tau}f''\Big(w_\varepsilon+t\sum_{i=1}^kB_{\varepsilon,\gamma_i,\tau_i}+t\Psi_{\varepsilon,\gamma,\tau}\Big)\Big(\sum_{i=1}^kB_{\varepsilon,\gamma_i,\tau_i}+\Psi_{\varepsilon,\gamma,\tau}\Big)\bigg\|_{L^p}
\end{multline}
for some function $t:\Omega\to\[0,1\]$. By using \eqref{Pr4Eq14} and proceeding as in \eqref{Pr4Eq15}--\eqref{Pr4Eq18}, we obtain 
\begin{equation}\label{Pr4Eq26}
\bigg\|\chi_{\varepsilon,\tau}f''\Big(w_\varepsilon+t\sum_{i=1}^kB_{\varepsilon,\gamma_i,\tau_i}+t\Psi_{\varepsilon,\gamma,\tau}\Big)\Big(\sum_{i=1}^kB_{\varepsilon,\gamma_i,\tau_i}+\Psi_{\varepsilon,\gamma,\tau}\Big)\bigg\|_{L^p}^p=\smallo\(1\)
\end{equation}
as $\varepsilon\to0$, uniformly in $(\gamma,\tau)\in \Gamma_\varepsilon^k\(\delta'\)\times T_\varepsilon^k\(\delta\)$, provided we choose $p$ such that $p<\(1-\delta'\)^2/\delta_0$. By putting together \eqref{Pr4Eq22} and \eqref{Pr4Eq24}--\eqref{Pr4Eq26}, we obtain
\begin{equation}\label{Pr4Eq27}
\left\|\partial_{\gamma_i}\[\Psi_{\varepsilon,\gamma,\tau}\]\right\|_{W^{2,p}}=\bigO\(\frac{1}{\overline\gamma_\varepsilon^2}\)+\smallo\(\left\|\partial_{\gamma_i}\[\Psi_{\varepsilon,\gamma,\tau}\]\right\|_{C^0}\)
\end{equation}
as $\varepsilon\to0$, uniformly in $(\gamma,\tau)\in \Gamma_\varepsilon^k\(\delta'\)\times T_\varepsilon^k\(\delta\)$. By choosing $p$ such that
$$2<p<\frac{\(1-\delta_0\)^2}{\delta_0}$$
and using the continuity of the embedding $W^{2,p}\(\Omega\)\hookrightarrow  C^{1}\(\overline\Omega\)$, the second inequality in \eqref{Pr4Eq2} then follows from \eqref{Pr4Eq27}.

Now, we prove \eqref{Pr4Eq3}. For every $i\in\left\{1,\dotsc,k\right\}$, by differentiating \eqref{Pr4Eq1} in $\tau_i$, we obtain
\begin{multline}\label{Pr4Eq28}
\Delta\[L_{\varepsilon,\tau}\(\partial_{\tau_i}\[\Psi_{\varepsilon,\gamma,\tau}\]\)\]=\lambda_\varepsilon h_\varepsilon f\big(\widetilde{U}_{\varepsilon,\gamma,\tau}+\Psi_{\varepsilon,\gamma,\tau}\big)\partial_{\tau_i}\big[\chi_{\varepsilon,\tau}\big]+\lambda_\varepsilon h_\varepsilon\chi_{\varepsilon,\tau}f'\big(\widetilde{U}_{\varepsilon,\gamma,\tau}+\Psi_{\varepsilon,\gamma,\tau}\big)\\
\times\partial_{\tau_i}\big[\widetilde{U}_{\varepsilon,\gamma,\tau}\big]+\lambda_\varepsilon h_\varepsilon\chi_{\varepsilon,\tau}\big(f'\big(\widetilde{U}_{\varepsilon,\gamma,\tau}+\Psi_{\varepsilon,\gamma,\tau}\big)-f'\(w_\varepsilon\)\big)\partial_{\tau_i}\[\Psi_{\varepsilon,\gamma,\tau}\],
\end{multline}
where $L_{\varepsilon,\tau}$ is as in \eqref{Pr4Eq4}. By using \eqref{Pr4Eq5} and \eqref{Pr4Eq28} together with a standard $L^p$--estimate and since $\lambda_\varepsilon\to\lambda_0$ and $h_\varepsilon\to h_0$ in $C^0\(\overline\Omega\)$, we obtain
\begin{multline}\label{Pr4Eq29}
\left\|\partial_{\tau_i}\[\Psi_{\varepsilon,\gamma,\tau}\]\right\|_{W^{2,p}}=\bigO\bigg(\left\|f\big(\widetilde{U}_{\varepsilon,\gamma,\tau}+\Psi_{\varepsilon,\gamma,\tau}\big)\partial_{\tau_i}\big[\chi_{\varepsilon,\tau}\big]\right\|_{L^p}+\Big\|\chi_{\varepsilon,\tau}f'\big(\widetilde{U}_{\varepsilon,\gamma,\tau}+\Psi_{\varepsilon,\gamma,\tau}\big)\\
\times\partial_{\tau_i}\big[\widetilde{U}_{\varepsilon,\gamma,\tau}\big]\Big\|_{L^p}+\left\|\chi_{\varepsilon,\tau}\big(f'\big(\widetilde{U}_{\varepsilon,\gamma,\tau}+\Psi_{\varepsilon,\gamma,\tau}\big)-f'\(w_\varepsilon\)\big)\partial_{\tau_i}\[\Psi_{\varepsilon,\gamma,\tau}\]\right\|_{L^p}\bigg).
\end{multline}
It is easy to see that
\begin{equation}\label{Pr4Eq30}
\partial_{\tau_i}\big[\chi_{\varepsilon,\tau}\big]=\bigO\(\frac{1}{r_\varepsilon^2}\mathbf{1}_{A\(\overline{\tau_i},r_\varepsilon,r_\varepsilon+r_\varepsilon^2\)}\)
\end{equation}
uniformly in $\Omega$. By using \eqref{Sec32Eq3} and \eqref{Sec32Eq5} and since $\delta'<1-\sqrt{2\delta_0}$, we obtain
\begin{equation}\label{Pr4Eq31}
\partial_{\tau_i}\big[\widetilde{U}_{\varepsilon,\gamma,\tau}\big]=\frac{2}{\gamma_i}\(1+\smallo\(1\)\)\frac{x_1-\tau_i}{\left|x-\overline{\tau_i}\right|^2}+\bigO\(\frac{1}{\gammae}\)
\end{equation}
uniformly in $x=\(x_1,x_2\)\in \Omega\backslash B\(\overline{\tau_i},r_\varepsilon\)$. By using \eqref{Pr4Eq14}, \eqref{Pr4Eq30}, \eqref{Pr4Eq31} and proceeding as in \eqref{Pr4Eq15}--\eqref{Pr4Eq18}, we obtain 
\begin{align}
&\left\|f\big(\widetilde{U}_{\varepsilon,\gamma,\tau}+\Psi_{\varepsilon,\gamma,\tau}\big)\partial_{\tau_i}\big[\chi_{\varepsilon,\tau}\big]\right\|^p_{L^p}\nonumber\\
&\qquad=\bigO\bigg(\frac{1}{r_\varepsilon^{2p}}\int_{r_\varepsilon}^{r_\varepsilon+r_\varepsilon^2}f\(\frac{2}{\gamma_i}\ln\frac{1}{r}\(1+\smallo\(1\)\)+\bigO\(\frac{1}{\gammae}\)\)^prdr\bigg)\nonumber\\
&\qquad=\bigO\(\frac{\overline\gamma_\varepsilon^{p+2}}{r_\varepsilon^{2p}}\int^{\delta_0\overline\gamma_\varepsilon^2/\gamma_i^2}_{\delta_0\overline\gamma_\varepsilon^2/\gamma_i^2-\(2/\gamma_i^2\)\ln\(1+r_\varepsilon\)}\exp\(\(ps-1\)s\gamma_i^2\(1+\smallo\(1\)\)\)ds\)\nonumber\allowdisplaybreaks\\
&\qquad=\bigO\(\frac{\overline\gamma_\varepsilon^p}{r_\varepsilon^{2p}}\ln\(1+r_\varepsilon\)\exp\(\(p\delta_0\frac{\overline\gamma_\varepsilon^2}{\gamma_i^2}-1\)\delta_0\overline\gamma_\varepsilon^2+\smallo\(\overline\gamma_\varepsilon^2\)\)\)\nonumber\allowdisplaybreaks\\
&\qquad=\bigO\(\overline\gamma_\varepsilon^p\exp\(p\delta_0^2\frac{\overline\gamma_\varepsilon^4}{\gamma_i^2}+\(p-\frac{3}{2}\)\delta_0\overline\gamma_\varepsilon^2+\smallo\(\overline\gamma_\varepsilon^2\)\)\)\nonumber\\
&\qquad=\bigO\(\overline\gamma_\varepsilon^p\exp\(\(\frac{p\delta_0}{\(1-\delta'\)^2}+p-\frac{3}{2}\)\delta_0\overline\gamma_\varepsilon^2+\smallo\(\overline\gamma_\varepsilon^2\)\)\)=\smallo\(\frac{1}{\overline\gamma_\varepsilon^p}\),\label{Pr4Eq32}\allowdisplaybreaks\\
&\left\|\chi_{\varepsilon,\tau}f'\big(\widetilde{U}_{\varepsilon,\gamma,\tau}+\Psi_{\varepsilon,\gamma,\tau}\big)\partial_{\tau_i}\big[\widetilde{U}_{\varepsilon,\gamma,\tau}\big]\right\|_{L^p}^p=\bigO\Bigg(\frac{1}{\gamma_i^p}\int_{\Omega_{R_\varepsilon,\tau}}\(\frac{1}{\left|x-\overline{\tau_i}\right|}+1\)^pdx\nonumber\\
&\qquad\quad+\frac{1}{\gamma_i^p}\int_{r_\varepsilon}^{R_\varepsilon}f'\(\frac{2}{\gamma_i}\ln\frac{1}{r}\(1+\smallo\(1\)\)+\bigO\(\frac{1}{\gammae}\)\)^pr^{1-p}dr\Bigg)\allowdisplaybreaks\nonumber\\
&=\bigO\(\frac{1}{\overline\gamma_\varepsilon^p}+\gamma_i^{p+2}\int_{\delta_0/[\overline\gamma_\varepsilon\(1+\delta'\)^2]}^{\delta_0/\(1-\delta'\)^2}\exp\(\(ps+\frac{p}{2}-1\)s\gamma_i^2+\smallo\(\gamma_i^2\)\)ds\)=\bigO\(\frac{1}{\overline\gamma_\varepsilon^p}\)\label{Pr4Eq33}
\end{align}
as $\varepsilon\to0$, uniformly in $(\gamma,\tau)\in \Gamma_\varepsilon^k\(\delta'\)\times T_\varepsilon^k\(\delta\)$, provided we choose $p$ such that 
\begin{multline*}
\frac{p\delta_0}{\(1-\delta'\)^2}+p-\frac{3}{2}<0\quad\text{and}\quad\frac{p\delta_0}{\(1-\delta'\)^2}+\frac{p}{2}-1<0,\\
\text{i.e. }\max\(\frac{1}{2}+\frac{\delta_0}{\(1-\delta'\)^2},\frac{2}{3}\(1+\frac{\delta_0}{\(1-\delta'\)^2}\)\)<\frac{1}{p}<1,
\end{multline*}
which is possible since $\delta'<1-\sqrt{2\delta_0}$. Note that in this case, we cannot choose $p>2$ and so $W^{2,p}\(\Omega\)$ does not embed into $C^1\(\overline\Omega\)$. Furthermore, by proceeding as in \eqref{Pr4Eq25}--\eqref{Pr4Eq26}, we obtain 
\begin{equation}\label{Pr4Eq34}
\left\|\chi_{\varepsilon,\tau}\big(f'\big(\widetilde{U}_{\varepsilon,\gamma,\tau}+\Psi_{\varepsilon,\gamma,\tau}\big)-f'\(w_\varepsilon\)\big)\partial_{\tau_i}\big[\Psi_{\varepsilon,\gamma,\tau}\big]\right\|_{L^p}=\smallo\(\left\|\partial_{\tau_i}\big[\Psi_{\varepsilon,\gamma,\tau}\big]\right\|_{C^0}\)
\end{equation}
as $\varepsilon\to0$, uniformly in $(\gamma,\tau)\in \Gamma_\varepsilon^k\(\delta'\)\times T_\varepsilon^k\(\delta\)$. By putting together \eqref{Pr4Eq29}, \eqref{Pr4Eq32}, \eqref{Pr4Eq33} and \eqref{Pr4Eq34}, we obtain
\begin{equation}\label{Pr4Eq35}
\left\|\partial_{\tau_i}\big[\Psi_{\varepsilon,\gamma,\tau}\big]\right\|_{W^{2,p}}=\bigO\(\frac{1}{\overline\gamma_\varepsilon}\)+\smallo\(\left\|\partial_{\tau_i}\big[\Psi_{\varepsilon,\gamma,\tau}\big]\right\|_{C^0}\)
\end{equation}
as $\varepsilon\to0$, uniformly in $(\gamma,\tau)\in \Gamma_\varepsilon^k\(\delta'\)\times T_\varepsilon^k\(\delta\)$. By using the continuity of the embeddings of $W^{2,p}\(\Omega\)$ into $C^0\(\overline\Omega\)$ and $H^1\(\Omega\)$, \eqref{Pr4Eq3} then follows from \eqref{Pr4Eq35}.

Note that \eqref{Pr4Eq32} corresponds to the first identity in \eqref{Pr4Eq3c}, while the second one follows from \eqref{Pr4Eq33} together with the already proven \eqref{Pr4Eq3} and \eqref{Pr4Eq3b}, which yield
\begin{align*}
\left\|f'\(U_{\ve,\gamma,\tau}\) D_{\tau_i}\Psi_{\ve,\gamma,\tau}\mathbf{1}_{\Omega_{r_\ve,\tau}}\right\|_{L^p}&=\bigO\(\gammae^2\left\|\exp\big(U_{\ve,\gamma,\tau}^2\big)\mathbf{1}_{\Omega_{r_\ve,\tau}}\right\|_{L^p}\left\|D_{\tau_i}\Psi_{\ve,\gamma,\tau}\right\|_{C^0}\)=\bigO\(\gammae\).
\end{align*}
 This ends the proof of Proposition~\ref{Pr4}.
\endproof

\section{Expansions of the bubble and its derivatives}\label{App}

In this section we give a precise asymptotic analysis of spherical solutions, and prove some useful consequences.

\begin{proposition}\label{Pr10}
For every $\gamma>0$, let $\overline{B}_{\gamma}$ be the unique radial solution to the problem 
$$\left\{\begin{aligned}&\Delta\overline{B}_{\gamma}=f\(\overline{B}_{\gamma}\)&&\text{in }\R^2\\&\overline{B}_{\gamma}\(0\)=\gamma,&&\end{aligned}\right.$$ 
where $f\(s\):=s\exp\(s^2\)$ for all $s\in\R$. Set
\begin{equation}\label{defmugamma}
\mu_\gamma^2:=4\gamma^{-2}\exp\(-\gamma^2\)\quad\text{and}\quad t\(r\):=\ln\(1+r^2\)\quad\forall r\ge0
\end{equation}
and let $\varphi$ be the unique radial solution to the problem
$$\left\{\begin{aligned}&\Delta\varphi=4e^{-2t}\(t^2 -t+2\varphi\) &&\text{in }\R^2\\&\varphi\(0\)=0.&&\end{aligned}\right.$$
Then 
$$\overline{B}_{\gamma}\(r\)=\gamma-\frac{t\(r/\mu_\gamma\)}{\gamma}+\frac{\varphi\(r/\mu_\gamma\)}{\gamma^3}+D_\gamma\(r/\mu_\gamma\),$$
where
\begin{equation}\label{Pr10Eq1}
D_\gamma\(r\)=\bigO\(\frac{t\(r\)}{\gamma^5}\)\quad\text{and}\quad D'_\gamma\(r\)=\bigO\(\frac{1}{\gamma^5r}\)
\end{equation}
as $\gamma\to\infty$, uniformly in $r\in\(0,\mu_\gamma^{\delta-1}\)$, $\delta\in\(0,1\)$ fixed. Furthermore, $\varphi\(r\)\sim-t\(r\)$ and $\varphi'\(r\)\sim-t'\(r\)$ as $r\to\infty$.
\end{proposition}

\proof
This was originally proven in~\cite{DruThi}, see Claim 5.1 and estimates (5.8) and (5.9) in particular (note that the function $B_\gamma$ in~\cite{DruThi} corresponds to the function $\overline{B}_\gamma$ via the relation $B_\gamma\(r\)=\overline{B}_\gamma(r/2)$). The estimates (5.8)--(5.9) in~\cite{DruThi} are valid as long as $0\le t(r/\mu_\gamma)\le\gamma^2-T_\gamma$, where $T_\gamma$ is chosen so that $\gamma^ke^{-T_\gamma}=\smallo\(1\)$ as $\gamma\to\infty$ for every $k\ge 0$. It is not difficult to see that this condition is satisfied uniformly for $0\le r\le \mu_\gamma^\delta$, for any fixed $\delta>0$.
\endproof

With regard to the derivative of $\overline{B}_{\gamma}$ with respect to $\gamma$, we obtain the following:

\begin{proposition}\label{Pr11}
Let $\overline{B}_{\gamma}$, $\mu_\gamma$, $t$ and $\varphi$ be as in Proposition~\ref{Pr10}. Set  $\overline{Z}_0\(r\):=\frac{1-r^2}{1+r^2}$ and let $\psi$ be the unique radial solution to the problem
$$\left\{\begin{aligned}
&\Delta\psi=4e^{-2t}\(\overline{Z}_0\(1-4t+2t^2+4\varphi\)+2\psi\)&&\text{in }\R^2\\
&\psi\(0\)=0.&&
\end{aligned}\right.$$
Then 
\begin{equation*}
Z_{0,\gamma}\(r\):=\partial_\gamma\[\overline{B}_{\gamma}\(r\)\]=\overline{Z}_0(r/\mu_\gamma)+\frac{\psi\(r/\mu_\gamma\)}{\gamma^2}+E_\gamma(r/\mu_\gamma),
\end{equation*}
where
\begin{equation}\label{Pr11Eq1}
E_\gamma\(r\)=\bigO\(\frac{1+t\(r\)}{\gamma^4}\)\quad\text{and}\quad E_\gamma'\(r\)=\bigO\(\frac{1}{\gamma^4r}\) 
\end{equation}
as $\gamma\to\infty$, uniformly in $r\in\(0,\mu_\gamma^{\delta-1}\)$, $\delta\in\(0,1\)$ fixed. 
Furthermore, $\psi\(r\)\sim t\(r\)$ and $\psi'\(r\)\sim t'\(r\)$ as $r\to\infty$.
\end{proposition}

\proof
We easily see that
$$\left\{\begin{aligned}
&\Delta Z_{0,\gamma}=f'\big(\overline{B}_\gamma\big)Z_{0,\gamma}&&\text{in }B(0,\mu_\gamma^\delta)\\
&Z_{0,\gamma}\(0\)=1,&&
\end{aligned}\right.$$
with $f\(s\)=se^{s^2}$. Set 
$$E_\gamma\(r\):=Z_{0,\gamma}\(\mu_\gamma r\)-\overline{Z}_0\(r\)-\frac{\psi\(r\)}{\gamma^2}$$
and observe that
$$\Delta \overline{Z}_0=8e^{-2t}\overline{Z}_0,$$
so that
\begin{equation}\label{Pr11Eq4}
\left\{\begin{aligned}
&\Delta E_\gamma=\mu_\gamma^2 f'\(\overline{B}_\gamma\(\mu_\gamma \cdot\)\)Z_{0,\gamma}\(\mu_\gamma \cdot\)-8e^{-2t}\overline{Z}_0-\frac{\Delta \psi}{\gamma^2}&&\text{in }B\(0,\mu_\gamma^{\delta-1}\)\\
&E_\gamma\(0\)=0.&&
\end{aligned}\right.
\end{equation}
In order to expand the right-hand side of \eqref{Pr11Eq4} we use \eqref{Pr10Eq1}, $\varphi=\bigO\(1+t\)$ and recalling that $\mu_\gamma^2\gamma^2e^{\gamma^2}=4$, we find
\begin{align}
f'\(\overline{B}_\gamma\(\mu_\gamma \cdot\)\)&=\big(1+2\overline{B}_\gamma^2\(\mu_\gamma \cdot\)\big)\exp\big(\overline{B}_\gamma^2\(\mu_\gamma \cdot\)\big)\notag\allowdisplaybreaks\\
&=\[1+2\(\gamma-\frac{t}{\gamma}+\frac{\varphi}{\gamma^3}+\bigO\(\frac{t}{\gamma^5}\)\)^2\]  e^{\(\gamma-\frac{t}{\gamma}+\frac{\varphi}{\gamma^3}+\bigO\(\frac{t}{\gamma^5}\)\)^2}\notag\\
&=\frac{4e^{-2t}}{\mu_\gamma^2}\[\frac{1}{\gamma^2}+2-\frac{4t}{\gamma^2}+\bigO\(\frac{1+t^2}{\gamma^4}\)\] e^{\frac{t^2}{\gamma^2}} e^{\frac{2\varphi}{\gamma^2}+\bigO\(\frac{1+t^2}{\gamma^4}\)}.\label{Pr11Eq5}
\end{align}
Using that $e^s=1+s+\bigO(s^2)e^s$ for $s>0$, we write
$$e^{\frac{t^2}{\gamma^2}}=1+\frac{t^2}{\gamma^2}+\bigO\(\frac{t^4}{\gamma^4}\) e^{\frac{t^2}{\gamma^2}},$$
and using that $t=\bigO\(\gamma^2\)$ uniformly on $\(0,\mu_\gamma^{\delta-1}\)$,
$$e^{\frac{2\varphi}{\gamma^2}+\bigO\(\frac{1+t^2}{\gamma^4}\)}=1+\frac{2\varphi}{\gamma^2}+\bigO\(\frac{1+t^2}{\gamma^4}\).$$
We now multiply and reorder, using that $\exp\(t^2/\gamma^2\)\ge 1$, to obtain
\begin{align*}
f'\(\overline{B}_\gamma\(\mu_\gamma\cdot\)\)&=\frac{4e^{-2t}}{\mu_\gamma^2}\(2+\frac{1}{\gamma^2}\(1-4t+4\varphi\)+\bigO\(\frac{1+t^4}{\gamma^4}\)\)e^{\frac{t^2}{\gamma^2}}\\
&=\frac{4e^{-2t}}{\mu_\gamma^2}\(2+\frac{1}{\gamma^2}\(1-4t+2t^2+4\varphi\)\)+\frac{e^{-2t+\frac{t^2}{\gamma^2}}}{\mu_\gamma^2}\bigO\(\frac{1+t^4}{\gamma^4}\).
\end{align*}
Together with \eqref{Pr11Eq4} and using that $\psi=\bigO\(1+t\)$ (as we shall prove later), we now estimate
\begin{align*}
\Delta E_\gamma
&=\mu_\gamma^2\(f'\(\overline{B}_\gamma\(\mu_\gamma\cdot\)\)\overline{Z}_{0}+\frac{f'\(\overline{B}_\gamma\(\mu_\gamma\cdot\)\)\psi}{\gamma^2}+  f'\(\overline{B}_\gamma\(\mu_\gamma\cdot\)\)E_\gamma\)\\
&\quad -8e^{-2t}\overline{Z}_0 -4e^{-2t}\(\frac{\overline{Z}_0}{\gamma^2}\(1-4t+2t^2+4\varphi\) +\frac{2\psi}{\gamma^2} \)\\
&=\mu_\gamma^2f'\(\overline{B}_\gamma\(\mu_\gamma\cdot\)\)E_\gamma +e^{-2t+\frac{t^2}{\gamma^2}}\bigO\(\frac{1+t^4}{\gamma^4}\).
\end{align*}
We now go back to \eqref{Pr11Eq5} and, still using that $t=\bigO\(\gamma^2\)$ on $B\(0,\mu_\gamma^{\delta-1}\)$, we bound
$$f'\(\overline{B}_\gamma\(\mu_\gamma\cdot\)\)=\bigO\(\frac{1}{\mu_\gamma^2}e^{-2t+\frac{t^2}{\gamma^2}}\),$$
so that
\begin{equation}\label{Pr11Eq5b}
\Delta E_\gamma=e^{-2t+\frac{t^2}{\gamma^2}}\(\bigO\(\left|E_\gamma\right|\)+\bigO\(\frac{1+t^4}{\gamma^4}\) \).
\end{equation}
Multiplying by $\gamma^4$ and using ODE theory, we see that
\begin{equation*}
\gamma^4E_\gamma\longrightarrow \widetilde E_\infty\quad \text{in }C^1_{\loc}\(\R^2\).\
\end{equation*}
In particular, for any fixed $T>0$ and for $\gamma$ large ($\gamma\ge \gamma_0\(T\)$), we have
\begin{equation}\label{Pr11Eq5c}
\left|E_\gamma\right|\le  \frac{C\(T\)}{\gamma^4}\quad\text{and}\quad\left|E'_\gamma\right|\le  \frac{C'\(T\)}{\gamma^4}\quad \text{on }\[0,T\].
\end{equation}
From now on, it is understood that $\gamma\ge \gamma_0\(T\)$, so that \eqref{Pr11Eq5c} holds.
In order to prove \eqref{Pr11Eq1}, observe that the first identity in \eqref{Pr11Eq1} follows from the second one and \eqref{Pr11Eq5c} by integration over $\[T,r\]$. Then, for $T,M>0$ to be chosen later, set
$$R_\gamma:=\sup\left\{r\in \(T,\mu_\gamma^{\delta-1}\]:\,\left|E_\gamma'\(\rho\)\right|\le \frac{M}{\gamma^4\rho},\,\forall \rho\in \[T,r\]\right\}.$$
We shall prove that for $T$ and $M$ suitable, we have $R_\gamma=\mu_\gamma^{\delta-1}$ for every $\gamma$ sufficiently large.

Arguing by contradiction, assume that $R_\gamma<\mu_\gamma^{\delta-1}$,
so that in particular
\begin{equation}\label{Pr11Eq5e}
\left|E'_\gamma(R_\gamma)\right|=\frac{M}{R_\gamma \gamma^4}.
\end{equation}
By definition of $R_\gamma$, using \eqref{Pr11Eq5c} and integrating, we get
\begin{equation}
\left|E_\gamma\(r\)\right|\le \left|E_\gamma\(T\)\right|+\int_T^r \frac{M}{\gamma^4\rho}d\rho\le \frac{C\(T\)}{\gamma^4}+\frac{Mt\(r\)}{2\gamma^4} \label{Pr11Eq6}  \quad \text{on }\[T,R_\gamma\].
\end{equation}
With the divergence theorem, \eqref{Pr11Eq5b} and \eqref{Pr11Eq6}, we now bound for $t\in [T,R_\gamma]$,
\begin{align}
\left|2\pi r E'_\gamma\(r\)\right|&\le\left|2\pi TE'_\gamma\(T\)\right|+\int_{B\(0,r\)\setminus B\(0,T\)}|\Delta E_\gamma\(x\)dx|\nonumber\\
&\le \frac{2\pi TC'\(T\)}{\gamma^4}+\int_{B\(0,r\)\setminus B\(0,T\)}e^{-2t+\frac{t^2}{\gamma^2}}\(\widetilde C|E_\gamma|+\widetilde C\(\frac{1+t^4}{\gamma^4}\) \)dx\nonumber\allowdisplaybreaks\\
&\le\frac{2\pi TC'\(T\)}{\gamma^4}+ \frac{\widetilde CM}{2\gamma^4}\int_{B\(0,r\)\setminus B\(0,T\)}e^{-2t+\frac{t^2}{\gamma^2}} t dx\nonumber\\
&\quad  + \frac{1}{\gamma^4}\int_{B\(0,r\)\setminus B\(0,T\)}e^{-2t+\frac{t^2}{\gamma^2}} \widetilde C\(C\(T\)+1+t^4\)dx\nonumber\\
&=:\frac{2\pi T C'\(T\)}{\gamma^4}+\frac{(I_\gamma)}{\gamma^4}+\frac{(II_\gamma)}{\gamma^4}.\label{Pr11Eq6b}
\end{align}
Observing that
$$-2t+\frac{t^2}{\gamma^2}\le -\(1+\delta\)t\quad\text{and}\quad e^{-\frac{\delta}{2}t}t^k =\bigO\(1\)\quad \text{on } B\(0,\mu_\gamma^{\delta-1}\),\quad \forall k\ge 0,$$
we bound
\begin{align*}
\int_{B\(0,\mu_\gamma^{\delta-1}\)\setminus B\(0,T\)}e^{-2t+\frac{t^2}{\gamma^2}} t^k dx
&=\bigO\(\int_{B\(0,T\)^c}e^{-\(1+\frac{\delta}{2}\)t}dx\)=\smallo_T\(1\),
\end{align*}
with $\smallo_T\(1\)\to 0$ as $T\to\infty$. We can therefore choose $T$ sufficiently large (independent of $M$) so that
$$(I_\gamma)\le \frac{\pi M}{2}.$$
Then, choosing $M$ sufficiently large (depending on $T$), so that
$$2\pi T C'\(T\) +(II_\gamma)\le \frac{\pi M}{2} $$
and dividing by $2\pi$ in \eqref{Pr11Eq6b}, we finally obtain
$$r\left|E'\(r\)\right|\le \frac{M}{2\gamma^4},\quad \forall r\in\[T,R_\gamma\],$$
which for $r=R_\gamma$ is a contradiction to \eqref{Pr11Eq5e}. Therefore $R_\gamma=\mu_\gamma^{\delta-1}$.

To prove that $\psi\(r\)\sim ~t\(r\)$ and $\psi'\(r\)\sim ~t'\(r\)$ as $r\to\infty$, we recall from~\cite{ManMar}*{Lemmas~15 and~16} (see also~\cite{DruThi}*{Lemma 5.1}) that if $\psi$ is radially symmetric and solves
$$\Delta\psi=4e^{-2t}\(g+2\psi\),$$ 
with $g\(r\)=\bigO\big(\(\ln r\)^k\big)$ as $r\to\infty$ for some $k\ge 1$, then
$$\psi\(r\)= \beta \ln r+\bigO\(r\),\quad \psi'\(r\)=\frac{\beta}{r}+\bigO\(\frac{\(\ln r\)^k}{r^3}\), \quad \beta:=\frac{2}{\pi}\int_{\R^2}\overline{Z}_0e^{-2t}gdx,$$
as $r\to\infty$.
With $g=\overline Z_0\(1-4t+2t^2+4\varphi\)$ we compute
\begin{align*}
&\int_{\R^2}\overline{Z}_0^2e^{-2t}dx=\frac{\pi}{3},\quad \int_{\R^2}\overline{Z}_0^2e^{-2t}4t dx=\frac{16\pi}{9},\allowdisplaybreaks\\
&\int_{\R^2}\overline{Z}_0^2e^{-2t}2t^2dx=\frac{70\pi}{27}\quad\text{and}\quad\int_{\R^2}\overline{Z}_0^2e^{-2t}4\varphi dx=-\frac{4\pi}{27},
\end{align*}
so that $\beta=2$.
\endproof

Let us see a few consequences of the above estimates.

\begin{proposition}\label{Pr12} 
Let $\overline{B}_{\gamma}$, $\mu_\gamma$, $t$, $\overline{Z}_0$ and $Z_{0,\gamma}$ be as in Propositions~\ref{Pr10} and~\ref{Pr11}. Given $\delta\in \(0,1\)$, $a,b\ge 0$ and $t_\gamma\(r\)=t(r/\mu_\gamma)$, we have
\begin{equation}
\int_{B\(0,r\)}\exp\big(\overline B_\gamma^2\big)\overline B_\gamma^b\(1+\bigO\(\frac{t_\gamma}{\gamma^2}\)\)^a dx=4\pi\gamma^{b-2} +\bigO\(\gamma^{b-4}\),\label{Pr12Eq1}
\end{equation}
as $\gamma\to\infty$, uniformly for $\gamma\mu_\gamma \le r\le\mu_\gamma^\delta$.
Moreover,
\begin{equation}\label{Pr12Eq3}
\int_{B\(0,r\)}f'\(\overline B_\gamma\)  Z_{0,\gamma} dx=\frac{-4\pi +\smallo\(1\)}{\gamma^2}
\end{equation}
as $\gamma\to\infty$, uniformly for $\gamma \mu_\gamma=\smallo\(r\)$ and $r\le\mu_\gamma^\delta$, and
\begin{equation}\label{Pr12Eq4}
\int_{B\(0,r\)}f'\(\overline B_\gamma\(x\)\)  \frac{2x_1^2}{\mu_\gamma^2+\left|x\right|^2} dx=4\pi+\bigO\(\frac{1}{\gamma^2}\)
\end{equation}
as $\gamma\to\infty$, uniformly for $\gamma\mu_\gamma \le r\le\mu_\gamma^\delta$.
\end{proposition}

\proof 
Using Proposition~\ref{Pr10} and noticing that $\varphi=\bigO\(1+t\)$, $t_\gamma =\bigO\(\gamma^2\)$ in $B\(0,r\)$ for $r\le\mu_\gamma^\delta$, we write
\begin{equation}\label{Pr12Eq5}
\exp\big(\overline B_\gamma^2\big)=e^{\[\gamma-\frac{t_\gamma}{\gamma}+\bigO\(\frac{1+t_\gamma}{\gamma^3}\)\]^2}=e^{\gamma^2}e^{-2t_\gamma}e^{\frac{t_\gamma^2}{\gamma^2}+\bigO\(\frac{1+t_\gamma}{\gamma^2}\)}=\frac{4e^{-2t_\gamma}}{\mu_\gamma^2 \gamma^2}\(1+\bigO\(\frac{t_\gamma^2}{\gamma^2}\)e^{\frac{t_\gamma^2}{\gamma^2}}\),
\end{equation}
where we used the inequality $|e^x-1|\le \left|x\right|e^{\left|x\right|}$ to estimate
$$\bigg|e^{\bigO\(\frac{t_\gamma^2}{\gamma^2}\)}-1 \bigg|=\bigO\(\frac{t_\gamma^2}{\gamma^2}\)e^{\frac{t_\gamma^2}{\gamma^2}}.$$
Further, we use Proposition~\ref{Pr10} together with $\(1+x\)^a=1+\bigO\(x\)$ uniformly for $x=\bigO\(1\)$, to bound 
\begin{align}\label{Pr12Eq5bis}
\overline{B}_\gamma^b\(1+\bigO\(\frac{t_\gamma}{\gamma^2}\)\)^a&=\(\gamma+\bigO\(\frac{1+t_\gamma}{\gamma}\)\)^b\(1+\bigO\(\frac{t_\gamma}{\gamma^2}\)\)^a\nonumber\\
&=\gamma^b\(1+\bigO\(\frac{1+t_\gamma}{\gamma}\)\)^{a+b}=\gamma^b\(1+\bigO\(\frac{1+t_\gamma^2}{\gamma^2}\)\).
\end{align}
We can then estimate the left-hand side of \eqref{Pr12Eq1} as
\begin{multline*}
\int_{B\(0,r\)}\frac{4e^{-2t_\gamma}}{\mu_\gamma^2 \gamma^{2-b}}\(1+\bigO\(\frac{1+t_\gamma^2}{\gamma^2}e^{\frac{t_\gamma^2}{\gamma^2}}\)\)dx\\
=\int_{B\(0,r/\mu_\gamma\)}\frac{4e^{-2t}}{\gamma^{2-b}}\(1+\bigO\(\frac{1+t^2}{\gamma^2}e^{\frac{t^2}{\gamma^2}}\)\)dx=\frac{4\pi}{\gamma^{2-b}}\int_0^{r/\mu_\gamma}\frac{2\rho}{\(1+\rho^2\)^2}d\rho\\
+\frac{1}{\gamma^{4-b}} \int_{B\(0,r/\mu_\gamma\)}\bigO\(\(1+t^2\) e^{-t\(2-\frac{t}{\gamma^2}\)}\)dx=:\(I\)_\gamma+(II)_\gamma.
\end{multline*}
Using that $r\ge \gamma\mu_\gamma$, one computes
$$\(I\)_\gamma=\frac{4\pi}{\gamma^{2-b}}\(1+\bigO\(\gamma^{-2}\)\),$$
and using that $0\le t/\gamma^2\le \(1-\delta+\smallo\(1\)\)$ in $B\(0,r/\mu_\gamma\)$ as $\gamma\to\infty$, uniformly for $r\le\mu_\gamma^\delta$, and observing that $\(1+t^2\)e^{-\(1+\delta'\)t}\in L^1\(\R^2\)$ for every $\delta'>0$, one has
\begin{equation}\label{Pr12Eq6}
(II)_\gamma=\bigO\(\frac{1}{\gamma^{4-b}}\int_{B\(0,r/\mu_\gamma\)}\(1+t^2\)e^{-t\(1+\delta+\smallo\(1\)\)}\)=\bigO\(\frac{1}{\gamma^{4-b}}\)
\end{equation}
as $\gamma\to\infty$, uniformly for $r\le\mu_\gamma^\delta$, so that \eqref{Pr12Eq1} is proven.


In order to prove \eqref{Pr12Eq3} we use Proposition~\ref{Pr11} to expand $Z_{0,\gamma}$ and compute
\begin{align*}
\int_{B\(0,r\)}f'\(\overline B_\gamma\)Z_{0,\gamma}dx&=\int_{B\(0,r\)}\partial_\gamma\[f\(\overline B_\gamma\)\] dx=\int_{B\(0,r\)}\Delta Z_{0,\gamma} dx\allowdisplaybreaks\\
&=\int_{B\(0,r\)}\Delta \(\overline Z_0\(\frac{x}{\mu_\gamma}\)+\frac{1}{\gamma^2}\psi\(\frac{x}{\mu_\gamma}\)+E_\gamma\(\frac{x}{\mu_\gamma}\)\) dx\\
&=-\frac{2\pi r}{\mu_\gamma}\(\overline Z_0'\(\frac{r}{\mu_\gamma}\) +\frac{1}{\gamma^2}\psi'\(\frac{r}{\mu_\gamma}\)+E_\gamma'\(\frac{r}{\mu_\gamma}\)\).
\end{align*}
A direct computation shows
$$\frac{r}{\mu_\gamma}\overline Z_0'\(\frac{r}{\mu_\gamma}\)=\bigO\(\frac{\mu_\gamma^2}{r^2}\)=\smallo\(\frac{1}{\gamma^2}\)$$
as $\gamma\to\infty$, uniformly for $\gamma\mu_\gamma=\smallo\(r\)$. Using that 
$$\psi'\(s \)=t'\(s\)\(1+\smallo\(1\)\)=\frac{2s}{1+s^s}=\frac{2}{s}\(1+\smallo\(1\)\)\quad \text{as } s\to \infty,$$
we obtain
$$\frac{r}{\mu_\gamma}\psi'\(\frac{r}{\mu_\gamma}\)= 2+\smallo\(1\).$$
Finally, from the second part of \eqref{Pr11Eq1}, we infer
$$\frac{r}{\mu_\gamma}E_\gamma'\(\frac{r}{\mu_\gamma}\)=\bigO\(\frac{1}{\gamma^4}\)=\smallo\(\frac{1}{\gamma^2}\).$$
Summing up, \eqref{Pr12Eq3} follows at once.

It remains to prove \eqref{Pr12Eq4}. Using \eqref{Pr12Eq5} and \eqref{Pr12Eq5bis}, we write
\begin{align*}
&\int_{B\(0,r\)}f'\(\overline B_\gamma\(x\)\)  \frac{2x_1^2}{\mu_\gamma^2+\left|x\right|^2} dx\\
&=4\gamma^2 e^{\gamma^2}\int_{B\(0,r\)}\(1+\bigO\(\frac{1+t_\gamma}{\gamma^2}\)\)e^{-3t_\gamma}\(1+\bigO\(\frac{t_\gamma^2}{\gamma^2}\)e^{\frac{t_\gamma^2}{\gamma^2}}\)\(\frac{y_1}{\mu_\gamma}\)^2dy\allowdisplaybreaks\\
&=\int_{B\(0,r/\mu_\gamma\)}16e^{-3t}y_1^2dy+\bigO\(\int_{B\(0,r/\mu_\gamma\)} \frac{1+t^2}{\gamma^2}e^{-t\(3-t/\gamma^2\)}y_1^2dy\)=:\(I\)_\gamma+(II)_\gamma
\end{align*}
To compute $\(I\)_\gamma$, we observe that its value does not change if we replace $y_1$ with $y_2$, so that
\begin{align*}
\(I\)_\gamma&=\frac{1}{2}\int_{B\(0,r/\mu_\gamma\)}16e^{-3t}\left|x\right|^2dy=16\pi\[-\frac{1+2\rho^2}{4\(1+\rho^2\)^2}\]_{\rho=0}^{r/\mu_\gamma}=4\pi +\smallo\(1\)
\end{align*}
as $\gamma\to\infty$, uniformly for $r\ge\gamma\mu_\gamma$. The term $(II)_\gamma$ can be estimated as in \eqref{Pr12Eq6} since $y_1^2\le e^t$, so that
$$(II)_\gamma=\bigO\(\int_{B\(0,r/\mu_\gamma\)} \frac{t^2}{\gamma^2}e^{-t\(1+\delta+\smallo\(1\)\)}dx\)=\bigO\(\frac{1}{\gamma^2}\)$$
as $\gamma\to\infty$, uniformly for $r\le\mu_\gamma^\delta$.
\endproof

\begin{proposition}\label{Pr13} 
Let $\overline{B}_{\gamma}$ and $\mu_\gamma$ be as in Proposition~\ref{Pr10}. Given $\delta_0\in\(0,1/2\)$, we have
\begin{equation}\label{Pr13Eq1}
\int_{B\(0,r\)}\exp\big(\overline B_\gamma\(x\)^2\big)\left|x\right|dx =\bigO\big(\mu_\gamma^{3\delta_0-2\delta_0^2+\smallo\(1\)}\big)
\end{equation}
as $\gamma\to\infty$, uniformly for $r=\bigO\(\mu_\gamma^{\delta_0}\)$.
\end{proposition}

\proof 
Let $t_\gamma$ be as in Proposition~\ref{Pr11}. Using Proposition~\ref{Pr10}, we write
\begin{align}\label{Pr13Eq3}
\exp\big(\overline B_\gamma^2\big)&=\exp\(\[\gamma-\frac{t_\gamma}{\gamma}+\bigO\(\frac{1+t_\gamma}{\gamma^3}\)\]^2\)=e^{\gamma^2}e^{-2t_\gamma+\frac{t_\gamma^2}{\gamma^2}}e^{\bigO\(\frac{1+t_\gamma}{\gamma^2}\)}\nonumber\\
&=\bigO\(\frac{e^{-t_\gamma\(2-t_\gamma/\gamma^2\)}}{\mu_\gamma^2 \gamma^2}\),\quad \text{for } r=\bigO\(\mu_\gamma^{\delta_0}\),
\end{align}
Then, using that 
\begin{equation}\label{Pr13Eq4}
t_\gamma\(r\)\le\(1-\delta_0+\smallo\(1\)\)\gamma^2,\quad\text{for }r=\bigO\(\mu_\gamma^{\delta_0}\),
\end{equation}
together with a change of variables, we get 
\begin{align*}
&\int_{B\(0,r\)}\exp\big(\overline B_\gamma\(x\)^2\big)\left|x\right|dx=\bigO\(\int_{B\(0,r\)}\frac{e^{-t_\gamma\(2-t_\gamma/\gamma^2\)}}{\mu_\gamma^2 \gamma^2}\left|x\right|dx\)\allowdisplaybreaks\\
&\qquad=\bigO\(\mu_\gamma\int_{B\(0,r/\mu_\gamma\)}\frac{\left|y\right|dy}{\big(1+\left|y\right|^2\big)^{2-t/\gamma^2}}\)=\bigO\(\mu_\gamma\int_{B\(0,r/\mu_\gamma\)}\frac{\left|y\right|dy}{\big(1+\left|y\right|^2\big)^{1+\delta_0+\smallo\(1\)}}\)\\
&\qquad=\bigO\(\mu_\gamma \(\frac{r}{\mu_\gamma}\)^{1-2\delta_0+\smallo\(1\)}\)=\bigO\(\mu_\gamma^{3\delta_0-2\delta_0^2+\smallo\(1\)} \)
\end{align*}
as $\gamma\to\infty$, uniformly for $r=\bigO\(\mu_\gamma^{\delta_0}\)$, which proves \eqref{Pr13Eq1}.
%
\endproof

\section{Poincar\'e--Sobolev inequalities}\label{SecPS}

The standard Poincar\'e--Sobolev inequality on $\Sph^2$ says that for every $p\in [1,\infty)$ there exists $C_p>0$ such that for every $\phi \in H^1(\Sph^2)$ with $\int_{\Sph^2}\phi dv_{\Sph^2}=0$, we have
\begin{equation}\label{PSA}
\int_{\Sph^2} \left|\phi\right|^p dv_{\Sph^2} \le C_p\(\int_{\Sph^2} \left|\nabla \phi\right|^2 dv_{\Sph^2}\)^\frac{p}{2}.
\end{equation}
Pulling back the spherical metric onto $\R^2$, we can also rewrite \eqref{PSA} as 
\begin{equation}\label{PSB}
\int_{\R^2}  \left|\phi\right|^p e^{-2t} dx \le C_p\(\int_{\R^2} \left|\nabla \phi\right|^2 dx\)^\frac{p}{2},
\end{equation}
for every $\phi \in D^{1,2}\(\R^2\)$ such that $\int_{\R^2} \phi  e^{-2t}dx=0$, where $t\(x\):=\ln\big(1+\left|x\right|^2\big)$, so that $4e^{-2t\(x\)}=4\big(1+\left|x\right|^2\big)^{-2}$ is the conformal factor of the pull-back metric.

\smallskip
We will need a perturbed version of \eqref{PSB}, where we replace $e^{-2t}$ with suitable scaled versions of $\exp\big(\overline B_\gamma^2\big)$.

\begin{lemma}\label{Pr14Lem} 
Let $(\chi_\ve)_{\ve>0}$ be a sequence of functions in $\R^2$ such that for every $q>1$, we have $\chi_\ve\to\chi_0$ as $\varepsilon\to0$ in $L^q\(\R^2,e^{-2t}dx\)$, i.e.
$$\int_{\R^2}\left|\chi_\ve-\chi_0\right|^q e^{-2t}dx\longrightarrow 0$$
for some function $\chi_0$ in $\R^2$ and further assume that
\begin{equation}\label{Pr14LEq1}
\int_{\R^2} \chi_0 e^{-2t}dx \ne 0.
\end{equation}
Then, for every $p\in\[1,\infty\)$, there exists a constant $C>0$ (depending on $p$ and $(\chi_\ve)$) such that for $\ve>0$ small enough, the following holds:
\begin{equation}\label{PSB'}
\int_{\R^2} \left|\phi\right|^p e^{-2t}dx \le C\(\int_{\R^2}\left|\nabla\phi\right|^2 dx \)^\frac{p}{2}
\end{equation}
for every $\phi\in D^{1,2}\(\R^2\)$ such that $\int_{\R^2}\phi \chi_\ve e^{-2t}dx=0$.
\end{lemma}

\proof
Assume by contradiction that there exists a sequence $(\phi_\ve)_\ve$ in $D^{1,2}\(\R^2\)$ such that
\begin{equation}\label{Pr14LEq2}
\int_{\R^2}|\phi_\ve|^pe^{-2t}dx=1, \quad \int_{\R^2}\phi_\ve \chi_\ve e^{-2t}dx=0,\quad \lim_{\ve\to 0}\int_{\R^2}|\nabla \phi_\ve|^2dx= 0.
\end{equation}
Let $\Pi:\Sph^2\to\R^2$ be the stereographic projection. By the first equation in \eqref{Pr14LEq2},  the average of $\phi_\ve\circ \Pi$ on $\Sph^2$ is bounded, so by the Sobolev--Poincar\'e inequality and weak compactness, up to a subsequence, $\phi_\ve\circ\Pi\to \phi_0\circ\Pi$ strongly in $L^{q'}(\Sph^2)$, in $L^p(\Sph^2)$, and weakly in $H^1(\Sph^2)$, for some function $\phi_0\in L^p(\R^2,e^{-2t}dx)$. By lower-semicontinuity of the Dirichlet integral we get $\|\nabla(\phi_0\circ\Pi)\|_{L^2(\Sph^2)}=\|\nabla \phi_0\|_{L^2(\R^2)}=0$, so that $\phi_0$ is constant, non-zero since $\|\phi_0\circ\Pi\|_{L^p(\Sph^2)}=1$. Then, we obtain
$$0=\int_{\R^2}\phi_\ve \chi_\ve  e^{-2t} dx\to \int_{\R^2}\phi_0 \chi_0 e^{-2t}dx \quad \Rightarrow \quad \int_{\R^2}\chi_0 e^{-2t}dx=0,$$
contradicting our assumption.
\endproof

\begin{proposition}\label{Pr14}
Let $\overline{B}_\gamma$, $\mu_\gamma$ and $t_\gamma$ be as in Propositions~\ref{Pr10} and~\ref{Pr11}. Let $\phi \in D^{1,2}\(\R^2\)$ be such that 
\begin{equation}\label{Pr14Eq1}
\int_{B_r} f\big(\overline{B}_\gamma\big)\phi dx =0
\end{equation}
for $r$ such that $\mu_\gamma=\smallo\(r\)$ and $r =\bigO\(\mu_\gamma^{\delta_0}\)$ for some $\delta_0\in \(0,1\)$. Then for every $p\in [1,\infty)$, we have
\begin{equation}\label{Pr14Eq2}
\int_{B\(0,r\)} \exp\big(\overline B_\gamma^2\big)\(1+t_\gamma\) \left|\phi\right|^p dx=\bigO\(\frac{1}{\gamma^2}\(\int_{\R^2}\left|\nabla\phi\right|^2dx\)^{\frac{p}{2}}\).
\end{equation}
\end{proposition}
\begin{proof}
With Proposition~\ref{Pr10} we can rewrite condition \eqref{Pr14Eq1} as
\begin{align*}
0&=\gamma\int_{B\(0,r\)} \(1+\bigO\(\frac{1+t_\gamma}{\gamma^2}\)
\)e^{\gamma^2}e^{-2t_\gamma +\frac{t_\gamma^2}{\gamma^2}+\bigO\(\frac{1+t_\gamma}{\gamma^2}\)} \phi dx\\
&=\frac{4}{\gamma}\int_{B\(0,r/\mu_\gamma\)} \(1+\bigO\(\frac{1+t}{\gamma^2}\)
\)e^{-2t +\frac{t^2}{\gamma^2}+\bigO\(\frac{1+t}{\gamma^2}\)} \phi (\mu_\gamma\cdot) dx=\frac{4}{\gamma}\int_{\R^2}e^{-2t}\chi_\gamma \tilde \phi dy,
\end{align*}
where $\widetilde{\Phi}\(y\) =\phi(\mu_\gamma y)$, and we claim that
\begin{multline}\label{Pr14Eq3}
\chi_\gamma:= \mathbf{1}_{B\(0,r/\mu_\gamma\)}\(1+\bigO\(\frac{1+t}{\gamma^2}\)\)e^{\frac{t^2}{\gamma^2}+\bigO\(\frac{1+t}{\gamma^2}\)}\longrightarrow \chi_0\equiv 1\\
\text{in }L^q(\R^2,e^{-2t}dx)\text{ for }1\le  q< \frac{1}{1-\delta_0}.
\end{multline}
Indeed, it is clear that $\chi_\gamma\to \chi_0$ pointwise, while we can uniformly bound $\chi_\gamma$ by a function in $L^q\(\R^2, e^{-2t}dx\)$ as follows. By using \eqref{Pr13Eq4}, we obtain
$$\chi_\gamma=\bigO\Big(e^{\frac{t^2}{\gamma^2}}\Big)=\bigO\big(e^{t\(1-\delta_0+\smallo\(1\)\)}\big),\quad\text{so that }\chi_\gamma^q=\bigO\big(e^{t q\(1-\delta_0+\smallo\(1\)\)}\big).$$
On the other hand,
\begin{align*}
\int_{\R^2}e^{t q\(1-\delta_0+\smallo\(1\)\)}e^{-2t}dx&=\int_{\R^2}e^{-t \(2-q+q\delta_0+\smallo\(1\)\)}dx=\int_{\R^2}\bigO\(\frac{1}{1+\left|x\right|^{4-2q+2q\delta_0+\smallo\(1\)}}\)dx\\
&=\bigO\(1\)\quad \text{for } 4-2q+2q\delta_0>2,\quad \text{i.e. }1\le q<\frac{1}{1-\delta_0},
\end{align*}
so that \eqref{Pr14Eq3} follows by dominated convergence.

We can then apply Lemma~\ref{Pr14Lem} to $\widetilde{\Phi}$, so that \eqref{PSB'} holds. On the other hand, for any $r\in [1,\infty)$,
\begin{multline}\label{Pr14Eq4}
\int_{B\(0,r\)}\exp\big(\overline B_\gamma^2\big)\(1+t_\gamma\) \left|\phi\right|^p dx\\
=\frac{1}{\gamma^2}\int_{B\(0,r/\mu_\gamma\)} e^{-2t+\frac{t^2}{\gamma^2}+\bigO\(\frac{1+t}{\gamma^2}\)}\(1+t\) |\tilde \phi|^p dx=\frac{1}{\gamma^2}\int_{\R^2}\tilde\chi_\gamma |\tilde \phi|^p e^{-2t}dx,
\end{multline}
where, as in \eqref{Pr14Eq3} we have
$$\tilde \chi_\gamma:= \mathbf{1}_{B\(0,r/\mu_\gamma\)}\(1+t\)e^{\frac{t^2}{\gamma^2}+\bigO\(\frac{1+t}{\gamma^2}\)} \longrightarrow 1+t \quad \text{in }L^q(\R^2,e^{-2t}dx)$$
for $q<1/\(1-\delta_0\)$, and with H\"older's inequality, we obtain
\begin{align*}
\int_{\R^2}\tilde\chi_\gamma |\tilde \phi|^p e^{-2t}dx\le \(\int_{\R^2}\tilde \chi_\gamma^qe^{-2t}dx\)^\frac{1}{q}\(\int_{\R^2}|\tilde \phi|^{pq'}e^{-2t}dx\)^\frac{1}{q'}&=\bigO\(\(\int_{\R^2}|\nabla \tilde \phi|^2\)^\frac{p}{2}\)\\
&=\bigO\(\(\int_{\R^2}\left|\nabla \phi\right|^2\)^\frac{p}{2}\).
\end{align*}
Substituting into \eqref{Pr14Eq4}, we then obtain \eqref{Pr14Eq2}.
\end{proof}

\section{Proof of Claim~\ref{Claim1}}

From \eqref{Sec31Eq1}, \eqref{Pr10Eq1} and the divergence theorem, we get
\begin{equation}
A_{\varepsilon,\gamma_i,\tau_i}=-2\pi r_\varepsilon\overline{B}'_{\varepsilon,\gamma_i,\overline{\tau_i}}\(r_\varepsilon\)=-2\pi r_\varepsilon\sqrt{\lambda_\varepsilon h_\varepsilon\(\overline{\tau}_i\)}\overline{B}'_{\gamma_i}\big(\sqrt{\lambda_\varepsilon h_\varepsilon\(\overline{\tau}_i\)}r_\varepsilon\big)=\frac{4\pi}{\gamma_i}+\bigO\(\frac{1}{\gammae^3}\).\label{Sec32Eq7}
\end{equation} 
Considering that
$$\ln \lambda_\varepsilon=\bigO\(1\)\quad\text{and}\quad\ln\mu_{\gamma_i}=-\frac{1}{2}\gamma_i^2-\ln\gamma_i+\bigO\(1\),$$
from \eqref{Pr10Eq1}, we infer
\begin{align}\label{Sec32Eq8}
\overline{B}_{\gamma_i}\big(\sqrt{\lambda_\varepsilon h_\varepsilon\(\overline{\tau}_i\)}r_\varepsilon\big)&=\gamma_i-\frac{2\ln(r_\varepsilon/\mu_{\gamma_i})+\ln(\lambda_\varepsilon h_\varepsilon\(\overline{\tau}_i\) )}{\gamma_i}+\bigO\(\frac{1}{\gamma_i}\)\nonumber\\
&=-\frac{2\ln r_\varepsilon}{\gamma_i}-\frac{2\ln\gamma_i}{\gamma_i}+\bigO\(\frac{1}{\gammae}\),
\end{align}
which together with \eqref{Sec31Eq2} and \eqref{Sec32Eq7} gives
$$C_{\varepsilon,\gamma_i,\tau_i} =-\frac{2\ln\gamma_i}{\gamma_i}+\bigO\(\frac{1}{\gammae}\)=-\frac{2\ln\overline{\gamma}_\varepsilon}{\gamma_i}+\bigO\(\frac{1}{\gammae}\).$$
This proves \eqref{Sec32Eq3}. Further, Proposition~\ref{Pr11} gives
\begin{align*}
\partial_{\gamma_i}\[A_{\varepsilon,\gamma_i,\tau_i}\]&=-2\pi r_\varepsilon\sqrt{\lambda_\varepsilon h_\varepsilon\(\overline{\tau}_i\)}\partial_{\gamma_i}\big[\overline{B}'_{\gamma_i}\big(\sqrt{\lambda_\varepsilon h_\varepsilon\(\overline{\tau}_i\)}r_\varepsilon\big)\big]\\
&=-\frac{4\pi}{\gamma_i^2}+\bigO\(\frac{\mu_{\gamma_i}^2}{r_\varepsilon^3}\)+\bigO\(\frac{1}{\gammae^4}\)=-\frac{4\pi}{\gamma_i^2}+\bigO\(\frac{1}{\gammae^4}\).
\end{align*} 
Similarly,
$$\partial_{\gamma_i}[\overline B_{\gamma_i}\big(\sqrt{\lambda_\varepsilon h_\varepsilon\(\overline{\tau}_i\)}r_\varepsilon\big)]=\frac{2\ln r_\varepsilon}{\gamma_i^2}+\frac{2\ln \gamma_i}{\gamma_i^2}+ \bigO\(\frac{1}{\gammae^2}\),$$
so that
\begin{align*}
\partial_{\gamma_i}\[C_{\varepsilon,\gamma_i,\tau_i}\]&=\partial_{\gamma_i}[\overline B_{\gamma_i}\big(\sqrt{\lambda_\varepsilon h_\varepsilon\(\overline{\tau}_i\)}r_\varepsilon\big)]+\frac{1}{2\pi}\partial_{\gamma_i}[A_{\varepsilon,\gamma_i,\tau_i}] \ln r_\varepsilon\\
&=\frac{2\ln \gamma_i}{\gamma_i^2}+ \bigO\(\frac{1}{\gammae^2}\) =\frac{2\ln \overline\gamma_\varepsilon}{\gamma_i^2}+ \bigO\(\frac{1}{\gammae^2}\),
\end{align*}
which proves \eqref{Sec32Eq4}. To prove \eqref{Sec32Eq5}, we observe that
\begin{align}\label{Sec32Eq9}
A_{\varepsilon,\gamma_i,\tau_i}=\lambda_\varepsilon h_\varepsilon\(\overline{\tau_i}\)\int_{B\(\overline{\tau_i},r_\varepsilon\)}f\(\overline{B}_{\varepsilon,\gamma_i,\overline{\tau_i}}\)dx&=2\pi\lambda_\varepsilon h_\varepsilon\(\overline{\tau_i}\)\int_0^{r_\varepsilon}f\big(\overline{B}_{\gamma_i}\big(\sqrt{\lambda_\varepsilon h_\varepsilon\(\overline{\tau_i}\)}r\big)\big)rdr\nonumber\\
&=2\pi\int_0^{\sqrt{\lambda_\varepsilon h_\varepsilon\(\overline{\tau_i}\)}r_\varepsilon}f\big(\overline{B}_{\gamma_i}\(r\)\big)rdr.
\end{align}
By differentiating \eqref{Sec32Eq9} in $\tau_i$, we obtain
\begin{equation}\label{Sec32Eq10}
\partial_{\tau_i}\[A_{\varepsilon,\gamma_i,\tau_i}\]=\pi\lambda_\varepsilon\partial_{x_1}h_\varepsilon\(\overline{\tau_i}\)r_\varepsilon^2f\big(\overline{B}_{\gamma_i}\big(\sqrt{\lambda_\varepsilon h_\varepsilon\(\overline{\tau_i}\)}r_\varepsilon\big)\big).
\end{equation}
By using \eqref{Sec32Eq8} together with the definition of $r_\varepsilon$, and using that $\gamma_i\ge\(1-\delta'\)\gammae$, we obtain
\begin{align}\label{Sec32Eq11}
r_\varepsilon^2f\big(\overline{B}_{\gamma_i}\big(\sqrt{\lambda_\varepsilon h_\varepsilon\(\overline{\tau_i}\)}r_\varepsilon\big)&=\bigO\(r_\varepsilon^2\overline\gamma_\varepsilon\exp\(\frac{4}{\gamma_i^2}\(\ln r_\varepsilon+\ln \gamma_i+\bigO\(1\)\)^2\)\)\nonumber\allowdisplaybreaks\\
&=\bigO\(\overline\gamma_\varepsilon\exp\(2\ln r_\varepsilon\(1+2\frac{\ln r_\varepsilon}{\gamma_i^2}+4\frac{\ln \gamma_i}{\gamma_i^2}\)\)\)\nonumber\allowdisplaybreaks\\
&=\bigO\(\exp\(-\delta_0\overline\gamma_\varepsilon^2\(1-\delta_0\frac{\overline\gamma_\varepsilon^2}{\gamma_i^2}+\smallo\(1\)\)+\ln\gammae\)\)\nonumber\\
&=\bigO\(\exp\(-\delta_0\overline\gamma_\varepsilon^2\(1-\frac{\delta_0}{\(1-\delta'\)^2}+\smallo\(1\)\)\)\)=\smallo\(\frac{1}{\gammae^a}\)
\end{align}
uniformly in $(\gamma,\tau)\in \Gamma_\varepsilon^k\(\delta'\)\times T_\varepsilon^k\(\delta\)$ for all $a\ge0$, provided $\delta'<1-\sqrt{\delta_0}$. By using \eqref{Sec32Eq10} and \eqref{Sec32Eq11} and since $\lambda_\varepsilon\to\lambda_0$ and $h_\varepsilon\to h_0$ in $C^1\(\overline\Omega\)$, we obtain the first part of \eqref{Sec32Eq5}. By differentiating $C_{\varepsilon,\gamma_i,\tau_i}$ in $\tau_i$ and using \eqref{Pr10Eq1}, \eqref{Sec32Eq10} and \eqref{Sec32Eq11}, we then obtain
\begin{multline*}
\partial_{\tau_i}\[C_{\varepsilon,\gamma_i,\tau_i}\]=\frac{\sqrt{\lambda_\varepsilon}\partial_{x_1}h_\varepsilon\(\overline{\tau}_i\)r_\varepsilon}{2\sqrt{h_\varepsilon\(\overline{\tau}_i\)}}\overline{B}'_{\gamma_i}\big(\sqrt{\lambda_\varepsilon h_\varepsilon\(\overline{\tau}_i\)}r_\varepsilon\big)\\
-\frac{1}{2\pi}\partial_{\tau_i}\[A_{\varepsilon,\gamma_i,\tau_i}\]\ln\frac{1}{r_\varepsilon}=-\frac{\partial_{x_1}h_\varepsilon\(\overline{\tau_i}\)}{h_\varepsilon\(\overline{\tau_i}\)\gamma_i}+\bigO\(\frac{\left|\partial_{x_1}h_\varepsilon\(\overline{\tau_i}\)\right|}{\gammae^3}\),
\end{multline*}
which gives the second part of \eqref{Sec32Eq5}. This ends the proof of Claim~\ref{Claim1}.
\endproof

\end{document}